%% file: main.tex
\documentclass[a4paper]{amsart}

\usepackage[dvipsnames]{xcolor}
\colorlet{vectorspace}{RoyalBlue}
\colorlet{abeliangroup}{ForestGreen}
\colorlet{group}{BurntOrange}
\colorlet{pointedset}{Red}

\usepackage{hyperref}
\hypersetup{
  colorlinks,
  linkcolor={red!80!black},
  citecolor={blue!80!black},
  urlcolor={blue!80!black}
}
\usepackage{environments}
\usepackage{macros}
\usepackage{mathtools}
\usepackage{amssymb}
\usepackage{mathscinet}
\usepackage{enumitem}
\usepackage{dynkin-diagrams}
\usepackage[only,mapsfrom]{stmaryrd}
\usepackage{chngcntr}
\usepackage{binarytree}
\usepackage[colorinlistoftodos]{todonotes}
\usepackage{tabularray}

\usepackage[
  backend=biber,
  style=alphabetic,
  natbib=true,
  isbn=false, % print no ISBN
  doi=true, % print no DOI
  url=false, % print no URL
  eprint=false, % print no preprint reference
  giveninits=true, % print initials, not full names
  backref=true, % print back references in the bibliography (hyperref must be loaded)
  maxbibnames=99, % print all authors in the bibliography
]{biblatex}
\AtEveryBibitem{%
\ifentrytype{unpublished}{}{\clearfield{note}}%
}
\renewbibmacro{in:}{} % remove "in" before journal title
\usepackage{doi}

\DeclareSymbolFont{sfoperators}{OT1}{cmss}{m}{n}%
\DeclareSymbolFontAlphabet{\mathsf}{sfoperators}%
\makeatletter%tr
\def\operator@font{\mathgroup\symsfoperators}%
\makeatother

\setlist[description]{font=\normalfont\space}

\title{Obstruction theory for $A_\infty$-bimodules}

\subjclass[2020]{Primary: 18M65; secondary: 18N40, 16E45, 55S35, 18G40, 18G80}
\keywords{Differential graded algebras; differential graded bimodules; $\A$-algebras; $\A$-bimodules; Hochschild cohomology; universal Massey product; operads}

\author[G.~Jasso]{Gustavo Jasso}%
\address[G.~Jasso]{%
  Mathematisches Institut, %
  Universität zu Köln, %
  Weyertal 86-90, %
  50931 Köln, %
  Germany}%
\email{gjasso@math.uni-koeln.de}%
\urladdr{https://gustavo.jasso.info}%

\author[F.~Muro]{Fernando Muro}%
\address[F.~Muro]{%
  Universidad de Sevilla, %2
  Facultad de Matemáticas, %
  Departamento de Álgebra, %
  Calle Tarfia s/n, %
  41012 Sevilla, %
  Spain%
} \email{fmuro@us.es} \urladdr{https://personal.us.es/fmuro/}

\DeclareRobustCommand{\SkipTocEntry}[5]{}

\begin{document}

\begin{abstract}
  \input{abstract.tex}
\end{abstract}

\maketitle

\setcounter{tocdepth}{1}
\tableofcontents

\input{content.tex}

% \input{sections/introduction.tex}

% \input{sections/structure.tex}

% \input{sections/acknowledgements.tex}

% \input{sections/support.tex}

% \input{sections/homotopy-theory.tex}

% \input{sections/cohomology.tex}

% \input{sections/obstruction_theory.tex}

% \input{sections/kadeishvili_bimodules.tex}

% \input{sections/sparse.tex}

\printbibliography
\end{document}

%% file: abstract.tex
We develop an obstruction theory for the extension of truncated minimal $\A$-bimodule
structures over truncated minimal $\A$-algebras. Obstructions live in far-away pages of
a (truncated) fringed spectral sequence of Bousfield--Kan type. The second page
of this spectral sequence is mostly given by a new cohomology theory
associated to a pair consisting of a graded algebra and a graded bimodule over
it. This new cohomology theory fits in a long exact sequence involving the
Hochschild cohomology of the algebra and the self-extensions of the bimodule. We
show that the second differential of this spectral sequence is given by the Gerstenhaber bracket with a
bimodule analogue of the universal Massey product of a minimal $\A$-algebra. We
also develop a closely-related obstruction theory for truncated minimal
$\A$-bimodule structures
over (the truncation of) a fixed minimal $\A$-algebra; the second page of the corresponding spectral
sequence is now mostly given by the vector spaces of self-extensions of the
underlying graded bimodule
and the second differential is described analogously to the previous one. We
also establish variants of the above for graded algebras and graded bimodules
that are $d$-sparse, that is they are concentrated in degrees that are multiples
of a fixed integer $d\geq1$. These obstruction theories are used to establish
intrinsic formality and almost formality theorems for differential graded
bimodules over differential graded
algebras. Our results hold, more generally, in the context of graded
operads with multiplication equipped with an associative operadic ideal, examples of which are the endomorphism operad of a
graded algebra and the linear endomorphism operad of a pair consisting of a
graded algebra and a graded bimodule over it.

%%% Local Variables:
%%% mode: LaTeX
%%% TeX-master: "main"
%%% End:

%% file: content.tex
\section{Introduction}

Let $\kk$ be a field that, for simplicity, we assume to be of characteristic $0$
throughout the introduction.\footnote{All of our results hold over an arbitrary
  field.} In this article we develop an obstruction theory for the extension of
truncated minimal $\A$-bimodule structures over truncated minimal $\A$-algebras, and leverage this obstruction theory in
order to establish \emph{intrinsic formality} and \emph{almost formality}
theorems for differential graded (=DG) bimodules over DG algebras.

\subsection{Intrinsic formality theorems for algebras and bimodules}

Our starting point is the following classical theorem of Kadeishvili.

\begin{theorem}[{\cite[Corollary~4]{Kad88}}]\label{Kadeishvili_algebras} Let $A$
  be a graded algebra. Suppose that the Hochschild cohomology of $A$ vanishes in
  the following range:
	\[
    \HH[n+2][-n]{A}=0,\qquad n\geq1.
  \]
	Then, every DG algebra $B$ such that $\dgH{B}\cong A$ as graded algebras is
  formal, that is $B$ is quasi-isomorphic to $A$ where the latter is viewed as a
  DG algebra with vanishing differential.
\end{theorem}

Recall that a graded algebra $A$ is \emph{intrinsically formal} if every DG
algebra $B$ with $\dgH{B}\cong A$ is formal. Kadeishvili's Theorem hence
provides a sufficient condition---expressed in terms of vanishing of Hochschild
cohomology---for a graded algebra to be intrinsically formal as a DG algebra.

The following is the first main result in this article; it provides a necessary
condition for a graded bimodule over a graded algebra to be
intrinsically formal as a DG bimodule, and hence it can be regarded as bimodule
analogue of \Cref{Kadeishvili_algebras}.

\begin{theorem}[{\Cref{Kadeishvili_bimodules-main-text}}]\label{Kadeishvili_bimodules}
	Let $A$ be a graded algebra and $M$ a graded $A$-bimodule. Suppose that the
  vector spaces of self-extensions of $M$ vanish in the following
  range:\footnote{Notice the difference in the range with respect to
    \Cref{Kadeishvili_algebras}.}
  \[
    \Ext[n+1][-n]{A^e}{M}{M}=0,\qquad n\geq1.
  \]
	Then, every DG $A$-bimodule $N$ such that $\dgH{N}\cong M$ as graded
  $A$-bimodules is formal, that is $N$ quasi-isomorphic to $M$ where the latter
  is viewed as a DG $A$-bimodule with vanishing differential.
\end{theorem}

We now address the question of when a graded algebra $A$ and a graded
$A$-bimodule $M$ are \emph{simultaneously} intrinsically formal. For this, we
let $\BimHC{M}$ be the bimodule complex of $M$ (whose cohomology is
isomorphic to the bigraded vector space of self-extensions of $M$)
and notice that $\BimHC{M}$ is a DG bimodule over the Hochschild complex
$\HC{A}$ of $A$ (\Cref{sec:hochschild}). We introduce the cochain complex
\[
  \RelBimHC[n]<r>{A}{M}\coloneqq \HC[n][r]{A}\oplus\BimHC[n-1]<r>{M},\qquad
  n\geq 0,\quad r\in\ZZ,
\]
called the \emph{bimodule Hochschild (cochain) complex} of the pair $(A,M)$, and
which is defined as the homotopy fibre (=mapping cocone) of the cochain map
\[
  \delta\colon\HC{A}\longrightarrow\BimHC{M},\qquad c\longmapsto
  \id*[M]\cdot c-c\cdot\id*[M],
\]
that, roughly, measures how far is $\BimHH{M}$ from being symmetric as a graded
$\HH{A}$-bimodule (see \Cref{connecting_morphism} and
\Cref{cor:symmetric_bimodule}). The cohomology
\[
  \RelBimHH{A}{M}\coloneqq\H[\bullet,*]{\RelBimHC{A}{M}}
\]
is the \emph{bimodule Hochschild cohomology} of the pair $(A,M)$.
By construction, there are long exact sequences of vector spaces $(r\in\ZZ)$
\[
  \begin{tikzcd}[column sep=small,row sep=small]
    \cdots\rar{\delta}&\BimHH[n-1]<r>{M}\rar&\RelBimHH[n]<r>{A}{M}\rar&\HH[n][r]{A}\ar[out=-10,in=170,overlay]{dll}[description]{\delta}\\
    &\BimHH[n]<r>{M}\rar&\RelBimHH[n+1]<r>{A}{M}\rar&\HH[n+1][r]{A}\rar{\delta}&\cdots
  \end{tikzcd}
\]
that relate the new cohomology $\RelBimHH[n]<r>{A}{M}$ to the cohomologies in
\Cref{Kadeishvili_algebras,Kadeishvili_bimodules}. With this new cohomology
theory at hand we may state the following \emph{simultaneous} intrinsic
formality theorem. For the precise definition of quasi-isomorphism of pairs we
refer the reader to \Cref{def:quasi-isomorphic_algebra-bimodule_pairs} (see also
\Cref{minimal_models_quasi_isomorphic_algebras-bimodules}).

\begin{theorem}[{\Cref{Kadeishvili_simultaneous}}]\label{thm:Kadeishvili_algebras_bimodules}
  Let $A$ be a graded algebra, $M$ a graded $A$-bimodule, and suppose that the
  bimodule Hochschild cohomology of the pair $(A,M)$ vanishes in the following
  range:
  \[
    \RelBimHH[n+2]<-n>{A}{M}=0,\qquad n\geq 1.
  \]
  Then, every pair $(B,N)$ consisting of a DG algebra $B$ such that $\H{B}\cong
  A$ as graded algebras and a DG $B$-bimodule $N$ such that $\H{N}\cong M$ as
  graded $A$-bimodules is quasi-isomorphic to the pair $(A,M)$.
\end{theorem}

\subsection{Almost formality theorems for algebras and bimodules}

In order to extend
\Cref{Kadeishvili_bimodules,thm:Kadeishvili_algebras_bimodules} to DG bimodules
over DG algebras (neither of which need be formal) we need a further cohomology
theory that can be thought of as a `refinement' of the bimodule Hochschild
cohomology of a pair. For this, we first recall the well-known relationship
between DG algebras and $\A$-algebras~\cite{Kad82}.

Recall that an \emph{$\A$-algebra} is a pair $(A,\Astr)$ consisting of a DG
vector space $A=(A,\Astr<1>)$ and a sequence $\Astr$ of homogeneous
\emph{operations}
\[
  \Astr<n>\colon A^{\otimes n}\longrightarrow A,\qquad n\geq 2,
\]
of degree $2-n$ satisfying a certain countable system of quadratic equations,
called the \emph{$\A$-equations}, that imply, in particular, that the binary
operation $\Astr<2>$ induces a graded algebra structure on $\H{A}$. For example,
DG algebras identify with $\A$-algebras with $\Astr<n>[]=0$, $n\geq3$. Every
$\A$-algebra is $\A$-quasi-isomorphic to both a DG algebra, its
\emph{strictification}, and to a \emph{minimal} $\A$-algebra, that is to an
$\A$-algebra with $\Astr<1>[]=0$, called its \emph{minimal model}. As a
consequence, minimal $\A$-algebras can be leveraged in the study of DG algebras.
Indeed, if $(A,\Astr)$ is a minimal $\A$-algebra, then the first possibly
non-vanishing higher operation $\Astr<3>$ is a cocycle in the Hochschild cochain
complex of the underlying graded algebra $A=\H{A,\Astr<1>}$; its cohomology
class\footnote{Throughout this article we only consider Hochschild cohomology of
  \emph{graded} algebras, even if such a graded algebra is equipped with a
  minimal $\A$-algebra structure. In particular all the Hochschild cohomologies
  that we consider are \emph{bi}graded.}
\[
  \Hclass{\Astr<3>}\in\HH[3][-1]{A}
\]
is an invariant called the \emph{universal Massey product} of the minimal
$\A$-algebra $(A,\Astr)$.\footnote{The universal Massey product of a a minimal $\A$-algebra was first
  investigated in \cite{BKS04}, although it goes back to the more general notion of universal Toda bracket in \cite{BD89}. It also plays a central role in \cite{Mur20}; we
  refer the reader to the latter article for further information on this
  invariant and references to related literature.} Recall that the Hochschild
cohomology $\HH{A}$ is a Gerstenhaber algebra and hence, in particular, a shifted
graded Lie algebra. The (minimal) $\A$-equations imply that
\[
  \tfrac{1}{2}[\Hclass{\Astr<3>},\Hclass{\Astr<3>}]=0,
\]
and the Gerstenhaber relations in Hochschild cohomology then show that the
bigraded vector space
\[
  \HMC[n]<r>{A}[\Hclass{\Astr<3>}]\coloneqq\HH[n][r]{A},\qquad n\geq 2,\qquad
  r\in\ZZ,
\]
can be endowed with the bidegree $(2,-1)$ differential
\[
  d\colon x\longmapsto [\Hclass{\Astr<3>},x]
\]
given by the Gerstenhaber bracket with the universal Massey product. We call the resulting cochain complex
$\HMC{A}[\Hclass{\Astr<3>}]$ the \emph{Hochschild--Massey (cochain) complex} of the
pair $(A,\Hclass{\Astr<3>})$, and its cohomology
\[
  \HMH{A}[\Hclass{\Astr<3>}]\coloneqq\H[\bullet,*]{\HMC{A}[\Hclass{\Astr<3>}]}
\]
the \emph{Hochschild--Massey cohomology} of the pair $(A,\Hclass{\Astr<3>})$. If
$A$ is instead a DG algebra, we define the universal Massey product of $A$ to be
the universal Massey product of any of its minimal models. The following
theorem, which is a generalisation of \Cref{Kadeishvili_algebras} for DG
algebras, is the case $d=1$ of \cite[Theorem~B]{JKM22}, see also~\cite{Mur22} where
this result is implicit in a special case.

\begin{theorem}[{\Cref{thm:B-main_text}}]
  \label{thm:B}
  Let $A$ be a DG algebra. Choose a minimal model $(\H{A},\Astr)$ of $A$ and
  suppose that the Hochschild--Massey cohomology of the pair
  $(\H{A},\Hclass{\Astr<3>[A]})$ vanishes in the following
  range:\footnote{Notice the strict inequality.}
  \[
    \HMH[n+2]<-n>{\H{A}}[\Hclass{\Astr<3>}],\qquad n>1.
  \]
  Then, every DG algebra $B$ such that $\H{B}\cong\H{A}$ as graded algebras and
  whose universal Massey product satisfies
  \[
    \Hclass{\Astr<3>[B]}=\Hclass{\Astr<3>[A]}\in\HH[3][-1]{\H{A}}
  \]
  is quasi-isomorphic to $A$.
\end{theorem}

We regard \Cref{thm:B} as an \emph{almost formality} theorem in the sense that
it gives a sufficient condition for the quasi-isomorphism type of a DG algebra to
be determined by its cohomology graded algebra and a finite collection of Hochschild-theoretic
data (in this case the universal Massey product of the DG algebra). Other theorems of
similar kind include~\cite[Theorem~5]{LP11} and~\cite[Theorem~B]{HK24}.

In order to explain our variant of \Cref{thm:B} for DG bimodules over DG
algebras we need further preparation. Recall that, given an $\A$-algebra
$(A,\Astr)$, an $\A$-bimodule over it is a pair $(M,\Astr[M])$
consisting of a DG vector space $M=(M,\Astr<1>[M])$ and homogeneous operations
\[
  \Astr<p,q>[M]\colon A^{\otimes p}\otimes M\otimes A^{\otimes q}\longrightarrow
  M,\ \qquad n\geq 2,\quad p+1+q=n,
\]
of degree $2-n$ such that the cochain
\[
  \Astr[A\ltimes M]\coloneqq\Astr[A]+\Astr[M]
\]
endows
the DG vector space $A\oplus M$ with the structure of an
$\A$-algebra,\footnote{Notice, however, that not all $\A$-algebra structures on
  $A\oplus M$ are of this form.} where
\[
  \Astr<n>[M]\coloneqq(\Astr<p,q>[M])_{p+1+q=n},\qquad n\geq 2.
\]
An $\A$-bimodule $M$ is \emph{minimal} if the differential
$\Astr<1>[M]=\Astr<0,0>[M]$ vanishes. Suppose now that $(A,\Astr)$ is a minimal
$\A$-algebra and $(M,\Astr[M])$ is a minimal $\A$-bimodule. We consider
$\Astr[A\ltimes M]$ as a cochain in the bimodule Hochschild complex
$\RelBimHC{A}{M}$ of the pair $(A,M)$. Similar to the Hochschild cohomology of a
graded algebra, the bimodule Hochschild cohomology $\RelBimHH{A}{M}$ is a
Gerstenhaber algebra. Also, the cochain $\Astr<3>[A\ltimes M]$ is a cocycle whose
class
\[
  \Hclass{\Astr<3>[A\ltimes M]}\in\RelBimHH[3]<-1>{A}{M}
\]
yields a novel invariant that we call the \emph{bimodule universal Massey
  product} of the minimal $\A$-bimodule $(M,\Astr[A\ltimes M])$. The \emph{bimodule Hochschild--Massey (cochain) complex} of the triple
$(A,M,\Hclass{\Astr<3>[A\ltimes M]})$ is the bigraded vector space
\[
  \AlgBimHMC{M}[3]\coloneqq\RelBimHH[n]<r>{A}{M}\qquad n\geq 2,\quad r\in\ZZ,
\]
endowed with the bidegree $(2,-1)$ differential
\[
  x\longmapsto[\Hclass{\Astr<3>[A\ltimes M]},x]
\]
given by the Gerstenhaber bracket with the bimodule universal Massey product of
the minimal $\A$-bimodule $(M,\Astr[A\ltimes M])$, and its cohomology
\[
  \AlgBimHMH{M}[3]\coloneqq\H[\bullet,*]{\AlgBimHMC{M}[3]}
\]
is the \emph{bimodule Hochschild--Massey cohomology} of the triple
$(A,M,\Hclass{\Astr<3>[A\ltimes M]})$. If $A$ is instead a DG algebra and $M$ is a
DG $A$-bimodule, we define the bimodule universal Massey product of $M$ by first
choosing a minimal model $(\H{A},\Astr)$ of $A$ and then considering the
bimodule universal Massey product of any minimal model of $M$ over
$(\H{A},\Astr[A])$. We then have the following simultaneous almost formality
theorem for DG algebras and DG bimodules that is a generalisation of
\Cref{thm:Kadeishvili_algebras_bimodules}.

\begin{theorem}[{\Cref{secondary_Kadeishvili_simultaneous_DG_uniqueness-main_text}}]
  \label{secondary_Kadeishvili_simultaneous_DG_uniqueness}
  Let $A$ be a DG algebra and $M$ a DG $A$-bimodule. Choose a minimal model
  $(\H{A},\Astr[A])$ of $A$, a compatible minimal model $(\H{M},\Astr[M])$ of
  $M$ over $(\H{A},\Astr)$.
  Let
  \[
    \Astr<3>[A\ltimes M]\coloneqq\Astr<3>[A]+\Astr<3>[M]\in\RelBimHH[3]<-1>{\H{A}}{\H{M}}
  \]
  and suppose that the bimodule
  Hochschild--Massey cohomology of the triple
  \[
    (\H{A},\H{M},\Hclass{\Astr<3>[A\ltimes M]})
  \]
  vanishes in the following range:
  \[
    \AlgBimHMH!\H{A}![n+2]<-n>{\H{M}}[3][A\ltimes M]=0,\qquad n>1.
  \]
  Then, every pair $(B,N)$ consisting of
  \begin{itemize}
  \item a DG algebra $B$ such that $\H{B}\cong\H{A}$ as graded algebras and
  \item a DG $B$-bimodule $N$ such that $\H{N}\cong\H{M}$ as graded
    $\H{A}$-bimodules and
  \item whose bimodule universal
    Massey product satisfies
    \[
      \Hclass{\Astr<3>[B\ltimes N]}=\Hclass{\Astr<3>[A\ltimes
        M]}\in\RelBimHH[3]<-1>{\H{A}}{\H{M}}
    \]
  \end{itemize}
  is quasi-isomorphic to the pair $(A,M)$.
\end{theorem}

Given a minimal $\A$-algebra $(A,\Astr[A])$ and a minimal $\A$-bimodule
$(M,\Astr[M])$, we also introduce the \emph{Massey bimodule
  (cochain) complex} of the pair $(M,\Hclass{\Astr<3>[A\ltimes M]})$, defined as
the bigraded vector space
\[
  \BimHMC[n]<r>{M}[3]\coloneqq\BimHH[n]<r>{M},\qquad n\geq 1,\quad r\in\ZZ,
\]
and which is endowed with the bidegree $(2,-1)$ differential
\[
  x\longmapsto[\Hclass{\Astr<3>[A\ltimes M]},x]
\]
given by the Gerstenhaber bracket with the bimodule universal Massey product. The cohomology
\[
  \BimHMH{M}[3]\coloneqq\H[\bullet,*]{\BimHMC{M}[3]}
\]
is the \emph{Massey bimodule cohomology} of the pair
$(M,\Hclass{\Astr<3>[A\ltimes M]})$. The following almost formality theorem for
DG bimodules over a \emph{fixed} DG algebra is a generalisation of
\Cref{Kadeishvili_bimodules}.

\begin{theorem}\label{thm:B_bimodules}
  Let $A$ be a DG algebra and $M$ a DG $A$-bimodule. Choose a minimal model
  $(\H{A},\Astr)$ of $A$, a compatible minimal model $(\H{M},\Astr[M])$
  of $M$ over $(\H{A},\Astr)$, and suppose that the Massey bimodule
  cohomology of the pair $(\H{M},\Hclass{\Astr<3>[A\ltimes M]})$ vanishes in the
  following range:
  \[
    \BimHMH!\H{A}![n+1]<-n>{\H{M}}[3][A\ltimes M]=0,\qquad n>1.
  \]
  Then, every DG $A$-bimodule $N$ such that $\H{N}\cong\H{M}$ as graded
  $\H{A}$-bimodules and such that
  \[
    \Hclass{\Astr<3>[M]-\Astr<3>[N]}=0\in\BimHH!\H{A}^e![2]<-1>{\H{M}}
  \]
  is quasi-isomorphic to $M$.
\end{theorem}

Our results have analogues for DG algebras and DG bimodules whose cohomologies
are $d$-sparse, that is they are concentrated in degrees that are multiples of a
fixed integer $d\geq1$, see \Cref{sec:sparse} for detailed statements. We also
mention that \Cref{thm:B_bimodules} and its $d$-sparse variants are key
technical ingredients in the proofs of the main theorems in our article
\cite{JM25}, a sequel to \cite{JKM22}, where we study certain bimodule right
Calabi--Yau algebras in the sense of Kontsevich~\cite{Kon93} (see also
\cite{KS09}). We remind the reader that a DG algebra $A$ is \emph{bimodule right
  $n$-Calabi--Yau}, $n\in\ZZ$, if the diagonal bimodule of $A$ and its shifted
$\kk$-linear dual $(DA)[-n]$ are quasi-isomorphic as DG $A$-bimodules.
Cohomological criteria for detecting the diagonal bimodule of a DG algebra are
therefore of interest in this context.

\subsection{Obstruction theory for algebras and bimodules}

In order to prove our main results we develop an obstruction theory for
extending minimal $\A[k]$-bimodule structures over minimal $\A[k]$-algebras, that is truncated
minimal $\A$-structures with operations only up to a given arity $k\geq2$. For
this, we rely heavily on previous work by the second-named author~\cite{Mur20}.

Let $A$ be a graded vector space. Consider the problem of extending a minimal
$\A[k]$-algebra structure
\[
  (A,\Astr<2>,\Astr<3>,\dots,\Astr<k-1>,\Astr<k>)
\]
on a graded vector space $A$ to a minimal $\A[k+1]$-algebra structure
\[
  (A,\Astr<2>,\Astr<3>,\dots,\Astr<k-1>,\widetilde{\Astr<k>},\Astr<k+1>[A])
\]
with the same underlying $\A[k-1]$-algebra structure but possibly with a different
operation of arity $k$. For $k\geq4$ there are obstructions in the Hochschild
cohomology
\[
  \HH[k+1][2-k]{A},
\]
see for example~\cite[Corollaire~B.1.2]{Lef03}. Let $k\geq5$. If we allow more
modifications to the given minimal $\A[k]$-algebra structure, there are
obstructions in far-away pages of a truncated spectral sequence constructed by
the second-named author in~\cite{Mur20}. The second second page  of this
spectral sequence is essentially given by the Hochschild cohomology $\HH{A}$ of
the underlying graded algebra $A$, and its second differential
\[
  \ssd{2}\colon x\longmapsto[\Hclass{\Astr<3>},x]
\]
is given by the Gerstenhaber bracket with the universal Massey product of the
original minimal $\A[k]$-algebra (compare with the definition of the
Hochschild--Massey complex of a minimal $\A$-algebra). Applications of
this obstruction theory and of its $d$-sparse variants in representation theory
and algebraic geometry can be found in~\cite{Mur22,JKM22,JKM24}.

In this article we introduce variants of the above obstruction theory for:
\begin{itemize}
\item \emph{simultaneously} extending a minimal $\A[k]$-algebra structure
  \[
    (A,\Astr<2>,\Astr<3>,\dots,\Astr<k>)
  \]
  on a graded vector space $A$ and a compatible minimal $\A[k]$-bimodule
  structure
  \[
    (M,\Astr<2>[A\ltimes M],\Astr<3>[A\ltimes M],\dots,\Astr<k>[A\ltimes M])
  \]
  on a graded vector space $M$;
\item extending a minimal $\A[k]$-bimodule structure
  \[
    (M,\Astr<2>[A\ltimes M],\Astr<3>[A\ltimes M],\dots,\Astr<k>[A\ltimes M])
  \]
  on a graded vector space $M$ over (the truncation of) a \emph{fixed} minimal $\A$-algebra
  structure $(A,\Astr)$ on a graded vector space $A$.
\end{itemize}
As with the obstruction theory for extending truncated minimal $\A$-algebras recalled
above, in both cases we allow for several modifications of the given minimal
$\A[k]$-structures in lower arities. In the first case the obstructions are
located in far-away pages of a truncated spectral sequence whose second page
 is essentially given by the bimodule Hochschild cohomology
$\RelBimHH{A}{M}$ of the pair $(A,M)$, and whose second differential
\[
  \ssd{2}\colon x\longmapsto[\Hclass{\Astr<3>[A\ltimes M]},x]
\]
is given by the Gerstenhaber bracket with the corresponding bimodule universal
Massey product (compare with the definition of the bimodule Hochschild--Massey
complex of a triple). Similarly, in the second case the obstructions are
located in far-away pages of a truncated spectral sequence whose second page
 is essentially given by the bigraded vector space of self-extensions
$\BimHH{M}$, and whose second differential
\[
  \ssd{2}\colon x\longmapsto[\Hclass{\Astr<3>[A\ltimes M]},x]
\]
is again given by the Gerstenhaber bracket with the corresponding bimodule
universal Massey product (compare with the definition of the Massey bimodule complex).

Our construction of the aforementioned truncated spectral sequences is quite
general and the setup is better described in the language of (non-symmetric)
DG operads. Recall that a DG operad $\O$ consists of a sequence $\O<n>$,
$n\geq0$, of DG vector spaces of \emph{operations} equipped with
\emph{infinitesimal composition} operations
\[
  \circ_i\colon\O<p>\otimes\O<q>\longrightarrow\O<p+q-1>,\qquad 1\leq i\leq
  p,\quad p\geq 0,
\]
and a \emph{unit} cocycle $\id\in\O<1>^0$ satisfying suitable axioms. The prototypical example is the \emph{endomorphism
  DG operad} $\E{V}$ of a DG vector space $V$, whose homogeneous operations of
arity $n$ are the homogeneous maps
\[
  V^{\otimes n}\longrightarrow V,
\]
and whose infinitesimal compositions are given by the apparent compositions of
multilinear maps. Given a pair of DG vector spaces $(V,W)$, we also consider the
\emph{linear endomorphism DG operad} $\E{V,W}$ introduced in~\cite{BM09}; this
is a suboperad $\E{V,W}\subseteq\E{V\oplus W}$ of the endomorphism DG operad of
the direct sum $V\oplus W$ that contains the endomorphism DG operad $\E{V}$ as a
suboperad (with different unit, though) and also an operadic ideal $\E*{V,W}\subseteq\E{V,W }$ consisting of
all homogeneous operations of the form
\[
  V^{\otimes p}\otimes W\otimes V^{\otimes q}\longrightarrow W,\qquad n\geq
  2,\qquad p+1+q=n.
\]
In particular, there is a canonical quotient map of operads
$\E{V,W}\twoheadrightarrow\E{V}\cong\E{V,W}/\E*{V,W}$ preserving the unit.

An $\A$-algebra structure on a DG vector space $A$ is equivalent to the datum of
a morphism of DG operads $\Astr\colon\A\to\E{A}$ from the DG operad $\A$;
similarly, a compatible $\A$-bimodule structure on a DG vector space $M$ is equivalent to
the datum of a morphism of DG operads $\Astr[A\ltimes M]\colon\A\to\E{A,M}$ such
that the following diagram commutes, where the vertical arrow denotes the
canonical projection:
\[
  \begin{tikzcd}
    &\E{A,M}\dar[two heads]\\
    \A\rar[swap]{\Astr}\urar{\Astr[A\ltimes M]}&\E{A}.
  \end{tikzcd}
\]
The DG operad $\A$ admits an exhaustive filtration
\[
  \A[2]\subset\A[3]\subset\cdots\subset\A[k]\subset\cdots\subset\A
\]
by DG suboperads, where $\A[k]$ is the DG operad for truncated $\A$-structures
with operations only up to a given arity $k\geq2$, and the previous discussion
extends to $\A[k]$-structures in the obvious way.

The upshot of this perspective is the following. Let $\infty\geq k\geq 2$. A
morphism of DG operads $\Astr[\O]\colon \A[k]\to\O$ can be identified with a
point in the Dwyer--Kan mapping space (a Kan complex)
\[
  \Map{\A[k]}{\O}=\Map[\Operads]{\A[k]}{\O}
\]
computed in the category $\Operads$ of DG operads endowed with the transferred
projective model structure~\cite{Lyu11,Mur11}. Thus, for $\O=\E{A}$ or
$\O=\E{A,M}$, the extension problems discussed above correspond to exhibiting a
point in the mapping space
\[
  \Map{\A[k+1]}{\O}
\]
that agrees with the original point $\Astr[\O]\colon\A[k]\to\O$ after restriction along the
induced maps
\[
  \Map{\A[k+1]}{\O}\twoheadrightarrow\Map{\A[k]}{\O}\twoheadrightarrow\Map{\A[\ell]}{\O}
\]
for some chosen $k\geq \ell\geq 2$. As mentioned above, the spectral sequences that we
construct are truncated, but if $n=\infty$ then they are defined in their
entirety and the term $\BK[\infty]$ contributes to the homotopy groups of the
mapping space
\[
  \Map{\A}{\O}\simeq\operatorname{holim}_{k\geq
    2}\Map{\A[k]}{\O},
\]
based at the given point $\Astr[\O]$~\cite[Section~IX.5]{BK72}. In fact, our (truncated) spectral
sequences are extended (!) fringed spectral sequences of Bousfield--Kan
type~\cite[Section~IX.4]{BK72}; we refer the reader to the introduction
to~\cite{Mur20} for a more through discussion of the computational advantages of
this kind of extended spectral sequences (at least in the case $\O=\E{A}$).
Finally, obstructions for extending a minimal $\A[k]$-bimodule structure on $M$
over a fixed minimal $\A$-structure $\Astr$ on $A$ arise from considering the spectral
sequence associated to the the homotopy fibre of the map (a Kan
fibration)
\[
  \Map{\A}{\E{A,M}}\twoheadrightarrow\Map{\A}{\E{A}}
\]
induced by the canonical quotient map $\E{A,M}\twoheadrightarrow\E{A}$, taken
above the point $\Astr$ of the target.

In the above discussion we may work, more generally, with an arbitrary graded operad $\O$
with multiplication (\Cref{operad_multiplication}) and an associative operadic
ideal ${\I\subset\O}$ (\Cref{def:associative_operadic_ideal}). We then construct
compatible spectral sequences associated to the induced morphism of mapping
spaces
\[
  \Map{\A}{\O}\twoheadrightarrow\Map{\A}{\O/\I},
\]
and to its homotopy fibre $\Str{\A}{h}{\O}$ above a given point $h\colon\A\to\O/\I$ of the target. The three resulting
spectral sequences fit into a sequence of composable morphisms that map
obstructions onto obstructions, similar to the long exact sequence involving
their second pages. The simultaneous obstruction theory fits into the general
setting of \cite[Section~6]{Mur20}, and the obstruction theory for
$\A$-bimodules over a fixed $\A$-algebra fits into the more general setting of
\cite[Section~5]{Mur20}. The key insight is that the good properties of the
tower of mapping spaces
\[
  \cdots\twoheadrightarrow\Map{\A[k]}{\O}\twoheadrightarrow\cdots\twoheadrightarrow\Map{\A[3]}{\O}\twoheadrightarrow\Map{\A[2]}{\O}
\]
only depend on the good properties of the DG operad $\A$ and of its filtration
by the suboperads $\A[k]$, $k\geq2$, as well as a minimal amount of additional
structure on the graded operad $\O$ (the datum of a multiplication).

%%%%%

\addtocontents{toc}{\SkipTocEntry}\subsection*{Structure of the article}

In \Cref{subsec:homotopy_theory} we discuss the homotopy theory of
(non-symmetric) DG operads, building upon previous work of the
second-named author~\cite{Mur14}. In \Cref{sec:operads_cohomology} we recall
the construction of the Gerstenhaber algebra
associated to a (graded) operad with multiplication, and describe the additional
algebraic structures that can be extracted under the additional presence of an operadic
ideal. In the special case of the linear endomorphism operad of a pair
consisting of a graded algebra and a graded bimodule over it we obtain the novel
bimodule Hochschild cochain complex of the pair. In \Cref{sec:obstructions} we
establish, by leveraging results from the second-named author's~\cite{Mur20} in
a crucial way, the various obstruction theories discussed in the introduction. In
\Cref{sec:Kadeishvili-type_theorems} we prove our main theorems as stated in the
introduction and, in \Cref{sec:sparse}, we prove their `sparse' variants.

\addtocontents{toc}{\SkipTocEntry}\subsection*{Conventions}

We work over an arbitrary field $\kk$. We let $\dgMod*{\kk}$ be the closed
symmetric monoidal category of DG vector spaces over $\kk$, with differentials
rising the degree, that is we work with cochain complexes. The tensor product
and internal $\operatorname{hom}$ are denoted by $V\otimes W=V\otimes_\kk W$ and
$\hom{V}{W}=\hom[\kk]{V}{W}$, respectively. We also work with the closed symmetric monoidal
subcategory of graded vector spaces, which we identify with the full subcategory
of DG vector spaces with vanishing differential. All operads are non-symmetric.

%%%%%

\addtocontents{toc}{\SkipTocEntry}\subsection*{Acknowledgements}

The authors would like to thank Víctor Carmona for his help with the proof
of~\Cref{bimodule_same_operad_quillen_equivalence}\eqref{victor}.

%%%%%

\addtocontents{toc}{\SkipTocEntry}\subsection*{Financial support}

G.~J.~was partially supported by the Swedish Research Council (Vetenskapsrådet) Research
Project Grant 2022-03748 `Higher structures in higher-di\-men\-sio\-nal homological algebra.'
F.~M.~is partially supported by grant PID2020-117971GB-C21 funded by
MCIN/AEI/10.13039/501100011033.

%%%%%

\section{Homotopy theory of operads, algebras and bimodules}\label{subsec:homotopy_theory}

In this section we discuss the homotopy theories of DG operads, of algebras over
such operads, and of bimodules over such algebras. All the operads considered in
this article are non-symmetric. We refer the reader to~\cite{LV12} for the
basics on these algebraic structures, but beware that there the authors work
with symmetric operads and hence bimodules are called `modules' therein.
For algebras over non-symmetric operads there is a different notion of module
(akin to one-sided modules over an associative algebra) that we do not consider in this
article, see for example~\cite{Kel01,Kel02} for the case of $\A$-algebras. The
results in this section are expressed in the language of Quillen model
categories, for which our main reference is~\cite{Hov99}.

\subsection{Model category structures}

The category $\dgMod*{\kk}$ of DG vector spaces is a closed symmetric monoidal
model category. The weak
equivalences in the model category $\dgMod*{\kk}$ are the quasi-isomorphisms, the fibrations are
the epimorphisms (=degree-wise surjective maps) and the cofibrations are the
monomorphisms (=degree-wise injective maps). All DG vector spaces are
bifibrant and weak equivalences are homotopy equivalences. The situation is
particularly favorable because, since $\kk$ is a field, the projective and
injective model category structures on $\dgMod*{\kk}$ coincide, see for
example~\cite[Section~2.3]{Hov99}.

In what follows we consider several categories whose objects are (collections
of) DG vector
spaces with extra structure. We say that such a category has the
\emph{transferred projective model category structure} if it has a (necessarily
unique) model category structure whose weak equivalences and fibrations are the
maps that are weak equivalences and fibrations of DG vector spaces after
forgetting the extra structure. All objects are fibrant in such model
categories, but cofibrations and cofibrant objects are in general more difficult
to describe. We also consider Quillen adjunctions between such model categories,
in which case the right adjoints preserve and reflect weak equivalences and
fibrations between arbitrary objects.

Recall that a \emph{DG operad} is a sequence $\O=\{\O<n>\}_{n\geq 0}$ of DG
vector spaces, the DG vector spaces of \emph{operations}, endowed with extra
structure given by \emph{infinitesimal composition} operations
\[
  \circ_i\colon\O<p>\otimes\O<q>\longrightarrow\O<p+q-1>,\qquad
  \mu\otimes\nu\longmapsto \mu\circ_i\nu,\qquad 1\leq i\leq p,\quad q\geq 0,
\]
and a unit cocycle $\id=\id[\O]\in\O<1>^0$ satisfying certain equations (see for
example~\cite[Section~2]{Mur11}) that imply, in particular, that the DG vector space
$\O<1>$ of unary operations is a DG algebra. A \emph{morphism} $\varphi\colon\O\to\P$
of DG operads is a sequence $\varphi=\{\varphi_n\colon\O<n>\to\P<n>\}_{n\geq0}$ of
morphisms of DG vector spaces that are strictly compatible with the
corresponding infinitesimal composition operations and the units in the apparent
sense: we have $\varphi_1(\id[\O])=\id[\P]$ and the squares below commute.
\[
  \begin{tikzcd}
    \O<p>\otimes\O<q>\rar{\circ_i}\dar[swap]{\varphi_p\otimes\varphi_q}&\O<p+q-1>\dar{\varphi_{p+q-1}}\\
    \P<p>\otimes\P<q>\rar{\circ_i}&\P<p+q-1>
  \end{tikzcd}\quad\quad 1\leq i\leq p,\quad q\geq0.
\]

The category $\Operads$ of (non-symmetric) DG operads has the transferred projective model
category structure~\cite{Lyu11,Mur11}. An \emph{elementary cofibration}
$\O\rightarrowtail\P$ in $\Operads$ is the inclusion of a suboperad such that,
as a graded operad, $\P$ is obtained from $\O$ by freely adjoining a new
operation $\nu\in\P$ with differential $d(\nu)\in\O$.
Cofibrations in $\Operads$ are retracts of transfinite compositions of
elementary cofibrations.

Given a DG operad $\O$, the category $\Alg{\O}$ of $\O$-algebras carries the
transferred projective model category structure~\cite[Theorem~1.2]{Mur11}.
Recall that an \emph{$\O$-algebra} is a DG vector $A$ space equipped with
structure maps
\[
  \O<n>\otimes A^{\otimes n}\longrightarrow A,\qquad \mu\otimes
  x_1\otimes\cdots\otimes x_n\longmapsto \mu(x_1,\dots,x_n),\qquad n\geq 0,
\]
satisfying certain equations; a \emph{(strict) morphism} $f\colon A\to B$ between $\O$-algebras
is a morphism of DG vector spaces that is strictly compatible with the structure
maps:
\[
  \begin{tikzcd}
    \O<n>\otimes A^{\otimes n}\rar\dar[swap]{\id*\otimes f^{\otimes n}}&A\dar{f}\\
    \O<n>\otimes B^{\otimes n}\rar&B
  \end{tikzcd}\qquad n\geq0.
\]

A morphism of DG operads $\varphi\colon\O\to\P$ induces a Quillen
adjunction
\begin{equation}\label{algebra_quillen_adjunction}
  \begin{tikzcd}
    \varphi_!\colon\Alg{\O}\rar[shift left]&\Alg{\P}\noloc \varphi^*\lar[shift
    left]
  \end{tikzcd}
\end{equation}
where $\varphi^*$ is the restriction of scalars along $\varphi$. If $\varphi$ is
a weak equivalence, then this adjunction is moreover a Quillen
equivalence~\cite[Theorem~1.3]{Mur11}.

Given a DG operad $\O$, the comma category $\commacat{\O}{\Operads}$ of operads
under $\O$ inherits a model category structure from $\Operads$. There is also a
Quillen adjunction~\cite[Definition~1.5 and Proposition~1.6]{BM09}
\begin{equation}\label{enveloping_operad_adjunction}
  \begin{tikzcd}[row sep = 0ex,
    /tikz/column 1/.append style={anchor=base east},
    /tikz/column 2/.append style={anchor=base west}]
    \Alg{\O}\ar[r,shift right]&\commacat{\O}{\Operads},\ar[l,shift right]\\
    \P<0>&\P,\ar[l,mapsto]\\
    A\ar[r,mapsto]&\O[A].
  \end{tikzcd}
\end{equation}
The right adjoint sends an operad $\P$ under $\O$ to the restriction of scalars
of the initial $\P$-algebra $\P<0>$ along the structure morphism $\O\to\P$. The
DG operad $\O[A]$ is called the \emph{enveloping DG operad} of the $\O$-algebra
$A$.

Let $\O$ be a DG operad and $A$ an $\O$-algebra. Recall that an
\emph{$A$-bimodule $M$ over $\O$}, also called \emph{$\O$-$A$-bimodule}, or
\emph{$\O$-bimodule over $A$}, or simply \emph{$\O$-bimodule} or
\emph{$A$-bimodule} if the other part is understood, is a DG vector space
equipped with structure maps
\begin{align*}
  \O<n>\otimes A^{\otimes p}\otimes M\otimes A^{\otimes q}&\longrightarrow M,&n\geq 1,\quad p+1+q=n,\\
  \mu\otimes x_1\otimes\cdots\otimes x_n&\longmapsto \mu(x_1,\dots,x_n)
\end{align*}
satisfying certain equations (see for example~\cite[Definition~1.1]{BM09}); (strict) morphisms between bimodules are defined in the
obvious way. There is an equivalence of categories~\cite[Theorem 1.10]{BM09}
\[
  \Bimod*{\O,A}\simeq\dgMod*{\O[A]<1>}
\]
between the category of $\O$-$A$-bimodules and the category of left DG modules
over the \emph{enveloping DG algebra} $\O[A]<1>$. Keeping this equivalence of
categories in mind, the following result is a consequence of
\cite[Theorem~4.1]{SS00}.

\begin{proposition}\label{bimodules_model_category}
  Let $\O$ be a DG operad and let $A$ be an $\O$-algebra. The category $\Bimod*{\O,A}$ of $\O$-$A$-bimodules has the transferred projective model category structure.
\end{proposition}

Given a DG operad $\O$ and a morphism of $\O$-algebras $f\colon A\to B$, there is a Quillen adjunction
\begin{equation}\label{bimodule_same_operad_quillen_adjunction}
  \begin{tikzcd}
    f_!\colon\Bimod*{\O,A}\rar[shift left]&\Bimod*{\O,B}\noloc f^*,\lar[shift left]
  \end{tikzcd}
\end{equation}
where $f^*$ is the restriction of scalars along $f$.

Given a DG operad map $\varphi\colon\O\to\P$ and a $\P$-algebra $C$, there is a Quillen adjunction
\begin{equation}\label{bimodule_different_operad_quillen_adjunction}
  \begin{tikzcd}
    \varphi_!\colon\Bimod*{\O,\varphi^*C}\rar[shift left]&\Bimod*{\P,C}\noloc \varphi^*,\lar[shift left]
  \end{tikzcd}
\end{equation}
where $\varphi^*C$ is the $\O$-algebra obtained from $C$ by applying the right
adjoint in \eqref{algebra_quillen_adjunction}. The right adjoint $\varphi^*$ in \eqref{bimodule_different_operad_quillen_adjunction} is the
restriction of scalars along $\varphi$.

Furthermore, if we have a morphism of DG operads $\varphi\colon\O\to\P$ and an
$\O$-algebra $A$, we can consider the Quillen adjunction
\begin{equation}\label{composite_bimodule_quillen_adjunction}
  \begin{tikzcd}
    \varphi_!\colon\Bimod*{\O,A}\rar[shift left]&\Bimod*{\P,\varphi_!A}\noloc\varphi^*,\lar[shift left]
  \end{tikzcd}
\end{equation}
which is the composite of two Quillen adjunctions of the form~\eqref{bimodule_same_operad_quillen_adjunction} and~\eqref{bimodule_different_operad_quillen_adjunction}
\[\begin{tikzcd}
    \Bimod*{\O,A}\ar[r,shift left,"\eta_!"] &
    \Bimod*{\O,\varphi^*\varphi_!A}\ar[r,shift left,"\varphi_!"]\ar[l,shift
    left,"\eta^*"] & \Bimod*{\P,\varphi_!A},\ar[l,shift left,"\varphi^*"]
  \end{tikzcd}
\]
where $\eta\colon A\to \varphi^*\varphi_!A$ is the unit of the
adjunction~\eqref{algebra_quillen_adjunction}.

Under additional assumptions, the Quillen
adjunctions~\eqref{bimodule_same_operad_quillen_adjunction}
and~\eqref{composite_bimodule_quillen_adjunction} are Quillen equivalences.

\begin{proposition}\label{bimodule_same_operad_quillen_equivalence}
  Let $\O$ be a DG operad and let $f\colon A\to B$ be an $\O$-algebra map.
  \begin{enumerate}
  \item\label{bsoqe_quillen_equivalence_conditions} If $f$ satisfies one of the
    two following properties, then the Quillen
    adjunction~\eqref{bimodule_same_operad_quillen_adjunction}
    \[
      \begin{tikzcd}
        f_!\colon\Bimod*{\O,A}\rar[shift left]&\Bimod*{\O,B}\noloc
        f^*,\lar[shift left]
      \end{tikzcd}
    \]
    is a Quillen equivalence:
    \begin{enumerate}
    \item $f$ is a trivial cofibration in $\Alg{\O}$.
    \item\label{between_cofibrant_algebras} $f$ is a weak equivalence between
      cofibrant $\O$-algebras.
    \end{enumerate}
  \item\label{victor} If $f$ is a trivial cofibration in $\Alg{\O}$, then for
    every $\O$-$A$-bimodule $M$ the unit $M\to f^*f_!M$ of the
    adjunction~\eqref{bimodule_same_operad_quillen_adjunction} is a weak
    equivalence.
  \end{enumerate}
\end{proposition}

\begin{proof}
  If $f$ is a trivial cofibration, then the induced morphism of DG operads
  $\O[A]\to\O[B]$ is a trivial cofibration of DG operads under $\O$ since
  \eqref{enveloping_operad_adjunction} is a Quillen adjunction. For the same
  reason, if $f$ is a weak equivalence between cofibrant $\O$-algebras, then
  $\O[A]\to\O[B]$ is a weak equivalence between cofibrant DG operads under $\O$.
  Both properties imply (independently from each other) that
  $\O[A]<1>\to\O[B]<1>$ is a quasi-isomorphism of DG algebras, hence
  \eqref{bimodule_same_operad_quillen_adjunction} is a Quillen equivalence by
  \cite[Theorem~4.3]{SS00}. However, this does not suffice for proving statement
  \eqref{victor}, which is stronger than statement
  \eqref{bsoqe_quillen_equivalence_conditions}.

  Suppose then that $f$ is a trivial cofibration. The unit $M\to f^*f_!M$ is the
  map
  \[
    (\O[A]<1>\to\O[B]<1>)\otimes_{\O[A]<1>}M.
  \]
  In order to prove that it is a quasi-isomorphism we analyse
  $\O[A]<1>\to\O[B]<1>$ as a map of right $\O[A]<1>$-modules. By the usual
  transfinite composition and retract argument, it suffices to assume that
  $f\colon A\to B$ is the pushout of a generating trivial cofibration in
  $\Alg{\O}$. Such a generating trivial cofibration is a free $\O$-algebra map
  on $0\to V$, where $V$ is a contractible DG vector space (=contractible
  complex). Under these assumptions, by
  \cite[Lemma~1.2]{Mur17}, the map
  $\O[A]<1>\to\O[B]<1>$ is the inclusion of the factor $n=0$ of the direct sum
  \[
    \textstyle\O[B]<1>=\bigoplus_{n\geq0}\left(\bigoplus_{1\leq i\leq n+1}\O[A]<n+1>\otimes
    V^{\otimes n} \right)
  \]
  As a right $\O[A]<1>$-module, the $(n,i)$ direct factor is the tensor product of $\O[A]<n+1>$ with the right $\O[A]<1>$-module structure given by $\circ_i\colon\O[A]<n+1>\otimes \O[A]<1>\to \O[A]<n+1>$, and the DG vector space $V^{\otimes n}$. Therefore,
  \[
    \textstyle\O[B]<1>\otimes_{\O[A]<1>}M=\bigoplus_{n\geq0}\left(\bigoplus_{1\leq i\leq
      n+1}(\O[A]<n+1>\otimes_{\O[A]<1>}M)\otimes V^{\otimes n}\right)
  \]
  Since $V$ is contractible, so is $V^{\otimes n}$ for $n\geq 1$. Hence, the
  inclusion of the factor $n=0$ into the latter direct sum is a
  quasi-isomorphism. This inclusion is precisely the unit map
  $(\O[A]<1>\to\O[B]<1>)\otimes_{\O[A]<1>}M$.
\end{proof}

\begin{proposition}\label{composite_bimodule_quillen_equivalence_cofibrant_algebra}
  Given a weak equivalence $\varphi\colon\O\to\P$ in $\Operads$ and a cofibrant
  $\O$-algebra $A$, the Quillen
  adjunction~\eqref{composite_bimodule_quillen_adjunction}
  \[
    \begin{tikzcd}
      \varphi_!\colon\Bimod*{\O,A}\rar[shift
      left]&\Bimod*{\P,\varphi_!A}\noloc\varphi^*\lar[shift left]
    \end{tikzcd}
  \]
  is a Quillen equivalence.
\end{proposition}

\begin{proof}
  By \cite[Theorem 4.3]{SS00}, it suffices to notice that the canonical morphism
  ${\O[A]\to\P[\varphi_!A]}$ between enveloping operads defined by $\varphi$ and
  the unit ${A\to\varphi^*\varphi_!A}$ of the
  adjunction~\eqref{algebra_quillen_adjunction} is a quasi-isomorphism,
  see~\cite[Proposition~3.4.2]{Mur17}.
\end{proof}

We now study alternative conditions under which the Quillen
adjunctions~\eqref{bimodule_same_operad_quillen_adjunction}
and~\eqref{bimodule_different_operad_quillen_adjunction} are Quillen
equivalences.

\begin{definition}\label{excellent_operad}
  A DG operad $\O$ is \emph{excellent} if the enveloping operad functor
  \[
    \Alg{\O}\longrightarrow\Operads,\qquad A\longmapsto\O[A],
  \]
  preserves quasi-isomorphisms.
\end{definition}

\Cref{excellent_operad} agrees with~\cite[Definition~3.1]{Mur17} because of the
aforementioned special properties of the model category $\dgMod*{\kk}$, namely
that every DG vector space is bifibrant and that the weak equivalences between
them are the homotopy equivalences.

\begin{example}[{\cite[Proposition~3.4.4
    and~3.4.5]{Mur17}}]\label{example_excellent_operads}
  The following DG operads are excellent:
  \begin{enumerate}
  \item $\Ass$, whose algebras are non-unital associative DG algebras.
  \item $\uAss$, whose algebras are unital associative DG algebras.
  \item Every cofibrant object in $\Operads$. For example:
    \begin{enumerate}
    \item $\A$, whose algebras are $\A$-algebras.
    \item $\A[k]$, whose algebras are $\A[k]$-algebras, $k\geq 1$.
    \end{enumerate}
  \end{enumerate}
  The operad $\A$ is a cofibrant resolution of $\Ass$ in $\Operads$, and the operads $\A[k]$ are the stages of the skeletal filtration of $\A$. We only consider non-unital $\A$-algebras and $\A[k]$-algebras in this article.
\end{example}

\begin{proposition}\label{bimodule_same_operad_quillen_equivalence_excellent}
  Let $\O$ be an excellent DG operad and let $f\colon A\to B$ be a weak
  equivalence between $\O$-algebras. Then, the Quillen
  adjunction~\eqref{bimodule_same_operad_quillen_adjunction}
  \[
    \begin{tikzcd}
      f_!\colon\Bimod*{\O,A}\rar[shift left]&\Bimod*{\O,B}\noloc f^*\lar[shift left]
    \end{tikzcd}
  \]
  is a Quillen equivalence.
\end{proposition}

\begin{proof}
  By definition of excellent DG operad, the enveloping operad map
  $\O[A]\to\O[B]$ is a weak equivalence. Hence $\O[A]<1>\to\O[B]<1>$ is a
  quasi-isomorphism of DG algebras and the claim follows from
  \cite[Theorem~4.3]{SS00}.
\end{proof}

\begin{proposition}\label{composite_bimodule_quillen_equivalence_excellent}
  Let $\varphi\colon\O\to\P$ be a weak equivalence in $\Operads$ between excellent DG
  operads and let $C$ be a $\P$-algebra. Then, the Quillen
  adjunction~\eqref{bimodule_different_operad_quillen_adjunction}
  \[
    \begin{tikzcd}
      \varphi_!\colon\Bimod*{\O,\varphi^*C}\rar[shift left]&\Bimod*{\P,C}\noloc
      \varphi^*\lar[shift left]
    \end{tikzcd}
  \]
  is a Quillen equivalence.
\end{proposition}

\begin{proof}
  As in the proof of
  \Cref{composite_bimodule_quillen_equivalence_cofibrant_algebra}, we prove that
  the canonical map $\O[\varphi^*C]\to\P[C]$ between enveloping operads is a
  weak equivalence.

  Let $A\to \varphi^*C$ be a cofibrant replacement in $\Alg{\O}$. The adjoint
  map $\varphi_!A\to C$ is a weak equivalence in $\Alg{\O}$ since
  \eqref{algebra_quillen_adjunction} is a Quillen equivalence. We have a
  commutative square in $\Operads$
  \[
    \begin{tikzcd}
      \O[A]\ar[r]\ar[d] & \P[\varphi_!A]\ar[d] \\
      \O[\varphi^*C]\ar[r] & \P[C]
    \end{tikzcd}
  \]
  in which the vertical maps are obtained by applying the enveloping operad
  functor to the maps $A\to \varphi^*C$ and $\varphi_!A\to C$, and hence they
  are weak equivalences since the DG operads $\O$ and $\P$ are excellent. The
  top horizontal map is a weak equivalence, as noticed in the proof of
  \Cref{composite_bimodule_quillen_equivalence_cofibrant_algebra}. Consequently,
  by the $2$-out-of-$3$ property, the bottom horizontal map is also a weak
  equivalence.
\end{proof}

The \emph{category $\Bimod*{\O}$ of all $\O$-bimodules} is defined as follows.
Objects are pairs $(A,M)$ given by an $\O$-algebra $A$ and an $A$-bimodule $M$.
Morphisms are pairs $(f,g)\colon (A,M)\to(B,N)$ given by a morphism of
$\O$-algebras $f\colon A\to B$ and a morphism of $A$-bimodules $g\colon M\to
f^*N$, where $f^*$ is the right adjoint
in~\eqref{bimodule_same_operad_quillen_adjunction}. The forgetful functor
\begin{equation}\label{Grothendieck_fibration}
  \Phi\colon\Bimod*{\O}\longrightarrow\Alg{\O},\qquad (A,M)\longmapsto A,
\end{equation}
is a Grothendieck bifibration with fibres
\[
  \Bimod*{\O,A},\qquad A\in\Alg{\O}.
\]
The category of all bimodules over an operad also has a natural model category
structure.

\begin{proposition}
  Let $\O$ be a DG operad. The category $\Bimod*{\O}$ of all $\O$-bimodules has the transferred projective model category structure.
\end{proposition}

\begin{proof}
  Objects in $\Bimod*{\O}$ have an underlying pair of DG vector spaces, so a morphism is a weak equivalence or fibration in the tentative transferred projective model category structure if and only if each coordinate is a weak equivalence or fibration in $\dgMod*{\kk}$.

  To prove the result we apply~\cite[2.3~Theorem]{Sta12a}, which fixes a gap
  in~\cite[Theorem~5.1]{Roi94}, see also~\cite[Theorem~4.4]{CM20} for a more
  general approach. We need to verify that the five conditions in~\cite[2.3~Theorem]{Sta12a} are satisfied.
  Indeed, (i$'$) and (ii$'$) correspond to the existence of the transferred
  projective model category structures on $\Alg{\O}$ (\cite[Theorem~1.2]{Mur11})
  and $\Bimod*{\O,A}$ for each $\O$-algebra $A$
  (\Cref{bimodules_model_category}), and (iii$'$) is the fact that the adjoint
  pair~\eqref{bimodule_same_operad_quillen_adjunction} is a Quillen adjunction
  for each $\O$-algebra morphism $f\colon A\to B$. Condition (a) also holds
  since the functor $f^*$ in the adjoint
  pair~\eqref{bimodule_same_operad_quillen_adjunction} preserves and reflects
  weak equivalences for every morphism of $\O$-algebras $f\colon A\to B$, not just for weak equivalences, and (b) is~\Cref{bimodule_same_operad_quillen_equivalence}\eqref{victor}.
\end{proof}

\begin{definition}
  \label{def:quasi-isomorphic_algebra-bimodule_pairs}
  We say that two pairs $(A,M)$ and $(B,N)$, consisting of $\O$-algebras $A$
  and $B$, an $\O$-$A$-bimodule $M$ and $\O$-$B$-bimodule $N$ are
  \emph{quasi-isomorphic} if they are isomorphic in the homotopy category of the
  category $\Bimod*{\O}$ of all $\O$-bimodules.
\end{definition}

A morphism of DG operads $\varphi\colon\O\to\P$ also gives rise to a Quillen adjunction
\begin{equation}\label{all_bimodules_quillen_adjunction}
  \begin{tikzcd}
    \varphi_!\colon\Bimod*{\O}\rar[shift left]&\Bimod*{\P}\noloc\varphi^*,\lar[shift left]
  \end{tikzcd}
\end{equation}
where the right adjoint $\varphi^*$ is defined as in~\eqref{bimodule_different_operad_quillen_adjunction} and the left adjoint $\varphi_!$ is defined as in~\eqref{composite_bimodule_quillen_adjunction}.

\begin{proposition}\label{bimodule_quillen_equivalence}
  Given a weak equivalence $\varphi\colon\O\to\P$ in $\Operads$, the Quillen
  adjunction~\eqref{all_bimodules_quillen_adjunction}
  \[
    \begin{tikzcd}
      \varphi_!\colon\Bimod*{\O}\rar[shift left]&\Bimod*{\P}\noloc\varphi^*,\lar[shift left]
    \end{tikzcd}
  \]
  is a Quillen equivalence.
\end{proposition}

\begin{proof}
  Let $(A,M)$ be a pair consisting of an $\O$-algebra $A$ and an
  $\O$-$A$-bimodule $M$. The unit of the adjunction
  $(A,M)\to\varphi^*\varphi_!(A,M)$ consists of the unit $A\to
  \varphi^*\varphi_!A$ of the adjunction~\eqref{algebra_quillen_adjunction} and the unit $M\to\varphi^*\varphi_!M$ of the adjunction~\eqref{composite_bimodule_quillen_adjunction}. The pair $(A,M)$ is cofibrant in $\Bimod*{\O}$ if and only if $A$ is cofibrant in $\Alg{\O}$ and $M$ is cofibrant in $\Bimod*{\O,A}$, see~\cite[2.3~Remark]{Sta12a}. In that case, $A\to \varphi^*\varphi_!A$ is a weak equivalence since the adjoint pair~\eqref{algebra_quillen_adjunction} is a Quillen equivalence, and $M\to\varphi^*\varphi_!M$ is also a weak equivalence by \Cref{composite_bimodule_quillen_equivalence_cofibrant_algebra}.
\end{proof}

\subsection{Spaces of algebras}\label{subsec:spaces_of_algebras}

Given a category $\C$, we denote its \emph{nerve} by $|\C|$. If $\C$ is a model
category, we let $w\C\subseteq\C$ be the subcategory of weak equivalences, and the \emph{classifying space} of $\C$ is $|w\C|$, called classification complex in \cite[\S2]{DK84}.

We denote the Dwyer--Kan mapping spaces (Kan complexes) between DG operads $\O$
and $\P$ in the model category $\Operads$ by
\[
  \Map{\O}{\P}=\Map[\Operads]{\O}{\P}.
\]
These are defined as maps from a cosimplicial resolution of the source DG operad
$\O$, hence they preserve fibrations and trivial fibrations on the second
variable. Points (=vertices) in the mapping space $\Map{\O}{\P}$ are morphisms
$\O_{\infty}\to\P$ from a cofibrant replacement $\O_{\infty}$ of $\O$, which is
a cofibrant DG operad $\O_{\infty}$ equipped with a trivial fibration
$\O_{\infty}\to\O$. In particular, maps $\O\to\P$ define special points, the
composites ${\O_{\infty}\to\O\to\P}$. If $\O$ is already cofibrant, we always
take $\O_{\infty}=\O$. The set $\pi_0\Map{\O}{\P}$ of connected components of
the mapping space $\Map{\O}{\P}$ is the set of homotopy classes of operad maps
$\O_\infty\to\P$ or, equivalently, the set of morphisms $\O\to\P$ in the
homotopy category of the model category $\Operads$.

Let $V$ be a DG vector space. Denote by $\E{V}$ the \emph{endomorphism DG
  operad} of $V$, whose DG vector spaces or operations are
\[
  \E{V}<n>\coloneqq\hom[\kk]{V^{\otimes n}}{V},\qquad n\geq 0,
\]
and with infinitesimal composition $\circ_i$ given by composition of multilinear
maps at the $i$-th slot. The datum of a morphism of DG operads $\O\to\E{V}$ is
equivalent to that of an $\O$-algebra structure on $V$: the $\O$-algebra
structure morphisms ${\O<n>\otimes V^{\otimes n}\to V}$ are the adjoints of the
maps ${\O<n>\to\E{V}<n>}$, $n\geq 0$.

\begin{definition}\label{def:gauge_O_infty_isomorphism_algebras}
  Let $A$ and $B$ be two $\O_{\infty}$-algebras with the same underlying DG
  vector space $V$.
  \begin{enumerate}
  \item A
    \emph{gauge $\O_{\infty}$-isomorphism} or \emph{$\O_{\infty}$-isomorphism
      with identity linear part} $A\leadsto B$ is a path between the
    corresponding points in $\Map{\O}{\E{V}}$, that is a homotopy between the
    corresponding structure maps $\O_{\infty}\to\E{V}$.
  \item The \emph{identity gauge $\O_{\infty}$-isomorphism} $A\leadsto A$ on a
    given $\O_{\infty}$-algebra $A$ with underlying DG vector space $V$ is the
    constant path at the point in $\Map{\O}{\E{V}}$ corresponding to $A$, that
    is the trivial homotopy on the structure map $\O_{\infty}\to\E{V}$.
  \end{enumerate}
  Composition of gauge $\O_{\infty}$-isomorphisms is defined by concatenation of
  paths in $\Map{\O}{\E{V}}$ or, equivalently, by pasting homotopies between the
  corresponding structure maps $\O_{\infty}\to\E{V}$.
\end{definition}

\begin{remark}
  \Cref{def:gauge_O_infty_isomorphism_algebras} is consistent with the usual
  definitions (\cite[Section~10.2]{LV12}) in the case where $\O$ is a Koszul DG
  operad and $\O_{\infty}$ is its minimal resolution, see
  \cite[Section~5.2.1]{Fre09}.
\end{remark}

The following result is a version of~\cite[Theorem~1.1.5]{Rez96} for DG vector
spaces.

\begin{theorem}[{\cite[Theorem~4.6]{Mur14}}]\label{rezk_theorem}
  Let $\O$ be a DG operad. The forgetful functor 
  \[\Alg{\O}\longrightarrow\dgMod*{\kk}\] 
  that sends an $\O$-algebra $A$ to its underlying DG vector space induces a map
  between classifying spaces
  \[|w\Alg{\O}|\longrightarrow|w\dgMod*{\kk}|\]
  whose homotopy fibre at a DG vector space $V$ is the mapping space
  \[\Map{\O}{\E{V}}.\]
\end{theorem}

\begin{definition}[{\cite[\S2]{Mur16b}}]
  Recall that a map of spaces $f\colon X\to Y$ is a \emph{homotopy monomorphism} when either of the two following equivalent conditions holds:
  \begin{enumerate}
  \item The map $f$ corestricts to a weak equivalence between $X$ and a subset of connected components of $Y$.
  \item The map $f$ induces an injection $\pi_0f\colon \pi_0X\hookrightarrow\pi_0Y$ and isomorphisms $\pi_n\colon\pi_n(X,x_0)\cong\pi_n(Y,f(x_0))$, for any $x_0\in X$ and $n\geq 1$.
  \end{enumerate}
\end{definition}

The following theorem is the main result of \cite{Mur16b}.

\begin{theorem}[{\cite[Theorem 1.2]{Mur16b}}]\label{homotopy_units}
  The canonical morphism $\Ass\to\uAss$ from the associative DG operad to the unital associative DG operad, whose induced restriction of scalars functor
  $\Alg{\uAss}\rightarrow\Alg{\Ass}$ is the forgetful functor from the category of unital DG algebras to the category of possibly non-unital DG algebras, is a homotopy epimorphism, i.e.~for any DG operad $\O$, the induced map
  \[\Map{\uAss}{\O}\longrightarrow\Map{\Ass}{\O}\]
  is a homotopy monomorphism.
\end{theorem}

As a consequence of \Cref{rezk_theorem,homotopy_units} and \cite[Lemma 6.3]{Mur14}, the following map between classifying spaces is a homotopy monomorphism:

\begin{proposition}[{\cite[Proposition 6.2]{Mur14}}]\label{homotopy_monomorphism_algebras}
  The forgetful functor $\Alg{\uAss}\to\Alg{\Ass}$ from the category of unital DG algebras to the category of possibly non-unital DG algebras induces a homotopy monomorphism
  \[|w\Alg{\uAss}|\longrightarrow|w\Alg{\Ass}|.\]
\end{proposition}

\begin{corollary}\label{quasi-iso_unit_algebras}
  Let $A$ and $B$ be unital DG algebras. Then, 
  $A$ and $B$ are quasi-isomorphic as \emph{unital} DG algebras, i.e.~as $\uAss$-algebras, if and only if they are quasi-isomorphic as \emph{non-unital} DG algebras, i.e.~as $\Ass$-algebras.
\end{corollary}

\begin{remark}
  \Cref{quasi-iso_unit_algebras} is equivalent to saying that, given unital DG algebras $A$ and $B$, if there exist quasi-isomorphisms of DG algebras $A\leftarrow C\rightarrow B$ with $C$ possibly non-unital, then we can find another diagram $A\leftarrow C'\rightarrow B$ with $C'$ unital where the two arrows are quasi-isomorphisms of DG algebras preserving the units.
\end{remark}

\Cref{rezk_theorem} can be used to define minimal models in great generality.

\begin{definition}\label{def:minimal_model_algebra} We make the following
  definitions:
  \begin{enumerate}
  \item Let $V$ be a DG vector space. A \emph{cocycle selection map}
    ${i\colon\dgH{V}\to V}$ is a morphism of DG vector spaces, regarding
    $\dgH{V}$ as a DG vector space with vanishing differential, which sends each
    cohomology class to a representing cocycle:
    \[
      \forall x\in \H{V},\qquad [i(x)]=x\in\H{V}.
    \]
    Notice that a cocycle selection map is necessarily a quasi-isomorphism.
  \item Given an $\O$-algebra $A$, a \emph{minimal model} of $A$ is a point in
    $\Map{\O}{\E{\dgH{A}}}$ defined by a cocycle selection map $i\colon\dgH{A}\to A$ in
    $\dgMod*{\kk}$, regarded as a path in $|w\dgMod*{\kk}|$, see
    \Cref{rezk_theorem}. More precisely, the path $i$ induces a weak equivalence
    between mapping spaces/homotopy fibres
    \[
      i_*\colon\Map{\O}{\E{\H{A}}}\stackrel{\sim}{\longrightarrow}\Map{\O}{\E{A}}
    \]
    and hence, in particular, a bijection on path connected components:
    \[
      i_*\colon\pi_0\Map{\O}{\E{\H{A}}}\stackrel{\cong}{\longrightarrow}\pi_0\Map{\O}{\E{A}}.
    \]
    A minimal model of $A$ is then a point $\O_\infty\to\E{\H{A}}$ in the source
    mapping space such that
    \[
      i_*[\O_\infty\to\E{\H{A}}]=[\O_\infty\to\O\to\E{A}]\in\pi_0\Map{\O}{\E{A}},
    \]
    where $\O\to\E{A}$ is the $\O$-algebra structure on $A$.
  \end{enumerate}
\end{definition}

Minimal models are unique up to gauge $\O_{\infty}$-isomorphism in the following
sense.

\begin{remark}\label{rem:minimal_model_algebra}
  \Cref{def:minimal_model_algebra} agrees with the usual definition of minimal
  models via homotopy transfer. Indeed, as can be seen from the proof of
  \cite[Theorem~4.6]{Mur14}, an $\O_{\infty}$-algebra structure on $\dgH{A}$
  defined by a cocycle selection map ${i\colon\dgH{A}\to A}$ is obtained by means
  of the following diagram:
  \begin{equation*}
    \begin{tikzcd}
      \E{\dgH{A}}&\E{i}\ar[l,"s"']\ar[r,"t"]&\E{A}\\
      &\O_{\infty}\rar\ar[lu,dashed]\ar[u,dashed]\ar[ru,""{name=X}]&
      \arrow[Leftarrow,from=X,to=1-2]\O\uar
    \end{tikzcd}
  \end{equation*}
  Here, the top horizontal arrows are the `source' and `target'
  quasi-isomorphisms from the endomorphism DG operad $\E{i}$ of $i$ in the category of maps
  in $\dgMod*{\kk}$ with the point-wise tensor product, hence\footnote{The DG
    operad $\E{i}$ is such that a
    map $\O_\infty\longrightarrow\E{i}$ is equivalent to the data of $\O_\infty$-algebra
    structures on $\H{A}$ and $A$ such that $i$ becomes a
    strict morphism of $\O_\infty$-algebras.}
  \[
    \E{i}<n>\coloneqq\hom[\kk]{i^{\otimes n}}{i},\qquad n\geq 0.
  \]
  The solid diagonal arrow is the $\O_{\infty}$-algebra structure on $A$
  corresponding to the special point $\O_{\infty}\to\O\to\E{A}$ of
  $\Map{\O}{\E{A}}$. Since $\O_{\infty}$ is cofibrant and all operads are
  fibrant in $\Operads$, we can lift this $\O_\infty$-algebra structure up to
  a homotopy ($\Rightarrow$) along $t$. Then, we compose the lift with the
  `source' quasi-isomorphism $s$ in order to obtain an $\O_\infty$-algebra
  structure on $\H{A}$. This diagram is equivalent to giving another
  $\O_{\infty}$-algebra $A'$ with the same underlying DG vector space as $A$
  (the composite of the vertical dashed arrow and $t$), a gauge
  $\O_{\infty}$-isomorphism $A'\leadsto A$ (the homotopy $\Rightarrow$), and an
  $\O_{\infty}$-algebra structure on $\dgH{A}$ (the dashed arrow on the left)
  such that $i\colon\dgH{A}\to A'$ is a morphism of $\O_{\infty}$-algebras (see
  \cite[Corollary 2.17]{Mur14}), that is a diagram of $\O_{\infty}$-algebras of
  the form
  \[
    \dgH{A}\stackrel{i}{\longrightarrow}A'\leadsto A.
  \]
  This agrees with the construction of minimal models given in
  \cite[Theorem~10.4.3]{LV12} in the case of Koszul DG operads. Since any two
  cocycle selection maps are homotopic, minimal models are indeed unique up to
  gauge $\O_{\infty}$-isomorphism.
\end{remark}

\begin{remark}
  The cocycle selection map $i\colon\dgH{A}\to A$ is clearly a monomorphism,
  hence a cofibration in $\dgMod*{\kk}$, but \cite[Theorem~4.6]{Mur14} uses 
  fibrations instead. Nevertheless, \cite[Theorem~4.6]{Mur14} also works
  replacing fibrations with cofibrations. Alternatively, we can replace
  ${i\colon\dgH{A}\to A}$ with a retraction ${p\colon A\to\dgH{A}}$, which is
  then a fibration and exists since all objects in $\dgMod*{\kk}$ are bifibrant.
\end{remark}

We can use \Cref{rezk_theorem} to prove that quasi-isomorphic $\O$-algebras can
be detected by means of their minimal models.

\begin{proposition}\label{minimal_models_quasi_isomorphic_algebras}
  Let $A$ and $B$ be $\O$-algebras. The following statements are equivalent:
  \begin{enumerate}
  \item\label{it:quasi-isomorphic} The $\O$-algebras $A$ and $B$ are
    quasi-isomorphic.
  \item\label{it:gauge_isomorphic} The $\O$-algebras $A$ and $B$ have gauge
    $\O_{\infty}$-isomorphic minimal models.
  \end{enumerate}
\end{proposition}

\begin{proof}
  \eqref{it:quasi-isomorphic}$\Rightarrow$\eqref{it:gauge_isomorphic} Since all
  objects in $\Alg{\O}$ are fibrant, there exist weak equivalences of
  $\O$-algebras $A\twoheadleftarrow C\rightarrow B$, where the first map is also a
  fibration. Every cocycle selection map $\dgH{A}\to A$ is a cofibration in
  $\dgMod*{\kk}$, and hence it can be lifted to a quasi-isomorphism $\dgH{A}\to
  C$ and then composed to obtain a further quasi-isomorphism $\dgH{A}\to B$:
  \[
    \begin{tikzcd}
      &C\dar[two heads]\rar&B\\
      \H{A}\rar[swap]{i}\urar[dotted]&A
    \end{tikzcd}
  \]
  Therefore, the corresponding points in $\Map{\O}{\E{\H{A}}}$, given by minimal
  models of $A$ and $B$, are connected by a path, that is by a gauge
  $\O_{\infty}$-isomorphism.

  \eqref{it:gauge_isomorphic}$\Rightarrow$\eqref{it:quasi-isomorphic} This
  implies that $\dgH{A}=\dgH{B}$ and that we have cocycle selection maps
  $\dgH{A}\to A$ and $\dgH{A}\to B$ whose corresponding points in
  $\Map{\O}{\E{\dgH{A}}}$ are connected by a path. Therefore, by
  \Cref{rezk_theorem}, $A$ and $B$ are also connected by a path in
  $|w\Alg{\O}|$, that is they are quasi-isomorphic as $\O$-algebras.
\end{proof}

\begin{remark}
  \label{rmk:unitality_and_minimal_models}
  Let $A$ and $B$ be unital DG algebras. 
  It follows from
  \Cref{quasi-iso_unit_algebras,minimal_models_quasi_isomorphic_algebras} for the DG operads $\O=\Ass$ and
  $\O_\infty=\A$ that $A$ and $B$ are quasi-isomorphic as \emph{unital} DG algebras if and
  only if their minimal models are gauge $\A$-isomorphic. We refer the reader to
  \cite[Section~I.2]{Sei08} for a survey of the various notions of `unitality'
  for $\A$-algebras.
\end{remark}

\subsection{Spaces of algebras and bimodules}\label{subsec:spaces_of_algebras_and_bimodules}

\Cref{rezk_theorem} admits an analogue for algebra-bimodule pairs. Below,
$\E{V,W}$ denotes the \emph{linear endomorphism DG operad} of the DG vector
spaces $V$ and $W$ introduced in~\cite{BM09}, whose DG vector spaces of
operations are given by
\begin{equation}\label{aritywise_linear_endomorphism_operad}
	\E{V,W}<n>\coloneqq\E{V}<n>\oplus\bigoplus_{p+1+q=n}\hom[\kk]{V^{\otimes p}\otimes W\otimes V^{\otimes q}}{W},\quad n\geq1.
\end{equation}
This is a suboperad of the endomorphism DG operad of the direct sum $V\oplus W$,
\begin{equation}\label{linear_endomorphism_inclusion}
  \E{V,W}\subset\E{V\oplus W},
\end{equation}
and contains the endomorphism DG operad $\E{V}\subset\E{V,W}$ as a suboperad (with a different unit, though).
The datum of a morphism of DG operads $\O\to\E{V,W}$ is equivalent to the data
of an $\O$-algebra structure on $V$ and an $\O$-$V$-bimodule structure on $W$.
Indeed, the projection onto the first factor
$\E{V,W}<n>\twoheadrightarrow\E{V}<n>$, $n\geq 0$, defines an operad map
\begin{equation}\label{linear_endomorphism_operad}
  \E{V,W}\twoheadrightarrow\E{V}
\end{equation}
(preserving units), actually a fibration in $\Operads$; the $\O$-algebra structure on $V$ induced by $\O\to\E{V,W}$ is the composite
\[\O\to\E{V,W}\twoheadrightarrow\E{V}\]
and the $\O$-$V$-bimodule structure maps $\O<n>\otimes V^{\otimes p}\otimes
W\otimes V^{\otimes q}\to W$ are the adjoints of the components
\[
  \O<n>\to\hom[\kk]{V^{\otimes p}\otimes W\otimes V^{\otimes q}}{W},\qquad
  n\geq 0,\quad p+1+q=n. 
\]

\begin{theorem}\label{rezk_theorem_for_bimodules}
  Let $\O$ be a DG operad. The forgetful functor
  \[\Bimod*{\O}\longrightarrow\dgMod*{\kk}\times \dgMod*{\kk}\]
  that sends a pair $(A,M)$ consisting of an $\O$-algebra $A$ and an
  $A$-bimodule $M$ to the pair of their underlying DG vector spaces
  induces a map
  \[|w\Bimod*{\O}|\longrightarrow|w\dgMod*{\kk}|\times|w\dgMod*{\kk}|\] whose
  homotopy fibre at a pair of DG vector spaces $(V,W)$ is the mapping space
  \[\Map{\O}{\E{V,W}}.\]
\end{theorem}

The proof of \Cref{rezk_theorem_for_bimodules} is entirely analogous to the
proof of \Cref{rezk_theorem} in \cite{Mur14}, replacing endomorphism DG
operads $\E{V}$ with the linear endomorphism DG operads $\E{V,W}$, and hence we
omit it.

The following result is a consequence of \Cref{rezk_theorem_for_bimodules,homotopy_units} and \cite[Lemma 6.3]{Mur14}.

\begin{proposition}\label{homotopy_monomorphism_algebras_bimodules}
  The forgetful functor $\Bimod*{\uAss}\to\Bimod*{\Ass}$ from the category of bimodules over unital DG associative algebras to the category of bimodules over possibly non-unital DG associative algebras induces a homotopy monomorphism
  \[|w\Bimod*{\uAss}|\longrightarrow|w\Bimod*{\Ass}|.\]
\end{proposition}

\begin{corollary}\label{quasi-iso_unit_algebras_bimodules}
  Let $A$ and $B$ be unital DG algebras, $M$ a unital $A$-bimodule and $N$ a unital $B$-bimodule. Then 
  $(A,M)$ and $(B,N)$ are quasi-isomorphic in $\Bimod*{\uAss}$ if and only if they are quasi-isomorphic in $\Bimod*{\Ass}$.
\end{corollary}

\begin{remark}
  \Cref{quasi-iso_unit_algebras} is equivalent to saying that, given unital DG algebras $A$ and $B$, a unital $A$-bimodule $M$ and a unital $B$-bimodule $N$, if there exist quasi-isomorphisms of DG algebras $A\leftarrow C\rightarrow B$ with $C$ possibly non-unital and quasi-isomorphisms of $C$-bimodules $M\leftarrow P\rightarrow N$, then we can find diagrams $A\leftarrow C'\rightarrow B$ and $M\leftarrow P'\rightarrow N$ with $C'$ a unital DG algebra and $P'$ a unital $C'$-bimodule such that the two arrows in $A\leftarrow C'\rightarrow B$ are quasi-isomorphisms of DG algebras preserving the units and the two arrows in $M\leftarrow P'\rightarrow N$ are quasi-isomorphisms of $C'$-bimodules.
\end{remark}

We have the following straightforward comparison result
between \Cref{rezk_theorem,rezk_theorem_for_bimodules}.

\begin{proposition}\label{rezk_comparison}
  Let $\O$ be a DG operad. Consider the commutative diagram
  \[\begin{tikzcd}
      {|w\Bimod*{\O}|}\ar[r]\ar[d]&{|w\dgMod*{\kk}|\times|w\dgMod*{\kk}|}\ar[d,"p_1"]\\
      {|w\Alg{\O}|}\ar[r]&{|w\dgMod*{\kk}|}
    \end{tikzcd}\] where the horizontal maps are the maps in
  \Cref{rezk_theorem,rezk_theorem_for_bimodules}, the left vertical map is induced by 
  the Grothendieck bifibration~\eqref{Grothendieck_fibration}
  \[
    \Phi\colon\Bimod*{\O}\longrightarrow\Alg{\O},\qquad (A,M)\longmapsto A,
  \]
  and the right vertical
  map is the projection onto the first factor. Let $V$ and $W$ be DG vector
  spaces. The induced map between the homotopy fibres of the top and bottom
  horizontal maps at $(V,W)$ and $V=p_1(V,W)$, respectively, is the map (a Kan
  fibration)
  \[\Map{\O}{\E{V,W}}\longrightarrow\Map{\O}{\E{V}}\]
  induced by \eqref{linear_endomorphism_operad}.
\end{proposition}

\begin{definition}
  \label{def:gauge_iso_pairs}
  We make the following definitions:
  \begin{enumerate}
  \item Let $A$ and $B$ be two $\O_{\infty}$-algebras $A$ and $B$ with the same
    underlying DG vector space $V$, and $M$ an $\O_{\infty}$-$A$-bimodule and
    $N$ an $\O_{\infty}$-$B$-bimodule, both with the same underlying DG vector
    space $W$. A
    \emph{gauge $\O_{\infty}$-iso\-morph\-ism} or \emph{$\O_{\infty}$-isomorphism
      with identity linear part} $(A,M)\leadsto (B,N)$ is a path
    between the corresponding points in $\Map{\O}{\E{V,W}}$, that is a homotopy
    between the corresponding structure maps $\O_{\infty}\to\E{V,W}$.
  \item The \emph{identity gauge $\O_{\infty}$-isomorphism} $(A,M)\leadsto
    (A,M)$ on a given pair with underlying DG vector spaces $(V,W)$ is the
    constant path at the corresponding point in $\Map{\O}{\E{V,W}}$, that is the
    trivial homotopy on the structure map $\O_{\infty}\to\E{V,W}$.
  \end{enumerate}
  Composition of such gauge $\O_{\infty}$-isomorphisms is defined by concatenation
  of paths in $\Map{\O}{\E{V,W}}$ or, equivalently, by pasting homotopies
  between the corresponding structure maps $\O_{\infty}\to\E{V,W}$.
\end{definition}

\Cref{rezk_theorem_for_bimodules} can be used to define minimal models for pairs
consisting of an $\O$-algebra and a compatible $\O$-bimodule (compare with
\Cref{def:minimal_model_algebra}).

\begin{definition}\label{def:minimal_model_simultaneous}
  Let $A$ be an $\O$-algebra $A$ and an $M$ an $\O$-$A$-bimodule. A
  \emph{minimal model} of the pair $(A,M)$ is a point in
  $\Map{\O}{\E{\dgH{A},\dgH{M}}}$, that is an $\O_{\infty}$-algebra structure on
  $\dgH{A}$ and an $\O_{\infty}$-bimodule structure on $\dgH{M}$ over it,
  defined by cocycle selection maps $\dgH{A}\to A$ and $\dgH{M}\to M$ in
  $\dgMod*{\kk}$, regarded as a path in $|w\dgMod*{\kk}|\times|w\dgMod*{\kk}|$,
  see \Cref{rezk_theorem_for_bimodules}. Since a minimal model of $(A,M)$ contains an
  underlying minimal model of $A$, namely its image under the map
  \[
    \Map{\O}{\E{\dgH{A},\dgH{M}}}\longrightarrow\Map{\O}{\E{\dgH{A}}}
  \]
  in \eqref{linear_endomorphism_operad}, we also speak of a \emph{minimal
    model of $M$} that is \emph{compatible} with the underlying minimal model of $A$.
\end{definition}

Minimal models of pairs in the sense of \Cref{def:minimal_model_simultaneous}
are unique up to gauge $\O_{\infty}$-isomorphism in the following sense (compare
with \Cref{rem:minimal_model_algebra}).

\begin{remark}
  \label{rmk:minimal_model_simultaneous}
  Minimal models for pairs $(A,M)$ consisting of an $\O$-algebra $A$ and an
  $\O$-$A$-bimodule $M$ can be constructed by homotopy transfer as in
  \Cref{rem:minimal_model_algebra}, from cocycle selection maps
  $i\colon\H{A}\to\H{A}$ and $j\colon\H{M}\to M$:
  \[
    \begin{tikzcd}
      \E{\dgH{A},\H{M}}&\E{i,j}\ar[l,"s"']\ar[r,"t"]&\E{A,M}\\
      &\O_{\infty}\rar\ar[lu,dashed]\ar[u,dashed]\ar[ru,""{name=X}]&
      \arrow[Leftarrow,from=X,to=1-2]\O\uar
    \end{tikzcd}
  \]
  Here, $\E{i,j}$ is the linear endomorphism DG operad of the pair $(i,j)$,
  viewed as objects of the category of maps in $\dgMod*{\kk}$ with the point-wise
  tensor product. A diagram as above corresponds to an $\O_{\infty}$-algebra $A'$ and an
  $\O_{\infty}$-$A'$-bimodule $M'$ with the same underlying DG vector spaces as
  $A$ and $M$, a gauge $\O_{\infty}$-isomorphism $(A',M')\leadsto (A,M)$, and an
  $\O_{\infty}$-algebra structure on $\dgH{A}$ and an $\O_{\infty}$-bimodule
  structure on $\dgH{M}$ over it, such that the given cocycle selection maps
  $\dgH{A}\to A'$ and $\dgH{M}\to M'$ are morphisms of $\O_{\infty}$-algebras
  and $\O_{\infty}$-bimodules, respectively,
  \[
    (\dgH{A},\dgH{M})\longrightarrow(A',M')\leadsto (A,M).
  \]
\end{remark}

\begin{remark}
  It follows from \Cref{rmk:minimal_model_simultaneous} that, when $\O=\Ass$ and
  $\O_\infty=\A$, \Cref{def:minimal_model_simultaneous} agrees with the
  definition of minimal models of pairs given in~\cite[Definition~5.67]{JM25}.
\end{remark}

By \Cref{rezk_theorem_for_bimodules}, quasi-isomorphic pairs of an $\O$-algebra
and a compatible $\O$-bimodule can be distinguished by their minimal models. The
proof is almost identical to that of
\Cref{minimal_models_quasi_isomorphic_algebras}, hence we omit it.

\begin{proposition}\label{minimal_models_quasi_isomorphic_algebras-bimodules}
  Let $A$ and $B$ be $\O$-algebras, $M$ an $\O$-$A$-bimodule and $N$ an
  $\O$-$B$-bimodule. The following statements are equivalent:
  \begin{enumerate}
  \item The pairs $(A,M)$ and $(B,N)$ are quasi-isomorphic, that is weakly equivalent in $\Bimod*{\O}$.
  \item The pairs $(A,M)$ and $(B,N)$ have gauge $\O_{\infty}$-isomorphic
    minimal models.
  \end{enumerate}
\end{proposition}

\subsection{Spaces of bimodules}\label{subsec:spaces_of_bimodules}

\Cref{rezk_theorem} also admits an analogue for bimodules over a fixed algebra.

\begin{proposition}\label{Grothendieck_fibration_homotopy_fibre}
  Let $\O$ be a DG operad and $A$ an $\O$-algebra. Assume that $A$ is cofibrant
  in $\Alg{\O}$ or that $\O$ is excellent. Then, the Grothendieck bifibration~\eqref{Grothendieck_fibration}
  \[
    \Phi\colon\Bimod*{\O}\longrightarrow\Alg{\O},\qquad (A,M)\longmapsto A,
  \]
  induces a map
  \[
    |w\Phi|\colon|w\Bimod*{\O}|\longrightarrow|w\Alg{\O}|
  \]
  whose homotopy fibre above $A$ is
  \[|w\Bimod*{\O,A}|.\]
\end{proposition}

\begin{proof}
  We first prove the statement for $\O$ excellent. The objects of the comma category $\commacat{A}{\Phi}$ of the functor $\Phi$ in~\eqref{Grothendieck_fibration} are pairs $(A\to B,M)$ where the first coordinate is an $\O$-algebra morphism and the second coordinate is a $B$-bimodule. A morphism $(A\to B,M)\to(A\to B',M')$ in $\commacat{A}{\Phi}$ is given by an $\O$-algebra morphism $B\to B'$ such that the triangle
  \[\begin{tikzcd}[column sep = small]
      &A\ar[dl]\ar[dr]&\\
      B\ar[rr]&&B'
    \end{tikzcd}\]
  commutes
  and a $B$-bimodule morphism $M\to M'$; here, for simplicity, we omit from the notation the restriction of scalars along an algebra morphism~\eqref{bimodule_same_operad_quillen_adjunction}. The adjunction
  \[\begin{tikzcd}[row sep = 0ex,
        /tikz/column 1/.append style={anchor=base east},
        /tikz/column 2/.append style={anchor=base west}]
      \Bimod*{\O,A}=\Phi^{-1}(A)\ar[r,shift left]&\commacat{A}{\Phi},\ar[l,shift left]\\
      (A,M)\ar[r,mapsto]&(\id*[A]\colon A\to A,M),\\
      (A,N)&(A\to B,N),\ar[l,mapsto]
    \end{tikzcd}\]
  restricts to an adjoint pair
  \[\begin{tikzcd}[row sep = 0ex,
        /tikz/column 1/.append style={anchor=base east},
        /tikz/column 2/.append style={anchor=base west}]
      w\Bimod*{\O,A},\ar[r,shift left]&A\downarrow w\Phi\ar[l,shift left]
    \end{tikzcd}\]
  that induces weak equivalences on nerves.

  A weak equivalence $f\colon A\to B$ in $\Alg{\O}$ gives rise to a commutative square
  \[\begin{tikzcd}
      \Bimod*{\O,B}\ar[r]\ar[d,"f^*"']&\commacat{B}{\Phi}\ar[d,"f^*"]\\
      \Bimod*{\O,A}\ar[r]&\commacat{A}{\Phi}
    \end{tikzcd}\]
  that also restricts to a commutative square
  \[\begin{tikzcd}
      w\Bimod*{\O,B}\ar[r]\ar[d,"f^*"']&\commacat{B}{w\Phi}\ar[d,"f^*"]\\
      w\Bimod*{\O,A}\ar[r]&\commacat{A}{w\Phi}.
    \end{tikzcd}\]
  The leftmost functor $f^*\colon \Bimod*{\O,B}\to \Bimod*{\O,A}$ is a right Quillen
  equivalence (\Cref{bimodule_same_operad_quillen_equivalence_excellent}).
  Therefore, taking nerves on the last commutative square yields weak
  equivalences. The statement then follows from Quillen's Theorem~B.

  We now indicate how to modify the previous argument in case $\O$ is not
  excellent but $A$ is cofibrant in $\Alg{\O}$. Let $\Alg{\O}_c$ be the category
  of cofibrant $\O$-algebras and let $\Bimod*{\O}_{u}$ be the category of all
  $\O$-bimodules over cofibrant $\O$-algebras. The functor
  \eqref{Grothendieck_fibration} restricts to a functor
  \[\Phi_c\colon\Bimod*{\O}_{u}\longrightarrow\Alg{\O}_c.\]
  % see~\cite[Theorem~2.3]{stanculescu_2012_bifibrations_weak_factorisation}.
  The canonical inclusions induce vertical weak equivalences in the following
  commutative diagram, see~\cite[Lemma~4.2.4]{Rez96}:
  \[\begin{tikzcd}
      {|w\Bimod*{\O}_{u}|}\ar[r,"|w\Phi_c|"]\ar[d]&{|w\Alg{\O}_c|}\ar[d]\\
      {|w\Bimod*{\O}|}\ar[r,"|w\Phi|"]&{|w\Alg{\O}|}
    \end{tikzcd}\]
  Hence, we can compute the homotopy fibre of the top horizontal map instead. The same argument as above then goes through, replacing $\Phi$ with $\Phi_c$ and the use of \Cref{bimodule_same_operad_quillen_equivalence_excellent} with \Cref{bimodule_same_operad_quillen_equivalence}\eqref{between_cofibrant_algebras}.
\end{proof}

\begin{proposition}\label{homotopy_monomorphism_bimodules}
  Let $A$ be a unital DG algebra. 
  The full inclusion functor $\Bimod*{\uAss,A}\hookrightarrow\Bimod*{\Ass,A}$ from the category of unital $A$-bimodules to the category of possibly non-unital $A$-bimodules induces a homotopy monomorphism
  \[|w\Bimod*{\uAss,A}|\longrightarrow|w\Bimod*{\Ass,A}|.\]
\end{proposition}

This follows easily from \Cref{homotopy_monomorphism_algebras,homotopy_monomorphism_algebras_bimodules,Grothendieck_fibration_homotopy_fibre,example_excellent_operads} by a homotopy long exact sequence argument.

\begin{corollary}\label{quasi-iso_unit_bimodules}
  Let $A$ be a unital DG algebra and $M,N$ unital $A$-bimodules. Then
  $M$ and $N$ are quasi-isomorphic in $\Bimod*{\uAss,A}$ if and only if they are quasi-isomorphic in $\Bimod*{\Ass,A}$.
\end{corollary}

\begin{remark}
  \Cref{quasi-iso_unit_bimodules} is equivalent to saying that, given a unital DG algebra $A$ and unital $A$-bimodules $M,N$, if there exist quasi-isomorphisms of $A$-bimodules $M\leftarrow P\rightarrow N$ with $P$ possibly non-unital, then we can find quasi-isomorphisms $M\leftarrow P'\rightarrow N$ with $P'$ unital.
\end{remark}

\begin{theorem}\label{total_homotopy_fibre_of_classification_spaces}
  Let $\O$ be an operad and $A$ an $\O$-algebra. Suppose that $A$ is cofibrant
  in $\Alg{\O}$ or that $\O$ is excellent. The forgetful functor
  \[\Bimod*{\O,A}\longrightarrow\dgMod*{\kk}\]
  that sends an $\O$-$A$-bimodule to its underlying DG vector space induces a map
  \[|w\Bimod*{\O,A}|\longrightarrow|w\dgMod*{\kk}|\]
  whose homotopy fibre at a DG vector space $W$ is the (homotopy) fibre of the
  map (Kan fibration) between mapping spaces
  \[\Map{\O}{\E{A,W}}\longrightarrow\Map{\O}{\E{A}}\]
  induced by \eqref{linear_endomorphism_operad}, taken above the $\O$-algebra structure map $\O\to\E{A}$.
\end{theorem}

\begin{proof}
  It suffices to notice that, by \Cref{Grothendieck_fibration_homotopy_fibre}, the first map is the homotopy fibre of the vertical maps in \Cref{rezk_comparison}, and the second map is the homotopy fibre of the horizontal maps in \Cref{rezk_comparison} by that very same result.
\end{proof}

\Cref{total_homotopy_fibre_of_classification_spaces} motivates the following
definition.

\begin{definition}\label{space_of_bimodule_structures}
  Let $\O$ be a DG operad, $A$ an $\O$-algebra and $W$ a DG vector space. The
  \emph{space of $\O$-$A$-bimodule structures on $W$}, denoted by
  $\Str{\O}{A}{W}$, is the fibre of the fibration
  \[\Map{\O}{\E{A,W}}\longrightarrow\Map{\O}{\E{A}}\]
  at the $\O$-algebra structure map of $A$.
\end{definition}

\begin{remark}
  \label{rmk:Str}
  We make the following observations:
  \begin{enumerate}
  \item Points in $\Str{\O}{A}{W}$ are $\O_{\infty}$-$A$-bimodule structures on
    $W$. In particular, $\O$-$A$-bimodule structures on $W$ are special points
    in $\Str{\O}{A}{W}$.
  \item By \Cref{total_homotopy_fibre_of_classification_spaces}, if the algebra $A$ is
    cofibrant in $\Alg{\O}$ or the operad $\O$ is excellent, then $\Str{\O}{A}{W}$ is also
    the homotopy fibre of the map
    \[|w\Bimod*{\O,A}|\longrightarrow|w\dgMod*{\kk}|\] induced by the forgetful
    functor at $W$.
  \item Given an $\O$-algebra $A$ and an $\O$-$A$-bimodule $M$, a minimal model
    of the pair $(A,M)$ in the sense of \Cref{def:minimal_model_simultaneous}
    can also be regarded as a point in $\Str{\O_{\infty}}{A'}{\dgH{M}}$, where
    $A'$ is a minimal model of $A$. Recall that $A'$ is an
    $\O_{\infty}$-algebra with underlying minimal DG vector space $\dgH{A}$.
  \end{enumerate}
\end{remark}

The following result follows from \Cref{homotopy_monomorphism_bimodules,total_homotopy_fibre_of_classification_spaces,example_excellent_operads} and \cite[Lemma 6.3]{Mur16b}.

\begin{proposition}
  Given a DG-vector space $W$, the map
  \[\Str{\uAss}{A}{W}\longrightarrow\Str{\Ass}{A}{W}\]
  induced by the canonical operad morphism $\uAss\to\Ass$ is a homotopy monomorphism. 
\end{proposition}

\begin{definition}
  We make the following definitions:
  \begin{enumerate}
  \item Let $A$ be an $\O_{\infty}$-algebra and $M$ and $N$ two
    $\O_{\infty}$-$A$-bimodules with the same underlying DG vector space $W$. A
    \emph{gauge
      $\O_{\infty}$-isomorphism} or \emph{$\O_{\infty}$-isomorphism with identity linear part} $M\leadsto N$ is a path between the
    corresponding points in $\Str{\O_{\infty}}{A}{W}$, that is a homotopy
    between the corresponding structure maps $\O_{\infty}\to\E{A,W}$ which,
    composed with $\E{A,W}\to\E{A}$ in \eqref{linear_endomorphism_operad}, is
    the trivial homotopy (=the identity gauge $\O_\infty$-isomor\-phism) on the
    $\O_\infty$-algebra structure map of $A$.
  \item The \emph{identity gauge $\O_{\infty}$-isomorphism} $M\leadsto M$ on a
    given $\O_{\infty}$-$A$-bimodule $M$ with underlying DG vector space $W$ is
    the constant path at the corresponding point in $\Str{\O_{\infty}}{A}{W}$.
  \end{enumerate}
  Composition of such gauge $\O_{\infty}$-isomorphisms is defined by
  concatenation of paths in $\Str{\O_{\infty}}{A}{W}$.
\end{definition}

By \Cref{minimal_models_quasi_isomorphic_algebras-bimodules}, quasi-isomorphic pairs of an $\O$-algebra
and a compatible $\O$-bimodule can be distinguished by their minimal models.

\begin{proposition}\label{minimal_models_quasi_isomorphic_bimodules}
  Let $\O$ be a DG operad and $A$ an $\O$-algebra. Suppose that $A$ is cofibrant
  or that $\O$ is excellent. Let $M$ be an $\O$-$A$-bimodule and $N$ an
  $\O$-$A$-bimodule. The following statements are equivalent:
  \begin{enumerate}
  \item\label{it:zig-zag_modules} The $\O$-$A$-bimodules $M$ and $N$
    quasi-isomorphic, that is weakly
    equivalent in $\Bimod*{\O,A}$.
  \item\label{it:gauge_isomorphism_modules} The $\O$-$A$-bimodules $M$ and $N$
    have gauge $\O_{\infty}$-isomorphic minimal models compatible with a given
    minimal model of $A$.
  \end{enumerate}
\end{proposition}

\begin{proof}
  \eqref{it:zig-zag_modules}$\Rightarrow$\eqref{it:gauge_isomorphism_modules}.
  As in the proof of \Cref{minimal_models_quasi_isomorphic_algebras}, we can
  choose weak equivalences $M\twoheadleftarrow P\rightarrow N$ in
  $\Bimod*{\O,A}$, where the first map is also a fibration. Choose a cocycle
  selection map $\dgH{M}\to M$, lift it to a map $\dgH{M}\to P$ and then compose
  it with $P\to N$ in order to obtain a quasi-isomorphism $\dgH{M}\to N$:
  \[
    \begin{tikzcd}
      &P\dar[two heads]\rar&N\\
      \H{M}\rar[swap]{i}\urar[dotted]&M
    \end{tikzcd}
  \]
  We also choose a cocycle selection map $\dgH{A}\to A$. We take minimal models of $(A,M)$, $(A,P)$ and $(A,N)$ in the sense of \Cref{def:minimal_model_simultaneous} associated to these choices. They all have the same underlying minimal model of $A$ since we are taking just one cocycle selection map $\dgH{A}\to A$. If we denote this minimal model by $A'$, the minimal models of the three previous pairs are points in $\Str{\O_{\infty}}{A'}{\dgH{M}}$. By \Cref{rezk_comparison,Grothendieck_fibration_homotopy_fibre}, the homotopy fibre of the map
  \[|w\Bimod*{\O,A}|\longrightarrow|w\dgMod*{\kk}|\]
  induced by the forgetful functor at $\dgH{M}$ is
  $\Str{\O_{\infty}}{A'}{\dgH{M}}$. Therefore, the quasi-isomorphisms $M\twoheadleftarrow P\rightarrow N$ in $\Bimod*{\O,A}$ define paths between those points in $\Str{\O_{\infty}}{A'}{\dgH{M}}$ and \eqref{it:gauge_isomorphism_modules} holds by definition.

  \eqref{it:gauge_isomorphism_modules}$\Rightarrow$\eqref{it:zig-zag_modules}.
  There is an equality $\dgH{M}=\dgH{N}$ and we have cocycle selection maps
  ${\dgH{M}\to M}$, ${\dgH{M}\to N}$ and ${\dgH{A}\to A}$ such that, if $A'$
  denotes the underlying minimal model of $A$, the corresponding points in
  $\Str{\O_{\infty}}{A'}{\dgH{M}}$ are connected by a path. Therefore, the
  homotopy fibre computation in the previous paragraph shows that $M$ and $N$
  are also connected by a path in $|w\Bimod*{\O,A}|$, that is they are
  quasi-isomorphic as $\O$-$A$-bimodules.
\end{proof}

%%%%%

\section{Cohomology of graded operads with multiplication}\label{sec:operads_cohomology}

In this section we recall the construction of the Gerstenhaber algebra associated to a
graded operad with multiplication. We also describe the additional structure
that can be extracted from this construction in the presence of an operadic
ideal. Finally, we apply this construction to the (shifted) linear endomorphism operad
associated to a pair consisting of a graded algebra and a compatible graded
bimodule in order to obtain a variant of Hochschild cohomology for
algebra-bimodule pairs.

\subsection{Operadic suspension}\label{subsec:associative_operad}

The \emph{operadic suspension} is an automorphism of the category of DG operads
\[\os\colon\Operads\longrightarrow\Operads\]
whose action on objects is given by
\begin{equation}\label{operadic_suspension}
	\os\O<n>\coloneqq\shift{\O<n>}[1-n],
\end{equation}
where
\[\shift{}\colon\dgMod*{\kk}\longrightarrow\dgMod*{\kk}\]
is the usual \emph{suspension} or \emph{shift} automorphism of the category of
DG vector spaces:
\begin{equation}\label{dg_shift}
	\shift{V}[1]^n=V^{n+1},\quad n\in\ZZ;\qquad d_{\shift{V}}=-d_V.
\end{equation}
As usual, the $n$-fold iteration of $\shift{}$ is denoted by $\shift{}[n]$,
$n\in\ZZ$. The operad structure of $\os\O$ coincides with that of $\O$ twisted
by signs. We adopt the sign conventions in~\cite[Definition~2.4]{Mur16}. Abusing notation, given an operadic ideal $\I\subset\O$, we will also denote by $\os\I\subset\os\O$ the corresponding operadic ideal defined arity-wise by the same formula as above.

\begin{remark}
  Endomorphism DG operads establish a connection between both suspensions,
  \begin{equation}\label{shifts_endomorphism}
    \os\E{V}\cong\E{\shift{V}},
  \end{equation}
  see \cite[Remark 2.5]{Mur16}. By means of the canonical
  inclusion~\eqref{linear_endomorphism_inclusion}, this isomorphism restricts to
  linear endomorphism operads,
  \begin{equation}\label{shifts_linear_endomorphism}
    \os\E{V,W}\cong\E{\shift{V},\shift{W}}.
  \end{equation}
\end{remark}

\subsection{Brace operations}\label{sec:brace_algebras}

We start by recalling the definition of the brace operations associated to a
graded operad~\cite{Kad88,Get93,GV95}.

\begin{definition}\label{operad_complex}
  The \emph{operad (cochain) complex} $\OC{\O}$ of a graded operad $\O$ is given
  by
  \begin{equation*}
    \OC[p][q]{\O}\coloneqq\O<p>^{p+q-1}, \qquad p\geq 0,\quad q\in\ZZ.
  \end{equation*}
\end{definition}

\begin{remark}
  The complex introduced in \Cref{operad_complex} is a plain bigraded vector
  space concentrated on the right half-plane. We need extra structure on the
  operad (namely a multiplication) in order to endow it with a differential.
  Nevertheless, in order to avoid awkward language, we shall abuse the
  terminology.
\end{remark}

\begin{notation}
  Let $\O$ be a graded operad. If $x\in \OC[p][q]{\O}$, its \emph{horizontal
    degree} is $|x|_h\coloneqq p$, its \emph{vertical degree} is $|x|_v\coloneqq q$, and
  its \emph{total degree} is
  \[
    |x|\coloneqq|x|_h+|x|_v= p+q;
  \]
  its
  \emph{shifted total degree} is $|x|-1$, and its \emph{bidegree} is $(p,q)$.
\end{notation}

\begin{example}
  \label{ex:OCosEV}
  Let $\O$ be a graded operad. Then,
  \[
    \OC[p][q]{\os\O}=\os\O<p>^{p+q-1}=\O<p>[1-p]^{p+q-1}=\O<p>^q, \qquad p\geq
    0,\quad q\in\ZZ.
  \]
  In particular, for a graded vector space $V$ and its endomorphism
  operad $\O=\E{V}$ we have
  \[
    \OC[p][q]{\os\E{V}}=\E{V}<p>^q=\hom{V^{\otimes p}}{V}[q], \qquad p\geq
    0,\quad q\in\ZZ.
  \]
  The perhaps unexpected bigrading in \Cref{operad_complex} is chosen with this
  example in mind, since our preferred sign conventions for the Hochschild
  cochain complex of a graded algebra arise from the sign involved
  in the definition of the infinitesimal compositions of the suspended endomorphism
  operad $\os\E{V}$.
\end{example}

\begin{definition}
  Let $\O$ be a graded operad. The \emph{brace operations} or \emph{braces} are
  the operations
  \begin{align*}
    \OC[p_0][q_0]{\O}\otimes\OC[p_1][q_1]{\O}\otimes\cdots\otimes \OC[p_n][q_n]{\O} & \longrightarrow \OC[p_0+p_1+\cdots+p_n-n][q_0+q_1+\cdots+q_n]{\O}, \\
    x_0\otimes x_1\otimes\cdots\otimes x_n & \longmapsto x_0\{x_1,\dots,x_n\},
  \end{align*}
  defined by
  \[x_0\{x_1,\dots,x_n\}\coloneqq\hspace{-20pt}\sum_{1\leq i_1<\cdots<i_n\leq
      p_0}\hspace{-20pt}(\cdots((x_0\circ_{i_1}x_1)\circ_{i_2+p_1-1}x_2)\cdots)\circ_{i_n+p_1+\cdots+p_{n-1}-(n-1)}x_n,\]
  and $x_0\{x_1,\dots,x_n\}=0$ if $n>p_0$. The brace operations satisfy the \emph{brace
    relation}
  \begin{multline*}
    x\{y_1,\dots,y_p\}\{z_1,\dots,z_q\}\\
    =\textstyle\sum_{0\leq i_1\leq j_1\leq\cdots\leq i_p\leq j_p\leq q}(-1)^\maltese x\{z_1,\dots,z_{i_1},y_1\{z_{i_1+1},\dots,z_{j_1}\},\dots\\
    \dots,y_p\{z_{i_p+1},\dots,z_{j_p}\}, z_{j_p+1},\dots,z_q\},
  \end{multline*}
  where the sign is given by the Koszul sign rule for the shifted total degree:
  \[\textstyle\maltese=\sum_{s=1}^p\sum_{t=1}^{i_s}(|y_s|-1)(|z_t|-1).\]
  The resulting structure on the operad complex $\OC{\O}$ is called a
  \emph{brace algebra}.
\end{definition}

The brace operations on the operad complex encode several perhaps more familiar
algebraic structures, as we now recall.

\begin{definition}
  The \emph{Gerstenhaber Lie bracket} is defined by
  \begin{align*}
    [-,-]\colon\OC[p][q]{\O}\otimes\OC[s][t]{\O}&\longrightarrow\OC[p+s-1][q+t]{\O}\\
    x\otimes y&\longmapsto x\{y\}-(-1)^{(p+q-1)(s+t-1)}y\{x\}.
  \end{align*}
  The brace relations imply that Gerstenhaber Lie bracket is graded
  anti-commutative and satisfies the graded Jacobi identity, both with respect
  to the shifted total degree:
  \begin{align*}
    [x,y]&=-(-1)^{(|x|-1)(|y|-1)}[y,x]\\
    [x,[y,z]]&=[[x,y],z]+(-1)^{(|x|-1)(|y|-1)}[y,[x,z]].
  \end{align*}
  Consequently $\shift{\OC{\O}}$ is a graded Lie algebra.
  In $\cchar{\kk} = 2$ we actually have
  \[[x,x]=0,\]
  and in $\cchar{\kk} = 3$
  \[[x,[x,x]]=0\]
  if $|x|$ is even.
\end{definition}

\begin{definition}
  Let $\O$ be a graded operad and $\I\subset\O$ an operadic ideal. We define the
  \emph{operadic ideal complex} $\IC{\I}$ as\footnote{Notice the shift in arity
    and degree with respect to the definition of $\OC{\O}$.}
  \begin{equation*}
    \IC[p][q]{\I}=\I<p+1>^{p+q},\qquad p\geq 0,\quad q\in\ZZ.
  \end{equation*}
  Hence,
  \[
    \IC[p][q]{\O}\subset\OC[p+1][q]{\O}\qquad\text{and}\qquad\IC{\I}\subset\shift{\OC{\O}},
  \]
  where the shift denotes the shift in the horizontal direction. Moreover,
  \[
    x_0\{x_1,\dots,x_n\}\in\IC{\I}
  \]
  as long as $x_i\in\IC{\I}$ for some $0\leq i\leq n$. The resulting structure
  on the operadic ideal complex $\IC{\I}$ is called a \emph{brace ideal}.
\end{definition}

\begin{remark}
  Let $\O$ be a graded operad and $\I\subset\O$ an operadic ideal. We observe
  that $\IC{\I}\subset\shift{\OC{\O}}$ is a graded Lie ideal and that there is a
  canonical identification
  \[
    \shift{\OC{\O}}/\IC{\I}=\shift{\OC{\O/\I}}.
  \]
\end{remark}

\begin{definition}
  \label{def:associative_operadic_ideal}
  An operadic ideal $\I\subset\O$ is \emph{associative} if
  \[
    x_0\{x_1,\dots,x_n\}=0
  \]
  whenever $x_i,x_j\in\OC{\I}$ for some $1\leq i<j\leq n$ (notice that the lower
  bound is $1$, not $0$).
\end{definition}

\begin{example}
  Given graded vector spaces $V$ and $W$, the kernel $\E*{V,W}$ of the
  projection $\E{V,W}\twoheadrightarrow\E{V}$ in
  \eqref{linear_endomorphism_operad} is an associative operadic ideal. Indeed,
  \begin{equation}\label{aritywise_linear_endomorphism_ideal}
    \E*{V,W}<n>=\bigoplus_{p+1+q=n}\hom{V^{\otimes p}\otimes W\otimes V^{\otimes q}}{W}.
  \end{equation}
  Therefore, given $x\in\E{V}$, $y\in \hom{V^{\otimes p}\otimes W\otimes
    V^{\otimes q}}{W}$, and $z\in\E*{V,W}$, we have
  \begin{align*}
    x\circ_j y & =0,\;\forall j; &
                                   y\circ_{p+1} x &=0;&
                                                    y\circ_jz & =0,\;\forall j\neq p+1; &
                                                                                        y\{z\}&=y\circ_{p+1}z.
  \end{align*}
\end{example}

\begin{lemma}\label{brace_associative}
  Let $\O$ be a graded operad and $\I\subset\O$ an associative operadic ideal.
  The following statements hold:
  \begin{enumerate}
  \item\label{it:circle_product-ass} The \emph{circle product}
    \[
      \IC[p][q]{\I}\otimes\IC[s][t]{\I}\longrightarrow\IC[p+s][q+t]{\I},\qquad
      x\circ y\coloneqq x\{y\}
    \]
    is associative; that is, given $x,y,z\in\IC{\I}$,
    \[(x\circ y)\circ z=x\circ(y\circ z).\]
  \item\label{circle_product-ass-consequence} Endowed with the circle product,
    $\IC{\I}$ is a (non-unital) graded associative algebra, and the Gerstenhaber
    Lie bracket on $\IC{\I}$ is the commutator of the circle product.
  \item\label{circle_product-ass-Gerstenhaber_rel} Given $a\in\OC{\O}$ and
    $x,y\in\IC{\I}$, the following \emph{Gerstenhaber relation} holds:
    \[[a,x\circ y]=[a,x]\circ y+(-1)^{(|a|-1)|x|}x\circ[a,y].\]
  \end{enumerate}
\end{lemma}

\begin{proof}
  Throughout the proof we use the total degree in $\IC{\I}$, which is equal to
  the shifted total degree in $\OC{\O}$.

  \eqref{it:circle_product-ass} and \eqref{circle_product-ass-consequence} By
  the brace relation,
  \begin{align*}
    x\{y\}\{z\}=x\{y,z\}+x\{y\{z\}\}+(-1)^{|y||z|}x\{z,y\}.
  \end{align*}
  By the associativity hypothesis $x\{y,z\}=x\{z,y\}=0$, therefore
  \[(x\circ y)=x\{y\}\{z\}=x\{y\{z\}\}=x\circ(y\circ z).\]

  \eqref{circle_product-ass-Gerstenhaber_rel} By definition of the Gerstenhaber
  Lie bracket,
  \begin{align*}
    [a,x\circ y] & = a\{x\{y\}\}-(-1)^{(|a|-1)(|x|+|y|)}x\{y\}\{a\}, \\
    [a,x]\circ y & = a\{x\}\{y\}-(-1)^{(|a|-1)|x|}x\{a\}\{y\},     \\
    x\circ [a,y] & = x\{a\{y\}\}-(-1)^{(|a|-1)|y|}x\{y\{a\}\}.
  \end{align*}
  Using the brace relation and the associativity hypothesis we obtain:
  \begin{align*}
    x\{y\}\{a\} & = x\{y,a\}+x\{y\{a\}\}+(-1)^{(|a|-1)|y|}x\{a,y\}, \\
    a\{x\}\{y\} & = a\{x,y\}+a\{x\{y\}\}+(-1)^{|x||y|}a\{y,x\} = a\{x\{y\}\}, \\
    x\{a\}\{y\} & = x\{a,y\}+x\{a\{y\}\}+(-1)^{(|a|-1)|y|}x\{y,a\}.
  \end{align*}
  The Gerstenhaber relation in the statement now follows easily.
\end{proof}

\subsection{Graded operads with multiplication}\label{sec:operads_with_multiplication}

We recall
the construction of the Gerstenhaber algebra associated to
an operad with multiplication. We adopt
the sign conventions in \cite[Section~1]{Mur20}.

\begin{definition}[{\cite[Section~1.2]{GV95}}]
  \label{operad_multiplication}
  Let $\O$ be a graded operad. An \emph{(associative) multiplication} is an
  element $m_2\in\O<2>^{1}$ that satisfies
  \[m_2\{m_2\}=0.\] This is equivalent to the datum of an operad morphism
  $\os\Ass\to\O$, where $\Ass$ is the associative operad.
\end{definition}

\begin{example}
  \label{multiplication_endomorphism_operad}
  If $A$ is a graded algebra, that is an $\Ass$-algebra, the
  suspended endomorphism operad $\os\E{A}$ has a multiplication
  \[m_2^A\in\os\E{A}<2>^1=\shift{\E{A}<2>}[-1]^1=\E{A}(2)^0=\hom{A\otimes A}{A}[0]\]
  given by the binary product in $A$:
  \[m_2^A\colon A\otimes A\longrightarrow A,\qquad m_2^A(x\otimes y)\coloneqq xy.\]
  This multiplication corresponds to the operadic suspension $\os\Ass\to\os\E{A}$ of the $\Ass$-algebra structure map $\Ass\to\E{A}$.
\end{example}

\begin{example}
  \label{multiplication_linear_endomorphism_operad}
  Let $A$ be a graded algebra and $M$ a graded $A$-bimodule. The suspended linear endomorphism operad $\os\E{A,M}$ has a multiplication
  \begin{align*}
    m_2^A+m_{0,1}^{M}+m_{1,0}^{M}&\in\os\E{A,M}<2>^1\\&=\shift{\E{A,M}<2>}[-1]^1\\&=\E{A,M}<2>^0\\
    &=\hom{A\otimes A}{A}[0]\oplus \hom{M\otimes A}{M}[0]\oplus \hom{A\otimes M}{M}[0],
  \end{align*}
  where
  \[m_{0,1}^{M}\colon M\otimes A\longrightarrow M,\qquad
  m_{1,0}^{M}\colon A\otimes M\longrightarrow M,\]
  are the right and left $A$-module structure maps of $M$, respectively.
  This multiplication corresponds to the operadic suspension $\os\Ass\to\os\E{A,M}$ of the $\Ass$-$A$-bimodule structure map $\Ass\to\E{A,M}$.
\end{example}

\begin{definition}[{\cite[Section~1.2, Proposition~2(3)]{GV95}}]
  \label{multiplication_differential_cup-product}
  Let $\O$ be a graded operad with multiplication $m_2\in\O<2>^1$.
  \begin{enumerate}
  \item 
    The
    \emph{differential} $d$ on the operad complex $\OC{\O}$ is given by 
    \[d\colon \OC[p][q]{\O}\longrightarrow \OC[p+1][q]{\O},\qquad
      d(x)\coloneqq[m_2,x].\] This differential has bidegree $(1,0)$ and,
    consequently, the cohomology
    \[
      \OH{\O}\coloneqq\dgH*{\OC{\O}}
    \]
    of the operad complex $\OC{\O}=(\OC{\O},d)$ is bigraded.
  \item The \emph{product} or \emph{cup product} is given by
    \[\OC[p][q]{\O}\otimes \OC[s][t]{\O}\stackrel{\cdot}{\longrightarrow}\OC[p+s][q+t]{\O},\qquad x\cdot y\coloneqq(-1)^{p+q-1}m_2\{x,y\}.\]
  \end{enumerate}
  Endowed with the above structure, the operadic complex $\OC{\O}$ is a
  (non-unital) DG associative algebra. Moreover, $\shift{\OC{\O}}$ is a DG Lie
  algebra with the Gerstenhaber Lie bracket.
\end{definition}

If $\O$ is a graded operad with multiplication, with respect to the total
degree, the cohomology $\OH{\O}$ is a
graded-commutative (non-unital) associative algebra with the induced product
\begin{align*}
\cdot\;\colon\OH[p][q]{\O}\otimes\OH[s][t]{\O}&\longrightarrow\OH[p+s][q+t]{\O},
\end{align*}
that is
\begin{align*}
  (x\cdot y)\cdot z & =x\cdot(y\cdot z), \\
  x\cdot y &= (-1)^{|x||y|}y\cdot x.
\end{align*}
Moreover, $\OH{\O}$ is a shifted graded Lie algebra with the induced
Gerstenhaber Lie bracket
\begin{align*}
[-,-]\colon\OH[p][q]{\O}\otimes\OH[s][t]{\O}&\longrightarrow\OH[p+s-1][q+t]{\O},
\end{align*}
that is,
\begin{align*}
  [x,y] &= -(-1)^{(|x|-1)(|y|-1)}[y,x], \\
    [x,[y,z]] &= [[x,y],z]+(-1)^{(|x|-1)(|y|-1)}[y,[x,z]],\\
  [x,x] & =0,& \cchar{\kk}=2, \\
  [x,[x,x]] & =0,& \cchar{\kk}=3\text{ and } |x|\text{ even}.
\end{align*}
Finally, the associative product is compatible with the Gerstenhaber Lie
bracket in the sense that the \emph{Gerstenhaber relation} is satisfied:
\begin{align*}
  [x,y\cdot z] & =[x,y]\cdot z+(-1)^{(|x|-1)|y|}y\cdot [x,z].
\end{align*}

\begin{definition}\label{def:Gerstenhaber_square}
  Let $\O$ be a graded operad with multiplication and $\cchar{\kk} = 2$. The quadratic map
  \[\Sq\colon\OC[p][q]{\O}\longrightarrow\OC[2p-1][2q]{\O},\qquad\Sq[x]=x\{x\},\]
  induces a quadratic map in cohomology called \emph{Gerstenhaber square},
  \[\Sq\colon\OH[p][q]{\O}\longrightarrow\OH[2p-1][2q]{\O}.\]
  This map satisfies the following relations:
  \begin{align*}
    \Sq[x+y]&=\Sq[x]+\Sq[y]+[x,y],\\
    \Sq[x\cdot y]&=\Sq[x]\cdot y^2+x\cdot[x,y]\cdot y+x^2\cdot\Sq[y],\\
    [\Sq[x],y]&=[x,[x,y]].
  \end{align*}
  In $\cchar{\kk}\neq 2$ we define the Gerstenhaber square in even total degree as
  \[\Sq\colon\OH[p][q]{\O}\longrightarrow\OH[2p-1][2q]{\O},\qquad\Sq[x]=\tfrac{1}{2}[x,x],\qquad p+q\text{ even}.\]
  It satisfies the previous three equations whenever they make sense. Both definitions agree at the cochain level.
\end{definition}

\begin{remark}
  Let $\O$ be a graded operad with multiplication. The structure on $\OH{\O}$
  described above is known as \emph{Gerstenhaber algebra}.
\end{remark}

\begin{definition}
  \label{def:op-Ext}
  Let $\O$ be a graded operad with multiplication and $\I\subset\O$ an operadic
  ideal. The operadic ideal complex $\shift{\IC{\I}}[-1]\subset\OC{\O}$ is an
  associative DG ideal (with respect to the DG algebra structure from
  \Cref{multiplication_differential_cup-product}). We denote its cohomology by
  \[\IH{\I}\coloneqq\dgH*{\IC{\I}}.\]
  The graded operad $\O/\I$ inherits the multiplication from $\O$ and, moreover,
  \[
    \OC{\O}/(\shift{\IC{\I}}[-1])=\OC{\O/\I}
  \]
  as DG associative algebras and
  \[
    \shift{\OC{\O}}/\IC{\I}=\shift{\OC{\O/\I}}
  \]
  as DG Lie algebras.
\end{definition}

\begin{remark}
If $\O$ is an operad with multiplication and $\I\subset\O$ is an operadic ideal, the short exact sequence of complexes
\begin{equation}
  \label{operad_multiplication_exact_sequence}
  \shift{\IC{\I}}[-1]\stackrel{i}\hookrightarrow\OC{\O}\stackrel{p}\twoheadrightarrow\OC{\O/\I}
\end{equation}
induces long exact sequences of vector spaces ($q\in\ZZ$)
\begin{equation}
  \label{operad_multiplication_long_exact_sequence}
  \cdots\to\IH[p-1][q]{\I}\xrightarrow{i_*}\OH[p][q]{\O}\xrightarrow{p_*}\OH[p][q]{\O/\I}\xrightarrow{\delta}\IH[p][q]{\I}\to\cdots
\end{equation}
Moreover, the map $p_*\colon\OH{\O}\to\OH{\O/\I}$ is a morphism of Gerstenhaber
algebras since $p$ is induced by the DG operad map $\O\to\O/\I$ that maps the multiplication in $\O$ to that in $\O/\I$.
\end{remark}

We now describe the algebraic structure on the cohomology of the operadic ideal complex $\IH{\I}$
when $\I\subset\O$ is an associative operadic ideal in the graded operad with multiplication $\O$.
Recall that if $A$ is a graded algebra and $M$ is an
$A$-bimodule, then $\shift{M}$ is also an $A$-bimodule with left $A$-module
structure twisted by a sign. More precisely, given $m\in M$ and $a,b\in A$, the
following formula relates the $A$-bimodule structure of $\shift{M}$ on the left
to the $A$-bimodule structure of $M$ on the right:
\[a\cdot \s m\cdot b\coloneqq(-1)^{|a|}\s( a\cdot m\cdot b),\qquad m\in M,\quad
  a,b\in A,\]
where $\s$ is the suspension operator of (cohomological) degree $|\s|=-1$, i.e.~$\s\colon M\to\shift{M}$ is a natural degree $-1$ morphism.

\begin{lemma}\label{associative_operadic_ideal_cohomology}
  Let $\O$ be a graded operad with multiplication and $\I\subset\O$ an associative operadic ideal.
  The following statements hold:
  \begin{enumerate}
    \item\label{square-zero} The DG associative ideal
      $\shift{\IC{\I}}[-1]\subset\OC{\O}$ is square-zero,
      \[
        x\cdot y=0,\qquad x,y\in\shift{\IC{\I}}[-1].
      \]
    \item\label{bimodule_over_quotient} The shifted operadic ideal complex
      $\shift{\IC{\I}}[-1]$ is an associative DG bimodule over the DG algebra $\OC{\O/\I}$.
    \item The cohomology $\shift{\IH{\I}}[-1]$ of $\shift{\IC{\I}}[-1]$ is an
      associative bimodule over the graded algebra $\OH{\O/\I}$, and hence so is
      the cohomology $\IH{\I}$.
    \item\label{bimodule_morphism} The map
      $i_*\colon\shift{\IH{\I}}[-1]\to\OH{\O}$ induced by the canonical
      inclusion $i\colon\shift{\IC{\I}}[-1]\hookrightarrow\OC{\O}$ is a morphism of $\OH{\O}$-bimodules.
    \item\label{Lie_bimodule} The cohomology $\IH{\I}$ is a graded Lie bimodule
      over the graded Lie algebra $\shift{\OH{\O}}$ and $i_*\colon\IH{\I}\to\shift{\OH{\O}}$ is a morphism of graded Lie $\shift{\OH{\O}}$-modules.
    \item\label{Lie_crossed_module} Given $x,y\in\IH{\I}$, \[[i_*(x),y]=x\circ y-(-1)^{|x||y|}y\circ x=[x,i_*(y)].\]
    \item\label{circle_square} If $\cchar{\kk}=2$ and $x\in\IH{\I}$, then \[i_*(x\circ x)=\Sq[i_*(x)].\]
    \item\label{Lie_and_associative_bimodule_interaction} Given $x\in\OH{\O}$, $y\in\OH{\O/\I}$ and $z\in\IH{\I}$, \begin{align*}[x,y\cdot z]&=[p_*(x),y]\cdot z+(-1)^{(|x|-1)|y|}y\cdot[x,z]\\ [x,z\cdot y]&=[x,z]\cdot y+(-1)^{(|x|-1)(|z|-1)}z\cdot[p_*(x),y].\end{align*}
    \item\label{circle_bimodule} Given $x,y\in\IH{\I}$ and $a\in\OH{\O/\I}$, \begin{align*}
      (a\cdot x)\circ y & = a\cdot(x\circ y),                   &
      (x\cdot a)\circ y & = (-1)^{|a||y|}(x\circ y)\cdot a, \\
      x\circ(a\cdot y)  & = (-1)^{|x||a|}a\cdot(x\circ y),  &
      x\circ (y\cdot a) & = (x\circ y)\cdot a.
    \end{align*}
  \end{enumerate}
\end{lemma}

\begin{proof}
  The square zero property in \eqref{square-zero} is a direct consequence of the definition of associative operadic ideal, and \eqref{bimodule_over_quotient}--\eqref{bimodule_morphism} follow immediately.

  Item \eqref{Lie_bimodule} is obvious and \eqref{Lie_crossed_module} holds at the cochain level by the very definition of the Gerstenhaber bracket and the circle product, and similarly \eqref{circle_square} with the Gerstenhaber square.

  The Gerstenhaber-like relations in \eqref{Lie_and_associative_bimodule_interaction} follows from the same cochain-level equations in $\HCE{A}{M}$ as the original, see \cite[p.~71]{Mur20}.

  Concerning \eqref{circle_bimodule}, the four equations in cohomology are consequences of the following cochain-level formulas, derived from \cite[Lemmas~1.4.1 and~1.4.2]{JKM24}, compare \cite[Theorem~3]{GV95}. Given $a\in \OC{\O}$ and $x,y\in \IC{\I}$,
  \begin{align*}
    (a\cdot x)\{y\}    & = a\cdot (x\{y\})+(-1)^{(|x|+1)|y|}a\{y\}\cdot x,                                                                \\
(x\cdot a)\{y\}    & = x\cdot a\{y\}+(-1)^{|a||y|}x\{y\}\cdot a,                                                 \\
    d(x\{a,y\}) & =d(x)\{a,y\}+(-1)^{|x|}x\{d(a),y\}+(-1)^{|x|+|a|-1}x\{a,d(y)\} \\
                       & \phantom{=}+(-1)^{|x|+(|x|+1)|a|}a\cdot x\{y\}
    +(-1)^{|x|+|a|-1}x\{a\cdot y\}\\
    & \phantom{=}+(-1)^{|x|+|a|}x\{a\}\cdot y,                                                                          \\
    d(x\{y,a\}) & =d(x)\{y,a\}+(-1)^{|x|}x\{d(y),a\}+(-1)^{|x|+|y|}x\{y,d(a)\} \\
                       & \phantom{=}
                       +(-1)^{|x|+(|x|+1)(|y|+1)}y\cdot x\{a\}+(-1)^{|x|+|y|}x\{y\cdot a\}\\
                       & \phantom{=}+(-1)^{|x|+|y|-1}x\{y\}\cdot a.
  \end{align*}
  We also use the brace ideal property and square-zero property in \eqref{square-zero} to derive that 
  \begin{align*}
    a\{y\}\cdot x&=0,&
    x\cdot a\{y\}&=0,&
    x\{a\}\cdot y&=0,&
    y\cdot x\{a\}&=0.\qedhere
  \end{align*}
\end{proof}

\begin{remark}
  We warn the reader that, despite the graded associative algebra $\OH{\O/\I}$ being
  graded-commutative, the graded $\OH{\O/\I}$-bimodule $\IH{\I}$ need not be
  graded-symmetric.
\end{remark}

\subsection{Hochschild cohomology and related variants}\label{sec:hochschild}

We specialise the discussions in previous sections to shifted endomorphism
operads of graded algebras and to shifted linear endomorphism operads of pairs
consisting of a graded algebra and a compatible graded bimodule. In the first
case we recover the classical Hochschild cohomology of graded algebras and in
the second case its variant for pairs discussed in the introduction.

Let $A$ be a graded algebra. Recall that the \emph{bar complex} $\B{A}$ is the
free resolution of the diagonal $A$-bimodule $A$ with the components
\[\B[n]{A}\coloneqq A\otimes A^{\otimes n}\otimes A,\qquad n\geq 0,\]
and differential
\[d\colon\B[n]{A}\longrightarrow\B[n-1]{A},\quad d(a_0\otimes\cdots\otimes
  a_{n+1})\coloneqq\sum_{i=0}^n(-1)^ia_0\otimes \cdots\otimes a_i
  a_{i+1}\otimes\cdots\otimes a_{n+1}.\] The
augmentation
\[
  \varepsilon\colon\B[0]{A}=A\otimes A\longrightarrow A,\qquad a_0\otimes
  a_1\longmapsto a_0a_1
\]
is the multiplication map of $A$.

\begin{definition}\label{cohomologies}
  Let $A$ be a graded algebra, the \emph{Hochschild (cochain) complex} and the
  \emph{Hochschild cohomology} of $A$ are defined as
    \[\HC{A}\coloneqq\OC{\os\E{A}}\cong\hom[A^e]{\B{A}}{A}\]
    and
    \[
      \HH{A}\coloneqq\OH{\os\E{A}}\cong \Ext{A^e}{A}{A},
    \]
    respectively (compare with
    \Cref{ex:OCosEV,multiplication_endomorphism_operad}).
\end{definition}

\begin{remark}
The cochain-level isomorphism in \Cref{cohomologies} (which implies the
isomorphisms at the cohomological level) is the composite
\begin{align*}
    \OC[p][q]{\os\E{A}}&=\os\E{A}<p>^{p+q-1}\\
    &=\shift{\E{A}<p>}[1-p]^{p+q-1}\\ 
    &= \E{A}<p>^q\\ 
    &= \hom[\kk]{A^{\otimes p}}{A}[q]\\
    &\cong \hom[A^e]{A\otimes A^{\otimes p}\otimes A}{A}[q]\\
    &= \hom[A^e]{\B[p]{A}}{A}[q],
\end{align*}
where some of the identifications must involve some signs (depending on the
bidegree) in order to match with the sign conventions in the differential of the
operad complex and that of the bar complex. We do not elaborate more on this
as the precise formulas do not play a role in the sequel.
\end{remark}

\begin{definition}
  \label{def:bimodule_cohomologies}
  Let $A$ be a graded algebra and $M$ a graded $A$-bimodule. The \emph{bimodule
    Hochschild (cochain) complex} and the \emph{bimodule Hochschild cohomology} of the
  pair $(A,M)$ are defined as 
  \[\HCE{A}{M}\coloneqq\OC{\os\E{A,M}}\]
  and
  \[
    \HHE{A}{M}\coloneqq\OH{\os\E{A,M}},
  \]
  respectively, compare with
  \Cref{ex:OCosEV,multiplication_linear_endomorphism_operad}. We also define the
  \emph{bimodule complex} of $M$ to be
    \[\BC{A^e}{M}\coloneqq\IC{\os\E*{A,M}}[\os\E{A,M}]\cong\hom[A^e]{\B{A}\otimes_AM\otimes_A\B{A}}{M}.\]
    Hence, its cohomology is
    \[\BH{A^e}{M}\coloneq\H[\bullet,*]{\BC{A^e}{M}}\cong\Ext{A^e}{M}{M}.\]
\end{definition}

\begin{remark}
  In the setting of \Cref{def:bimodule_cohomologies}, we have
\begin{align*}
    \IC[p][q]{\os\E*{A,M}}&=\os\E*{A,M}<p+1>^{p+q}\\
    &=\shift{\E*{A,M}<p+1>}[-p]^{p+q}\\
    &= \E*{A,M}<p+1>^q\\
    &= \bigoplus_{s+t=p}\hom[\kk]{A^{\otimes s}\otimes M\otimes A^{\otimes t}}{M}[q]\\
    &\cong \bigoplus_{s+t=p} \hom[A^e]{A\otimes A^{\otimes s}\otimes A\otimes_A M\otimes_A\otimes A\otimes A^{\otimes t}\otimes A}{M}[q]\\
    &=\bigoplus_{s+t=p} \hom[A^e]{\B[s]{A}\otimes_AM\otimes_A\B[t]{A}}{M}[q],
\end{align*}
we must also introduce signs in some of the identifications in order to
match the signs of the differentials.
\end{remark}

\begin{remark}
  Let $A$ be a graded algebra and $M$ a graded $A$-bimodule. The results in
  \Cref{sec:brace_algebras,sec:operads_with_multiplication} show the Hochschild
  cohomology $\HH{A}$ and the bimodule Hochschild cohomology $\HHE{A}{M}$ are
  Gerstenhaber algebras. The former Gerstenhaber algebra is well known, and the
  commutative product corresponds to the Yoneda product in $\Ext{A^e}{A}{A}$.
  The latter Gerstenhaber algebra is new, we believe. The circle product in
  $\Ext{A^e}{M}{M}$ also corresponds to the Yoneda product. The graded
  $\HH{A}$-bimodule structure is new, we also think, as well as the graded Lie
  $\shift{\HHE{A}{M}}$-module structure. As in
  \eqref{operad_multiplication_long_exact_sequence}, there are long exact
  sequences of vector spaces ($q\in\ZZ$)
\begin{equation}\label{long_exact_sequence_Hochschild}
    \cdots\rightarrow \Ext[p-1][q]{A^e}{M}{M}\xrightarrow{i_*} \HHE[p][q]{A}{M}\xrightarrow{p_*} \HH[p][q]{A}\xrightarrow{\delta} \Ext[p][q]{A^e}{M}{M}\rightarrow\cdots
\end{equation}
that satisfy several compatibilities with respect to the various algebraic
structures; for example, $p_*$ is a morphism of Gerstenhaber algebras. However,
in this particular setting, the three cohomologies and the long exact sequence
have special features that we now explore. For instance, the three associative
algebras are unital. The unit of $\HH{A}$ and $\HHE{A}{M}$ are induced by the algebra unit
\[
  \id[A]\in A=\hom[\kk]{\kk}{A}[0]=\HC[0][0]{A}\subset\HCE[0][0]{A}{M}.
\]
The unit of $\Ext{A^e}{M}{M}$ is induced by the identity morphism
\[
  \id*[M]\in\hom[A^e]{M}{M}[0]=\BC[0][0]{A^e}{M}.
\]
The morphism $p_*$ preserves units, but the morphism $i_*$ does not.
\end{remark}

Our next task is to compute the connecting morphism $\delta$ in the long exact
sequence \eqref{long_exact_sequence_Hochschild}. We recall the following
standard definition.

\begin{definition}\label{def:square_zero_extension}
  Let $B$ be a DG algebra, $N$ a DG $B$-bimodule and $x\in N^0$ a cocycle,
  $d_N(x)=0$. We define the DG algebra $B\ltimes_x \shift{N}[-1]$ as the homotopy
  fibre (=mapping cocone) of the cochain map
  \[B\longrightarrow N,\qquad b\longmapsto x\cdot b-b\cdot x;\] explicitly, 
  \[(B\ltimes_x \shift{N}[-1])^n\coloneqq B^n\oplus  N^{n-1}\] and, for $b\in B$
  and $n\in N$,
  \[
    d(b+\s[-1]n)\coloneqq d_B(b)-\s[-1](d_N(n)+x\cdot b-b\cdot x);
  \]
  here we use the desuspension operator $\s[-1]$ of (cohomological) degree $|\s[-1]|=1$ for the
  sake of clarity, i.e.~$\s[-1]\colon N\mapsto\shift{N}[-1]$ is natural degree $1$ isomorphism. The product is the usual square-zero extension product,
  \[(b+\s[-1]n)(b'+\s[-1]n')\coloneqq b\cdot b'+b\cdot \s[-1]n'+\s[-1]n\cdot b'.\] In this
  last formula we use the induced $B$-bimodule structure on $\shift{N}[-1]$. We
  leave the reader to check that this defines a DG algebra.
\end{definition}

% Below, we use the unshifted $B$-bimodule structure on $N$.

% \begin{multline*}
%     d(n+b)(n'+b')=(-d(n)+x\cdot b-b\cdot x+d(b))(n'+b')\\
%     =-d(n)\cdot b'+x\cdot b\cdot b'-b\cdot x\cdot b'+(-1)^{|b|+1}d(b)\cdot n'+d(b)\cdot b',\\
%     (n+b)d(n'+b')=(n+b)(-d(n')+x\cdot b'-b'\cdot x+d(b'))\\
%     =n\cdot d(b')-(-1)^{|b|}b\cdot d(n')+(-1)^{|b|}b\cdot x\cdot b'-(-1)^{|b|}b\cdot b'\cdot x+b\cdot d(b'),\\
%     d((n+b)(n'+b'))=d(n\cdot b'+(-1)^{|b|}b\cdot n'+b\cdot b')\\
%     =-d(n\cdot b'+(-1)^{|b|}b\cdot n')+x\cdot b\cdot b'-b\cdot b'\cdot x+d(b\cdot b')\\
%     =-d(n)\cdot b'-(-1)^{|b|-1}n\cdot d(b')-(-1)^{|b|}d(b)\cdot n'-b\cdot d(n')\\+x\cdot b\cdot b'-b\cdot b'\cdot x+d(b)\cdot b'+(-1)^{|b|}b\cdot d(b').
%   \end{multline*}

\begin{proposition}\label{connecting_morphism}
  If $A$ is a graded algebra and $M$ is a graded $A$-bimodule, the bimodule
  Hochschild complex is, as a DG algebra,
  \[\HCE{A}{M}=\HC{A}\ltimes_{\id*[M]}\shift{\BC{A^e}{M}}[-1],\]
  for the degree $0$ cocyle $\id*[M]\in\BC{A^e}{M}$.
  In particular, the morphism $\delta$ in
  \eqref{long_exact_sequence_Hochschild} is given by
  \[\delta\colon\HH{A}\longrightarrow\Ext{A^e}{M}{M},\qquad a\longmapsto
    \id*[M]\cdot a-a\cdot\id*[M].\]
\end{proposition}

\begin{proof}
  The second part is obviously a consequence of the first part, so let us
  concentrate in the latter. The equality holds bidegree-wise by the very
  definition of the linear endomorphism operad in
  \eqref{aritywise_linear_endomorphism_operad},
  \[\HCE[p][q]{A}{M}=\HC[p][q]{A}\ltimes_{\id*[M]}\BC[p-1][q]{A^e}{M}.\]
  The product in $\HCE{A}{M}$ is the square-zero extension product by the form
  of the multiplication in $\os\E{A,M}$, see
  \Cref{multiplication_linear_endomorphism_operad}. Moreover, the canonical inclusion
  \[
    \shift{\BC{A^e}{M}}[-1]\subset\HCE{A}{M}
  \]
  is the inclusion of a subcomplex. Hence, it is only left to check the formula
  for the differential of $\alpha\in\HC{A}$ in $\HCE{A}{M}$. Keeping in mind the
  above inclusion, we have that
    \begin{align*}
        [m_2^A+m_{0,1}^{M}+m_{1,0}^{M},\alpha]&=[m_2^A,\alpha]+[m_{0,1}^{M},\alpha]+[m_{1,0}^{M},\alpha],\\
        [m_{0,1}^{M},\alpha]&=m_{0,1}^{M}\{\alpha\}=m_{0,1}^{M}\{\s[-1]\id*[M],\alpha\}=\s[-1]\id*[M]\cdot\alpha,\\
        [m_{1,0}^{M},\alpha]&=m_{1,0}^{M}\{\alpha\}=m_{1,0}^{M}\{\alpha,\s[-1]\id*[M]\}=(-1)^{|\alpha|-1}\alpha\cdot\s[-1]\id*[M].
    \end{align*}
    Consequently,
    \[d_{\HCE{A}{M}}(\alpha)=[m_2^A+m_{0,1}^{M}+m_{1,0}^{M},\alpha]=d_{\HC{A}}(\alpha)+\s[-1](\id*[M]\cdot
      \alpha-\alpha\cdot\id*[M])\]
    since
    \[
      \alpha\cdot\s[-1]\id*[M]=(-1)^{|\alpha|}\s[-1](\alpha\cdot\id*[M]).
    \]
    This finishes the proof.
\end{proof}

It follows form \Cref{connecting_morphism} that the connecting homomorphism
$\delta$ vanishes when the $\HH{A}$-bimodule $\Ext{A^e}{M}{M}$ is graded-symmetric,
that is the left $\HH{A}$-module structure coincides with the right
$\HH{A}$-module structure up the Koszul sign. We now prove the converse.

\begin{corollary}
  \label{cor:symmetric_bimodule}
    Given a graded algebra $A$ and a graded $A$-bimodule $M$, the
    $\HH{A}$-bimodule $\Ext{A^e}{M}{M}$ is graded-symmetric if and only if for
    all $\alpha\in\HH{A}$ we have
    \[\alpha\cdot\id*[M]=\id*[M]\cdot \alpha\in\Ext{A^e}{M}{M}.\]
\end{corollary}

\begin{proof}
  The `only if' part is obvious. Given $x\in\Ext{A^e}{M}{M}$,
  \begin{align*}
    \alpha\cdot x & = \alpha\cdot (x\circ\id*[M])             \\
             & = (-1)^{|\alpha||x|}x\circ(\alpha\cdot \id*[M]) \\
             & = (-1)^{|\alpha||x|}x\circ(\id*[M]\cdot \alpha) \\
             & = (-1)^{|\alpha||x|}(x\circ\id*[M])\cdot \alpha \\
             & = (-1)^{|\alpha||x|}x\cdot \alpha,
  \end{align*}
  where in the second and fourth equalities we use the identities in
  \Cref{associative_operadic_ideal_cohomology}\eqref{circle_bimodule} and
  the hypothesis for the converse is used in the third equality.
\end{proof}

\begin{remark}
  Explicit cochain-level formulas for the operations on the Hochschild cochain
  complex and its analogue for algebra-bimodule pairs are given
  in~\cite[Sections~3.6 and~4.1]{JM25}, which have a deliberate overlap with the
  contents of this subsection.
\end{remark}

%%%%%

\section{Obstruction theory for $\A$-algebras and
  $\A$-bimodules}\label{sec:obstructions}

In this section we leverage results from the second-named author's~\cite{Mur20}
in order to establish an obstruction theory for the extension of DG operad
morphisms $\A[k]\to\O$ with source the DG operad of $\A[k]$-algebras and target
a given graded operad. When $\O=\E{A}$ is the graded endomorphism operad of a
graded algebra $A$, we recover the obstruction theory for extending minimal
$\A[k]$-algebra structures with underlying graded algebra $A$
considered by the second-named author in~\cite{Mur20}. When $\O=\E{A,M}$ is the
linear endomorphism operad of a graded algebra $A$ and a graded $A$-bimodule, we
obtain an obstruction theory for \emph{simultaneously} extending compatible
minimal $\A[k]$-algebra-bimodule structures on the algebra $A$ and the bimodule $M$; passing to fibres as
in \Cref{space_of_bimodule_structures}, we obtain a further obstruction theory
for extending a minimal $\A[k]$-bimodule structure on $M$ over (the truncation
of) a \emph{fixed} minimal $\A$-algebra structure on $A$.

\subsection{Graded operads}\label{subsec:operad_obstruction_theory}

The DG operad $\A[k]$ is such that a morphism ${f\colon\A[k]\to\O}$ to a graded
operad $\O$ is equivalent to the data of a sequence of homogeneous operations
\[m^f_n\in\OC[n][2-n]{\os\O}=\O<n>^{2-n},\qquad 2\leq n\leq k,\] that satisfy
the following equations in the brace algebra $\OC{\os\O}$:
\begin{equation}\label{A_infinity_equations}
	\sum_{p+q=n+2}m^f_p\{m^f_q\}=0,\qquad 2\leq n\leq k-1.
\end{equation}

\begin{remark}
  \label{rmk:Ak-equations}
  Let $\O$ be a graded operad and $f\colon\A[k]\to\O$ be a morphism of DG
  operads. We make the following observations:
  \begin{enumerate}
  \item The last element in the corresponding sequence of operations, $m^f_k$,
    does not appear in any of the equations~\eqref{A_infinity_equations}. In
    particular,
    \[
      m_k^f\in\OC[k][2-k]{\os\O}=\O<k>^{2-k}
    \]
    is an arbitrary operation.
  \item If $k\geq 3$ and $n=2$, equation \eqref{A_infinity_equations}
    simplifies to
    \[m^f_2\{m^f_2\}=0,\] hence $\os\O$ is an operad with multiplication
    \[
      m^f_2\in\OC[2][0]{\os\O}=\O<2>^0
    \]
    and $\HC{\os\O}$ is a (cochain) complex with the bidegree $(1,0)$ differential
    \[
      d\colon\OC{\os\O}\longrightarrow\OC[\bullet+1]{\os\O},\qquad x\longmapsto [m_2^f,x],
    \]
    see~\Cref{multiplication_differential_cup-product}.
  \item If $k\geq 4$ and $n=3$, equation \eqref{A_infinity_equations} reads
    \[m^f_2\{m^f_3\}+m^f_3\{m^f_2\}=[m^f_2,m^f_3]=d(m_3^f)=0.\] Therefore
    the operation $m^f_3\in\OC[3][-1]{\os\O}$ is a cocycle (in general, this is
    the only operation that is a cocycle).
  \end{enumerate}
\end{remark}

The observations in \Cref{rmk:Ak-equations} motivate the following definition.

\begin{definition}
  \label{def:UMP-operad}
  Let $\O$ be a graded operad and $f\colon\A[k]\to\O$ be a morphism of DG
  operads, with $k\geq 4$. The \emph{universal Massey product of $f$} is the
  cohomology class
  \[\Hclass{m^f_3}\in\OH[3][-1]{\os\O}.\]
\end{definition}

\begin{remark}
  \label{rmk:UMP-gauge_invariant}
  The universal Massey product introduced in \Cref{def:UMP-operad} is a homotopy
  invariant. More precisely, given two DG operad morphisms ${f,g\colon\A[k]\to\O}$, $k\geq 3$, their restrictions to $\A[2]$ are
  homotopic if and only if they induce the same multiplication $m_2^f=m_2^g$ on
  $\os\O$, in which case $f$ and $g$ induce the same differential on
  $\OC{\os\O}$ and therefore define the same cohomology $\OH{\os\O}$. Moreover,
  if $k\geq 4$ and the restrictions of $f$ and $g$ to $\A[3]$ are homotopic,
  then $f$ and $g$ have the same universal Massey product:
  \[
    \Hclass{m^f_3}=\Hclass{m^g_3}\in\HC[3][-1]{\os\O}.
  \]
\end{remark}

Let $\O$ be a graded operad with multiplication. In the following statement, we
denote by $\OZ{\os\O}\subset\OC{\os\O}$ the bigraded vector space of cocycles in
the operad complex. Given a cocycle $x\in\OZ{\os\O}$, we denote its cohomology
class by ${\Hclass{x}\in\OH{\os\O}}$.

\begin{theorem}
	\label{operad_obstruction_theory}
	Let $\O$ be a graded operad equipped with a morphism of DG operads
  $f\colon\A[k+2]\to\O$ for some $\infty\geq k\geq 0$. The following statements hold:
	\begin{enumerate}
		\item\label{truncated_ss} 	There is a spectral sequence $\BK$
      with pages defined in the range $1\leq r\leq \lfloor\frac{k+3}{2}\rfloor$
      (hence truncated if $k<\infty$)
      and bidegree $(r,r-1)$ differentials
      \[
        \ssd{r}\colon\BK[r][s][t]\longrightarrow\BK[r][s+r][t+r-1]
      \]
      defined up to the second-to-last page (see \Cref{diferentials_ss,ss_cohomology} for the
      source bidegrees in which the differentials are defined).
		\item For each defined page of the spectral sequence, the terms $\BK[r][s][t]$ are defined in the right half-plane, except if $t<s<2r-3$.
		\item The terms $\BK[r][s][t]$ are vector spaces, with the following exceptions:
		\begin{enumerate}
		\item The term $\BK[r][s][t]$ is a pointed set for $0\leq s=t\leq r-2$.
		\item The term $\BK[r][s][t]$ is an abelian group in either of
      the following two ranges:
		\begin{enumerate}
		\item $r-1\leq s=t\leq 2r-4$.
		\item $1\leq s+1=t\leq r-2$.
		\end{enumerate}
		\end{enumerate}
		\item\label{diferentials_ss} The differentials \[\ssd{r}\colon\BK[r][s][t]\longrightarrow\BK[r][s+r][t+r-1]\] are defined on the right half plane, except for $s=t\leq r-1$. If the source and the target are vector spaces, then this differential is a vector space homomorphism, except if $r\geq 2$ and $s+1=t=r-1$. Otherwise, it is an abelian group homomorphism. 
		\item\label{ss_cohomology} The term $\BK[r+1][s][t]$ is the cohomology of
      the differential of the previous page, as long as the incoming and the
      outgoing differentials are defined on the source term $\BK[r][s][t]$. By
      convention, for $s<r$ and $t>s$ the incoming differential is ${0\to\BK[r][s][t]}$.
		\item\label{Bousfield-Kan} When $k=\infty$, that is $f\colon\A\to\O$, all
      pages of the spectral sequence are defined, and it coincides for $t\geq s$
      with the Bousfield--Kan fringed spectral sequence of the tower of
      fibrations \[\cdots\twoheadrightarrow\Map{\A[n+1]}{\O}\twoheadrightarrow
        \Map{\A[n]}{\O}\twoheadrightarrow\cdots\twoheadrightarrow\Map{\A[2]}{\O},\]
      based at the corresponding restrictions of $f$, see~\cite[Section~IX.4]{BK72}.
		\item\label{lower-diagonal} For $0\leq s\leq \min\{r-1,k-r+1\}$, the term
      $\BK[r][s][s]$ is the pointed set of homotopy classes of maps
      $g\colon\A[s+2]\to\O$ that can be extended to $\A[s+r+1]$ and whose restriction to
      $\A[s+1]$ is homotopic to the restriction of $f\colon\A[k+2]\to\O$:
      \[
        \begin{tikzcd}
          \A[s+1]\dar[hook]\rar[hook]\ar[phantom]{dr}[description]{\Rightarrow}&\A[s+2]\dar{g}\rar[hook]&\A[s+r+1]\dlar[dotted]\\
          \A[k+2]\rar[swap]{f}&\O
        \end{tikzcd}
      \]
      The base point is the restriction of $f$ to $\A[s+2]$. This pointed set is actually an abelian group for $s=r-1$ but we cannot describe the sum with this characterisation.
		\item\label{obstructions} For each $\lceil\frac{k-1}{2}\rceil\leq \ell\leq
      k$, there is an obstruction in $\BK[k+1-\ell][k+1][k]$ that vanishes if
      and only if there exists a morphism $g\colon\A[k+3]\to\O$ such that the
      restrictions of ${f\colon\A[k+2]\to\O}$ and ${g\colon\A[k+3]\to\O}$ to
      $\A[\ell+2]$ coincide:
      \[
        \begin{tikzcd}
          \A[\ell+2]\dar[hook]\rar[hook]\ar[phantom]{dr}[description]{=}&\A[k+3]\dar[dotted]{g}\\
          \A[k+2]\rar[swap]{f}&\O
        \end{tikzcd}
      \]
		\item\label{first_page} The first page of the truncated spectral sequence is
      given by the cochains in operad complex in \Cref{operad_complex}: \[\BK[1][s][t]=\OC[s+2][-t]{\os\O},\quad s\geq 0,\quad t\in\ZZ.\]
		\item\label{first_differential} If $k\geq 1$, the differential
      $\ssd{1}\colon\BK[1][s][t]\to\BK[1][s+1][t]$ of the first page coincides with the
      differential of the operad complex in
      \Cref{multiplication_differential_cup-product}:
      \begin{align*}
        \ssd{1}\colon\OC[s+2][-t]{\os\O}&\longrightarrow\OC[s+3][-t]{\os\O},\\
        x&\longmapsto[m^f_2,x].
      \end{align*}
		\item\label{second_page} If $k\geq 1$, the second page $\BK[2]$ of the
      truncated spectral sequence is given by cohomology of the operad
      complex, \[\BK[2][s][t]=\OH[s+2][-t]{\os\O},\quad s> 0,\quad t\in\ZZ,\] by
      the cocycles of the operad complex, \[\BK[2][0][t]=\OZ[2][-t]{\os\O},\quad
        t>0,\] and, finally, the term $\BK[2][0][0]$ is the pointed set of
      multiplications in $\O$, based at the multiplication $m_2^f$ induced by
      $f\colon\A[k]\to\O$.
		\item\label{lower-diagonal_second_page} Let $k\geq2$, $r=2$ and $s=1$, and let
      \[
        \Hclass{m_3^f}\in\OH[3][-1]{\os\O}
      \]
      be the universal Massey product of $f$.
      \Cref{lower-diagonal,second_page} define a pointed bijection between the
      pointed set of homotopy classes of maps $g\colon\A[3]\to\O$ which extend
      to $\A[4]$ and such that the restrictions of $f$ and $g$ to $\A[2]$ are homotopic,
      \[
        \begin{tikzcd}
          \A[2]\dar[hook]\rar[hook]\ar[phantom]{dr}[description]{\Rightarrow}&\A[3]\dar{g}\rar[hook]&\A[4]\dlar[dotted]\\
          \A[k+2]\rar[swap]{f}&\O,
        \end{tikzcd}
      \]
      and the term $\BK[2][1][1]=\OH[3][-1]{\os\O}$. The base point in the
      source is the restriction of $f$ to $\A[3]$ and the base point in the
      target is $0$. The bijection maps the homotopy class of $g\colon\A[3]\to\O$ to $\Hclass{m^g_3}-\Hclass{m^f_3}$.
		\item\label{second_differential} If $k\geq 3$, the differential
      $\ssd{2}\colon\BK[2][s][t]\to\BK[2][s+2][t+1]$ of the second page is given
      by the Gerstenhaber Lie bracket with the universal Massey product,
      \begin{align*}
        \ssd{2}\colon\OH[s+2][-t]{\os\O}&\longrightarrow\OH[s+4][-t-1]{\os\O}&s\geq 1,\\
        x&\longmapsto[\Hclass{m^f_3},x],
      \end{align*}
      possibly composed with the projection of cocycles onto cohomology,
      \begin{align*}
        \ssd{2}\colon\OZ[2][-t]{\os\O}&\longrightarrow\OH[4][-t-1]{\os\O}&t>1,\\
        x&\longmapsto[\Hclass{m^f_3},\Hclass{x}],
      \end{align*}
      and, finally,
      \begin{align*}
        \ssd{2}\colon\OZ[2][-1]{\os\O}&\longrightarrow\OH[4][-2]{\os\O}\\
        x&\longmapsto\Hclass{x}^2+[\Hclass{m^f_3},\Hclass{x}].
      \end{align*}
      Notice that,  since $|x|=1$ is odd, $\Hclass{x}^2=0$ if $\cchar{\kk}\neq 2$ by graded
      commutativity.
		\item\label{first_obstruction} For $k=2$ and $\ell=1$, the obstruction in \Cref{obstructions} to
      the existence of a map $g\colon\A[5]\to\O$ such that the restrictions of
      $f$ and $g$ to $\A[3]$ coincide,
      \[
        \begin{tikzcd}
          \A[3]\rar[hook]\dar[hook]\ar[phantom]{dr}[description]{=}&\A[5]\dar[dotted]{g}\\
          \A[4]\rar[swap]{f}&\O
        \end{tikzcd}
      \]
      is the Gerstenhaber square of the universal
      Massey product \[\Sq[\Hclass{m^f_3}]\in\BK[2][3][2]=\OH[5][-2]{\os\O}.\]
		\item\label{primary_obstructions} More generally, for $k\geq 2$ and
      $\ell=k-1$, the obstruction in \Cref{obstructions} to the existence of a map
      $g\colon\A[k+3]\to\O$ such that the restrictions of $f$ and $g$ to
      $\A[k+1]$ coincide,
      \[
        \begin{tikzcd}
          \A[k+1]\rar[hook]\dar[hook]\ar[phantom]{dr}[description]{=}&\A[k+3]\dar[dotted]{g}\\
          \A[k+2]\rar[swap]{f}&\O
        \end{tikzcd}
      \]
      which lies in the term
      ${\BK[2][k+1][k]=\OH[k+3][-k]{\os\O}}$, is represented by the operad
      cocycle \[\sum_{p+q=k+4}m^f_p\{m^f_q\}\in\OZ[k+3][-k]{\os\O}.\]
	\end{enumerate}
\end{theorem}

\begin{proof}
	This result is a compendium of \cite[Sections~4, 5 and 6]{Mur20}. Some of the
  results in \cite[Section~6]{Mur20} are stated for the specific case the
  endomorphism operad $\O=\E{V}$ of a graded vector space. Nevertheless, the
  same proofs work for general graded operads since they only depend on the
  homotopy theory of the operad $\A$ developed in \cite[Section~3]{Mur20}. One
  only needs to replace the Hochschild complex with the operad complex, compare
  with \Cref{ex:OCosEV} and \Cref{cohomologies}.

	In \cite{Mur20}, some signs appear in the differentials $\ssd{1}$ and
  $\ssd{2}$ of \Cref{first_differential,second_differential}, but one can tweak
  the identifications in \Cref{first_page,second_page} with suitable signs in
  order to get rid of them in the differentials.

	\Cref{lower-diagonal_second_page} was not explicitly stated in \cite{Mur20},
  but it follows easily from \cite[Equation~(4.8), Propositions~3.4
  and~4.16]{Mur20}. As a side remark, we mention that the map
  $g\mapsto\Hclass{m^g_3}$ also defines a bijection, which is not pointed if
  $\Hclass{m^f_3}\neq0$.

	\Cref{primary_obstructions} is a generalisation of \Cref{first_obstruction} to any $k\geq 2$. The latter was proved in \cite[Proposition~6.7]{Mur20} using a specific case of the formula in \cite[Proposition~3.4]{Mur20}. \Cref{primary_obstructions} follows from the same argument, but using the general formula.
\end{proof}

\Cref{fig:mrd} depicts a generic page of the (truncated) spectral sequence in
\Cref{operad_obstruction_theory}.

\begin{remark}\label{second_representative}
	The cocycle in \Cref{operad_obstruction_theory}\eqref{primary_obstructions} can be rewritten as 
	\[[m^f_2,m^f_{k+2}]+\sum_{\substack{p+q=k+4\\p,q>2}}m^f_p\{m^f_q\}\in\OZ[k+3][-k]{\os\O}.\]
	Since $[m^f_2,m^f_{k+2}]=d(m^f_{k+2})$, another representative is
	\[\sum_{\substack{p+q=k+4\\p,q>2}}m^f_p\{m^f_q\}\in\OZ[k+3][-k]{\os\O}.\]
\end{remark}

\setcounter{rpage}{5}
\begin{figure}[t]
	\begin{tikzpicture}
	% Malla
	\draw[step=\rescale,gray,very thin] (0,{-\abajo*\rescale+\margen}) grid ({(3*\rpage -3 +\derecha)*\rescale-\margen},{(3*\rpage -3 +\arriba)*\rescale-\margen});
	% Parte roja
	\filldraw[fill=red,draw=none,opacity=0.2] (0,0) -- ({(3*\rpage -3 +\arriba)*\rescale-\margen},{(3*\rpage -3 +\arriba)*\rescale-\margen}) --
	(0,{(3*\rpage -3 +\arriba)*\rescale-\margen}) -- cycle;
	\draw[red!70,thick]  (0,0) -- ({(3*\rpage -3 +\arriba)*\rescale-\margen},{(3*\rpage -3 +\arriba)*\rescale-\margen});
	% Puntos rojos
	\foreach  \x in {2,...,\rpage}
	\node[fill=red,draw=none,circle,inner sep=.5mm,opacity=1]   at ({(\x-2)*\rescale},{(\x-2)*\rescale}) {};
	% Parte azul 
	\filldraw[fill=blue,draw=none,opacity=0.2]  (0,{(3*\rpage -3 +\arriba)*\rescale-\margen}) --
	(0,2*\rescale) -- ({(\rpage-2)*\rescale},{\rpage*\rescale}) -- ({(\rpage-2)*\rescale},{(\rpage-1)*\rescale}) --({(2*\rpage-3)*\rescale},{(2*\rpage-2)*\rescale}) -- ({(2*\rpage-3)*\rescale},{-\abajo*\rescale+\margen}) -- ({(3*\rpage -3 +\derecha)*\rescale-\margen},{-\abajo*\rescale+\margen}) --
	({(3*\rpage -3 +\derecha)*\rescale-\margen},{(3*\rpage -3 +\arriba)*\rescale-\margen}) -- cycle;
	\draw[blue!70,thick] (0,2*\rescale) -- ({(\rpage-2)*\rescale},{\rpage*\rescale}) -- ({(\rpage-2)*\rescale},{(\rpage-1)*\rescale}) --({(2*\rpage-3)*\rescale},{(2*\rpage-2)*\rescale}) -- ({(2*\rpage-3)*\rescale},{-\abajo*\rescale+\margen}) node[black,anchor=north] {$\scriptscriptstyle 2r-3$};
	\node[black,anchor=north] at ({(\rpage-2)*\rescale},{-\abajo*\rescale+\margen}) {$\scriptscriptstyle r-2$};
%	% Parte azul
%%	\filldraw[fill=blue,draw=none,opacity=0.2] ({(2*\rpage-2)*\rescale},{-\abajo*\rescale+\margen}) -- ({(3*\rpage -3 +\derecha)*\rescale-\margen},{-\abajo*\rescale+\margen}) --
%	({(3*\rpage -3 +\derecha)*\rescale-\margen},{(3*\rpage -3 +\arriba)*\rescale-\margen}) -- ({(2*\rpage-2)*\rescale},{(3*\rpage -3 +\arriba)*\rescale-\margen}) -- cycle;
%	\draw[blue!70,thick]  ({(2*\rpage-2)*\rescale},{-\abajo*\rescale+\margen}) node[black,anchor=north] {$\scriptscriptstyle 2r-2$} -- ({(2*\rpage-2)*\rescale},{(3*\rpage -3 +\arriba)*\rescale-\margen});
	% Eje horizontal
	\draw [->] (0,0)  -- ({(3*\rpage -3 + \derecha)*\rescale},0) node[anchor=north] {$\scriptstyle s$};
	% Eje vertical
	\draw [->] (0,-{\abajo*\rescale}) -- (0,{(3*\rpage -3 +\arriba)*\rescale})node[anchor=east] {$\scriptstyle t$};
	% Flecha
	\node[fill=black,draw=none,circle,inner sep=.5mm,opacity=1]   at ({\rpage*\rescale},{\rpage*\rescale}) {};
	\draw [black,thick,->] ({\rpage*\rescale},{\rpage*\rescale})  -- ({2*\rpage*\rescale},{(2*\rpage-1)*\rescale});
	\end{tikzpicture}
	\caption{Range of definition of the extended (truncated) spectral sequence in
    \Cref{operad_obstruction_theory} with $r=\arabic{rpage}$. The blue region
    consist of vector spaces, the red region consists of abelian groups, with
    the exception of the red dots that are plain pointed sets. We depict a
    differential that `jumps' from the range of definition of the classical
    Bousfield--Kan spectral sequence (red region) to its extended part.}
  \label{fig:mrd}
\end{figure}
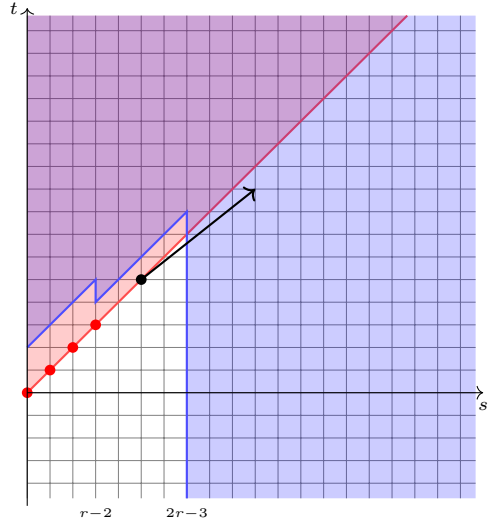

\begin{remark}
  Notice that statements \eqref{first_page}--\eqref{primary_obstructions} in
  \Cref{operad_obstruction_theory}
  (with the exception of \Cref{lower-diagonal_second_page}), require the terms
  $\BK[r][s][t]$ to be defined, at least partially, \emph{outside} the range of
  definition ${t\geq s\geq0}$ of the classical Bousfield--Kan spectral sequence
  of a tower of fibrations~\cite[Section~IX.4]{BK72}.
\end{remark}

\subsubsection{Graded algebras}\label{algebra_obstruction_theory}

We now explain how to interpret \Cref{operad_obstruction_theory} when $\O$ is
the endomorphism operad of a graded algebra, which is the case studied in detail
in~\cite[Section~6]{Mur20}. For concrete applications of
\Cref{operad_obstruction_theory} in this case we refer the reader
to~\cite{Mur22,JKM22,JKM24}.

Let $k\geq 3$. A \emph{minimal $\A[k]$-algebra} $(A,m_2^A,m^A_3,\dots,m^A_k)$ consists
of a graded algebra $A=(A,m_2^A)$ equipped with a sequence of operations
\[m^A_n\colon A\otimes\stackrel{n}{\cdots}\otimes A\longrightarrow A,\qquad
  3\leq n\leq k,\] of degree $2-n$, such that the following equations in the
Hochschild complex $\HC{A}$ hold:
\[\sum_{p+q=n+2}m^A_p\{m^A_q\}=0,\qquad 3\leq n\leq k-1.\]
The last operation, $m^A_k$, does not appear in any of these equations an hence
it is an arbitrary cochain in $\HC[k][2-k]{A}$. Equivalently, the above
operations define a map $m^A\colon\A[k]\to\E{A}$. For $k\geq 4$, the cochain $m^A_3\in\HC[3][-1]{A}$ is a cocycle
(\Cref{rmk:Ak-equations}) and its cohomology class
\[\Hclass{m^A_3}\in\HH[3][-1]{A}.\]
is the \emph{universal Massey product of $A$}.

A \emph{gauge $\A[k]$-isomorphism}
between minimal $\A[k]$-algebras with the same underlying graded algebra ($m_2^A=\bar{m}_2^A$)
\[
  (f_2,\dots,f_k)\colon (A,m_2^A,m^A_3,\dots,m^A_k)\longrightarrow(A,\bar{m}_2^A,\bar{m}^A_3,\dots,\bar{m}^A_k)
\]
is a sequence of degree $1-n$ maps
\[f_n\colon A^{\otimes n}\longrightarrow A,\qquad 2\leq n\leq k,\] satisfying
certain equations (corresponding to
\Cref{def:gauge_O_infty_isomorphism_algebras} with $\O_\infty=\O=\A[k]$, keeping in
mind that the DG operad $\A[k]$ is cofibrant). The
universal Massey product is a gauge invariant.\footnote{In fact, it is invariant by gauge
$\A[3]$-isomorphisms.}

\begin{remark}
A minimal $\A[2]$-algebra $(A,m_2)$ consists of a graded vector space $A$ with a
degree $0$ bilinear map $m_2\colon A\otimes A\to A$ (without any further
conditions). Graded algebras are therefore examples of minimal $\A[2]$-algebras.
\end{remark}

\Cref{table:EA} summarises the replacements to be made in
\Cref{operad_obstruction_theory} order to interpret everything in terms of
minimal $\A[k]$-algebras and Hochschild cohomology. Here, $\HZ{A}\subset\HC{A}$
denotes the bigraded subspace of cocycles in the Hochschild complex.

\begin{table}
\noindent\begin{tblr}{colspec={X[1,r]|X[1,l]X[1]X[2,r]|X[3,l]}}
	\cline[1.5pt]{-}
	$\O$ & $\E{A}$&&
	$f\colon\A[k+2]\to\O$ & $(A,m_2^A,m^A_3,\dots,m^A_{k+2})$\\
	$\OC{\os\O}$ & $\HC{A}$&&$m^f_i$ & $m^A_i$\\
	$\OZ{\os\O}$ & $\HZ{A}$&&$g\colon\A[k+3]\to\O$ & $(A,\bar{m}_2^A,\bar{m}^A_3,\dots,\bar{m}^A_{k+3})$\\
	$\OH{\os\O}$ & $\HH{A}$&&$m^g_i$ & $\bar{m}^A_i$\\
	\cline[1.5pt]{-}
\end{tblr}\vspace{1em}
\caption{Interpretation of \Cref{operad_obstruction_theory} for $\O=\E{A}$, that
is for minimal $\A[k]$-algebra structures on $A$.}
\label{table:EA}
\end{table}

\subsubsection{Graded algebra-bimodule pairs}\label{simultaneous_obstruction_theory}

We now explain the interpretation of \Cref{operad_obstruction_theory} when $\O$ is
the linear endomorphism operad of a graded algebra and a compatible graded
bimodule.

Let $k\geq 3$ and let $(A,m_2^A,m^A_3,\dots,m^A_k)$ be a minimal $\A[k]$-algebra. A
\emph{minimal $\A[k]$-bimodule} $(M,m_2^M,m^M_3,\dots,m^M_k)$ over it consists of a
graded $A$-bimodule $(M,m_2^M)$ together with sequences
\[
  m^M_n=(m^M_{p,q})_{p+1+q=n},\qquad 3\leq n\leq k,
\]
of homogeneous operations
\[m^M_{p,q}\colon A^{\otimes p}\otimes M\otimes A^{\otimes q}\longrightarrow
  M,\] of degree $2-n=1-p-q$ such that, if $m_2^M=m_{0,1}^M+m_{1,0}^M$
and $m^M_{0,1}\colon M\otimes A\to M$ and
$m^M_{1,0}\colon A\otimes M\to M$ are the right and left $A$-module structure
maps, then the following equations hold in the
bimodule complex $\BC{A^e}{M}$:
\[\sum_{p+q=n+2}[m^M_p,m^A_q]+m^M_p\circ m^M_q=0,\qquad 3\leq n\leq k-1.\]
Notice that the operations $m^A_k$ and $m^M_k$
do not appear in any of these equations.
Equivalently, the operations
\[
  m_n^{A\ltimes M}=m_n^A+m_n^M,\qquad 2\leq n\leq k,
\]
define a map $\Astr[A\ltimes M]\colon\A[k]\to\E{A,M}$
rendering the following diagram commutative, where the vertical map is the canonical projection and the horizontal map is the
$\A[k]$-algebra structure on $A$:
\[
  \begin{tikzcd}
    &\E{A,M}\dar[two heads]\\
    \A[k]\rar[swap]{\Astr[A]}\urar{\Astr[A\ltimes M]}&\E{A},
  \end{tikzcd}
\]

One can simultaneously define a minimal $\A[k]$-algebra $(A,m_2^A,m^A_3,\dots,m^A_k)$
and a minimal $\A[k]$-bimodule $(M,m_2^M,m^M_3,\dots,m^M_k)$ as a graded algebra $A$,
an $A$-bimodule $M$, with homogeneous operations as above and such that the
following equations hold in the bimodule Hochschild complex $\HCE{A}{M}$:
\[\sum_{p+q=n+2}(m^A_p+m^M_p)\{m^A_q+m^M_q\}=0,\qquad 3\leq n\leq k-1.\]
For $k\geq 4$, the cochain
\[
  \Astr<3>[A\ltimes M]\coloneqq m^A_3+m^M_3\in\HCE[3][-1]{A}{M}
\]
is a cocycle (\Cref{rmk:Ak-equations}) and its cohomology class
\[\Hclass{\Astr<3>[A\ltimes M]}\in\HHE[3][-1]{A}{M}\]
is called the \emph{bimodule universal Massey product}. As in the algebra case, we have
\emph{gauge $\A[k]$-isomorphisms} between pairs consisting of a minimal
$\A[k]$-algebra and a minimal $\A[k]$-bimodule, all of them with the same
underlying graded algebra and bimodule (corresponding to
\Cref{def:gauge_iso_pairs} with $\O_\infty=\O=\A[k]$). The bimodule universal Massey product is also
invariant under gauge $\A[k]$-isomorphisms of pairs.

\begin{remark}
  A minimal $\A[2]$-bimodule $(M, m_{0,1}^M+m_{1,0}^M)$ over
  the a minimal $\A[2]$-algebra $(A,m_2^A)$ is a graded vector space $M$ with
  degree $0$ bilinear maps
  \[
    {m_{0,1}^M\colon M\otimes A\to M}\qquad\text{and}\qquad{m_{1,0}\colon
  A\otimes M\to M}
  \]
  (without any further conditions). Bimodules over graded
  associative algebras are examples of minimal $\A[2]$-bimodules.
\end{remark}

We can apply \Cref{operad_obstruction_theory} to the linear endomorphism operad
$\E{A,M}$. In this case, for $k\geq 3$, maps $\A[k]\to\E{A,M}$ correspond to
minimal $\A[k]$-algebras structures on the
underlying graded vector space of $A$, equipped with a compatible minimal $\A[k]$-bimodule
structure on on the underlying graded vector space of
$M$. The operad complex is the bimodule Hochschild complex, etc., see
\Cref{table:EAM}, where we denote by $\HZE{A}{M}\subset\HCE{A}{M}$ the bigraded subspace of
cocycles in the bimodule Hochschild complex.

\begin{table}
\noindent\begin{tblr}{colspec={X[1.85,r]|X[1.85,l]X[0.05]X[2.9,r]|X[8.4,l]}}
	\cline[1.5pt]{-}
	$\O$ & $\E{A,M}$&&$f\colon\A[k+2]\to\O$ & $(A,m_2^A,\dots,m^A_{k+2}),(M,m_2^M,\dots,m^M_{k+2})$\\
  $\OC{\os\O}$&$\HCE{A}{M}$&&$m^f_i$ & $m^A_i+m^M_i$\\
	$\OZ{\os\O}$&$\HZE{A}{M}$&&$g\colon\A[k+3]\to\O$ & $(A,\bar{m}_2^A,\dots,\bar{m}^A_{k+3}),(M,\bar{m}_2^M,\dots,\bar{m}^M_{k+3})$\\
	$\OH{\os\O}$&$\HHE{A}{M}$&&$m^g_i$&$\bar{m}^A_i+\bar{m}^M_i$\\
	\cline[1.5pt]{-}
\end{tblr}
\vspace{1em}
\caption{Interpretation of \Cref{operad_obstruction_theory} for $\O=\E{A,M}$,
  that is for minimal $\A[k]$-algebra-bimodule structures on the pair $(A,M)$.}
\label{table:EAM}
\end{table}

The second representative of the obstruction in \Cref{operad_obstruction_theory}\eqref{primary_obstructions}, given in \Cref{second_representative}, is the following cocycle in $\HZE[k+3][-k]{A}{M}$:
\begin{align*}
	\sum_{\substack{p+q-1=k+3\\p,q>2}}(m^A_p+m^M_p)&\{m^A_q+m^M_q\}=\\
	&\sum_{\substack{p+q-1=k+3\\p,q>2}}m^A_p\{m^A_q\}
	+\sum_{\substack{p+q-1=k+3\\p,q>2}}m^M_p\{m^A_q+m^M_q\}.
\end{align*}

\subsection{Fibre-wise obstructions}\label{sec:fiber-wise}

Let $\O$ be a graded operad $\O$ equipped with an associative operadic ideal
$\I\subset\O$ and a map $h\colon\A\to\O/\I$. In this section we study maps
$f\colon\A[k]\to\O$ lifting the restriction of $h$ to $\A[k]$ along the
canonical projection $q\colon\O\twoheadrightarrow\O/\I$,
\[
  \begin{tikzcd}
    \A[k]\dar[hook]\rar[dotted]{f}\ar[phantom]{dr}[description]{=}&\O\dar[two heads]{q}\\
    \A\rar[swap]{h}&\O/\I
  \end{tikzcd}
\]
These maps $f$ are the points of the homotopy fibre $\Str{\A[k]}{h}{\I}$ of the fibration
\[\Map{\A[k]}{\O}\longrightarrow\Map{\A[k]}{\O/\I}\]
induced by the canonical projection at the restriction of $h$ to $\A[k]$,
compare with \Cref{space_of_bimodule_structures}.

\begin{definition}\label{fibre-wise_homotopy}
  In the previous setting,
  two maps $f,g\colon\A[k]\to\O$ are \emph{fibre-wise homotopic} if there is a homotopy between them which becomes the trivial homotopy when composed with $\O\twoheadrightarrow\O/\I$.
\end{definition}

The set $\pi_0\Str{\A[k]}{h}{\I}$ is the set of fibre-wise homotopy classes of maps $f\colon\A[k]\to\O$ such that $qf=h|_{\A[k]}$.

\begin{theorem}
	\label{fibre-wise_obstruction_theory}
	Let $\O$ be a graded operad, $\I\subset\O$ an associative operadic ideal, and
  $h\colon\A\to\O/\I$ a map. Denote $q\colon\O\twoheadrightarrow\O/\I$ the
  canonical projection. Assume that we have a map $f\colon\A[k+2]\to\O$ for some
  $\infty\geq k\geq 0$ such that $qf\colon\A[k+2]\to\O/\I$ is the restriction of
  $h\colon\A\to\O/\I$ to $\A[k+2]$, that is $q(m^f_i)=m^h_i$, $2\leq i\leq k+2$:
  \[
  \begin{tikzcd}
    \A[k+2]\dar[hook]\rar{f}\ar[phantom]{dr}[description]{=}&\O\dar[two heads]{q}\\
    \A\rar[swap]{h}&\O/\I
  \end{tikzcd}
  \]
  Then, statements \eqref{truncated_ss}--\eqref{ss_cohomology} in
  \Cref{operad_obstruction_theory} hold as stated and, moreover, we have
  the following:
	\begin{enumerate}\setcounter{enumi}{5}
		\item\label{Bousfield-Kan_fibre-wise} When $k=\infty$, that
      is~$f\colon\A\to\O$ is a map with $qf=h$, all pages of the spectral
      sequence are defined and it coincides for $t\geq s$ with the
      Bousfield--Kan fringed spectral sequence of the tower of
      fibrations \[\cdots\twoheadrightarrow\Str{\A[n+1]}{h}{\I}\twoheadrightarrow
        \Str{\A[n]}{h}{\I}\twoheadrightarrow\cdots\twoheadrightarrow\Str{\A[2]}{h}{\I}\] based at the
      restrictions of $f$ to $\A[n]$, see ~\cite[Section~IX.4]{BK72}.
		\item\label{lower-diagonal_fibre-wise} For $0\leq s\leq\min\{r-1,k-r+1\}$,
      the term $\BK[r][s][s]$ is the set of fibre-wise homotopy classes of maps
      $g\colon\A[s+2]\to\O$ that extend to $\A[s+r+1]$ such that the
      restrictions of $f$ and $g$ to $\A[s+1]$ are fibre-wise homotopic, and such that $qg\colon\A[s+2]\to\O/\I$ equals the restriction of $h\colon\A\to\O/\I$ to
      $\A[s+2]$:
      \[
        \begin{tikzcd}
          \A[s+1]\dar[hook]\rar[hook]\ar[phantom]{dr}[description]{\Rightarrow}&\A[s+2]\dar{g}\rar[equals]\ar[phantom]{ddr}[description]{=}&\A[s+2]\ar[hook]{dd}\\
          \A[k+2]\rar[swap]{f}\dar[hook]\ar[phantom]{dr}[description]{=}&\O\dar[two heads]{q}\\
          \A\rar[swap]{h}&\O/\I&\A\lar{h}
        \end{tikzcd}\qquad
        \begin{tikzcd}
          \A[s+2]\rar[hook]\dar[swap]{g}&\A[s+r+1]\ar[dotted]{dl}\\
          \O
        \end{tikzcd}
      \]
      This set is pointed by the restriction of $f\colon\A[k+2]\to\O$ to $\A[s+2]$.
      % intersection
      % of \[\ker\left[\pi_0\Str{\A[s+2]}{h}{\I}\to\pi_0\Str{\A[s+1]}{h}{\I}\right]\]
      % with \[\im\left[\pi_0\Str{\A[s+r+1]}{h}{\I}\to\pi_0\Str{\A[s+2]}{h}{\I}\right].\]
      % Here, the sets of connected components is pointed at the restrictions of $f$ to $\A[s+2]$, $\A[s+1]$ and $\A[s+r+1]$, respectively, and the maps are pointed.
		\item\label{obstructions_fibre-wise} For each $\lceil\frac{k-1}{2}\rceil\leq
      \ell\leq k$, there is an obstruction in $\BK[k+1-\ell][k+1][k]$ which
      vanishes if and only if there exists a map $g\colon\A[k+3]\to\O$ such that
      $qg$ equals the restriction of $h$ to $\A[k+3]$ and the restrictions of
      $f$ and $g$ to $\A[\ell+2]$ coincide:
      \[
        \begin{tikzcd}
          \A[\ell+2]\dar[hook]\rar[hook]\ar[phantom]{dr}[description]{=}&\A[k+3]\dar[dotted]{g}\rar[equals]\ar[phantom]{ddr}[description]{=}&\A[k+3]\ar[hook]{dd}\\
          \A[k+2]\ar{r}[swap]{f}\dar[hook]\rar\ar[phantom]{dr}[description]{=}&\O\dar[two heads]{q}\\
          \A\rar[swap]{h}&\O/\I&\A\lar{h}
        \end{tikzcd}
      \]
		\item\label{first_pag_fibre-wise} The first page of the truncated spectral
      sequence is given by the cochains in the operadic ideal complex in \Cref{cohomologies}, \[\BK[1][s][t]=\IC[s+1][-t]{\os\I}[\os\O],\quad s\geq 0,\quad t\in\ZZ.\]
		\item\label{first_differential_fibre-wise} For $k\geq1$, the differential
      $\ssd{1}\colon\BK[1][s][t]\to\BK[1][s+1][t]$ of the first page coincides with the
      differential of operadic ideal complex,
      \begin{align*}
        \ssd{1}\colon\IC[s+1][-t]{\os\I}[\os\O]&\longrightarrow\IC[s+2][-t]{\os\I}[\os\O],\\
        x&\longmapsto[m^f_2,x].
      \end{align*}
		\item\label{second_page_fibre-wise} For $k\geq1$, the second page $\BK[2]$
      of the truncated spectral sequence is given by the
      cohomology (see \Cref{def:op-Ext}) \[\BK[2][s][t]=\IH[s+1][-t]{\os\I}[\os\O],\quad s> 0,\quad t\in\ZZ,\]
      by the cocycles in the operadic ideal
      complex \[\BK[2][0][t]=\IZ[1][-t]{\os\I},\quad t>0,\] and, finally, $\BK[2][0][0]$ is the pointed set of multiplications in $\O$ whose projection to $\O/\I$ is the multiplication induced by $h\colon\A\to\O/\I$, based at the multiplication induced by $f\colon\A[k+2]\to\O$.
		\item\label{lower-diagonal_second_page_fibre-wise} Let $k\geq2$, $r=2$ and
      $s=1$. \Cref{lower-diagonal_fibre-wise,second_page_fibre-wise} define a
      pointed bijection between the set of homotopy fibre-wise classes of maps
      $g\colon\A[3]\to\O$ that extend to $\A[4]$, such that the restrictions of
      $f$ and $g$ to $\A[2]$ are fibre-wise homotopic, and such that
      $qg\colon\A[3]\to\O/\I$ equals the restriction of $h\colon\A\to\O/\I$ to $\A[3]$,
      \[
        \begin{tikzcd}
          \A[2]\dar[hook]\rar[hook]\ar[phantom]{dr}[description]{\Rightarrow}&\A[3]\dar{g}\rar[equals]\ar[phantom]{ddr}[description]{=}&\A[3]\ar[hook]{dd}\\
          \A[k+2]\rar[swap]{f}\dar[hook]\ar[phantom]{dr}[description]{=}&\O\dar[two heads]{q}\\
          \A\rar[swap]{h}&\O/\I&\A\lar{h}
        \end{tikzcd}\qquad
        \begin{tikzcd}
          \A[3]\rar[hook]\dar[swap]{g}&\A[4]\ar[dotted]{dl}\\
          \O
        \end{tikzcd}
      \]
      % the intersection of
      % \[\ker\left[\pi_0\Str{\A[3]}{h}{\I}\to\pi_0\Str{\A[2]}{h}{\I}\right]\]
      % with
      % \[\im\left[\pi_0\Str{\A[4]}{h}{\I}\to\pi_0\Str{\A[3]}{h}{\I}\right],\]
      and
      the pointed set $\IH[2][-1]{\os\I}[\os\O]$. The base point in the source is the
      restriction of $f\colon\A[k+2]\to\O$ to $\A[3]$ and the base point in the target is $0$. The
      bijection maps the homotopy class of $g\colon\A[3]\to\O$ to $\Hclass{m^g_3-m^f_3}$.
		\item\label{second_differential_fibre-wise} If $k\geq 3$, the differential
      $\ssd{2}\colon\BK[2][s][t]\to\BK[2][s+2][t+1]$ of the second page is given
      by the Lie bracket with the universal Massey product of the
      map $f\colon\A[k+2]\to\O$,
      \begin{align*}
        \ssd{2}\colon\IH[s+1][-t]{\os\I}[\os\O]&\longrightarrow\IH[s+3][-t-1]{\os\I}[\os\O]&s\geq 1,\\
        x&\longmapsto[\Hclass{m^f_3},x],
      \end{align*}
        possibly composed with the natural projection of bimodule cocycles onto
        its cohomology
        \begin{align*}
          \ssd{2}\colon\IZ[1][-t]{\os\I}&\longrightarrow\IH[3][-t-1]{\os\I}[\os\O]&t\geq1,\\
          x&\longmapsto[\Hclass{m^f_3},\Hclass{x}].
        \end{align*}
		\setcounter{enumi}{14}
		\item\label{primary_obstructions_fibre-wise} For $k\geq 1$ and $\ell=k-1$, the
      obstruction in $\BK[2][k+1][k]=\IH[k+2][-k]{\os\I}[\os\O]$ to the existence of a
      map $g\colon\A[k+3]\to\O$ such that $qg\colon\A[k+3]\to\O/\I$ is the
      restriction of $h\colon\A\to\O/\I$ to $\A[k+3]$ and the restrictions of
      $f\colon\A[k+2]\to\O$ and $g\colon\A[k+3]\to\O$ to $\A[k+1]$ coincide,
      \[
        \begin{tikzcd}
          \A[k+1]\rar[hook]\dar[hook]\ar[phantom]{dr}[description]{=}&\A[k+3]\dar[dotted]{g}\rar[equals]\ar[phantom]{ddr}[description]{=}&\A[k+3]\ar[hook]{dd}\\
          \A[k+2]\rar{f}\dar[hook]\ar[phantom]{dr}[description]{=}&\O\dar[two heads]{q}\\
          \A\rar[swap]{h}&\O/\I&\A\lar{h}
        \end{tikzcd}
      \]
      is represented by the
      cocycle \[\sum_{p+q=k+4}m^f_p\{m^f_q\}\in\IC[k+2][-k]{\os\I}[\os\O],\]
      compare with \Cref{obstructions_fibre-wise}.
		\item\label{first_obstruction_fibre-wise} For $k=1$ and $\ell=0$, the
      obstruction in the term
      \[
        \BK[2][2][1]=\IH[3][-1]{\os\I}[\os\O]
      \]
      to the existence of a map
      $g\colon\A[4]\to\O$ such that $qg\colon\A[4]\to\O/\I$ is the restriction
      of $h\colon\A\to\O/\I$ to $\A[4]$ and the restrictions of $f$ and $g$ to
      $\A[2]$ coincide,
      \[
        \begin{tikzcd}
          \A[2]\rar[hook]\dar[hook]\ar[phantom]{dr}[description]{=}&\A[4]\dar[dotted]{g}\rar[equals]\ar[phantom]{ddr}[description]{=}&\A[4]\ar[hook]{dd}\\
          \A[3]\rar{f}\dar[hook]\ar[phantom]{dr}[description]{=}&\O\dar[two heads]{q}\\
          \A\rar[swap]{h}&\O/\I&\A\lar{h}
        \end{tikzcd}
      \]
      is the cohomology class
      \[
        \delta(\Hclass{m^{h}_3})\in\IH[3][-1]{\os\I}[\os\O].
      \]
      Here, $\delta$ is the connecting morphism in the long exact sequence
      \eqref{operad_multiplication_long_exact_sequence}.
	\end{enumerate}
\end{theorem}

\begin{proof}
	Consider the following commutative diagram 
	\[
	\begin{tikzcd}[column sep = 12pt]
		\cdots \arrow[hook,r] & \Str{\A[n+1]}{h}{\I} \arrow[hook,d] \arrow[hook,r] & \Str{\A[n]}{h}{\I} \arrow[hook,d] \arrow[two heads,r] & \cdots \arrow[two heads,r] & \Str{\A[2]}{h}{\I} \arrow[hook,d] \\
		\cdots \arrow[two heads,r] & \Map{\A[n+1]}{\O} \arrow[two heads,d] \arrow[two heads,r] & \Map{\A[n]}{\O} \arrow[two heads,d] \arrow[two heads,r] & \cdots \arrow[two heads,r] & \Map{\A[2]}{\O} \arrow[two heads,d] \\
		\cdots \arrow[two heads,r] & \Map{\A[n+1]}{\O/\I} \arrow[two heads,r] & \Map{\A[n]}{\O/\I} \arrow[two heads,r] & \cdots \arrow[two heads,r] & \Map{\A[2]}{\O/\I}
	\end{tikzcd}
	\]
	The rows are towers of fibrations, the bottom tower is based at the restrictions of $h$, the two other towers are based from $\A[k+2]$ downwards, the columns are fibre sequences and define maps between the towers of fibrations.

	By the results in \cite[Section~6]{Mur20}, the middle and the bottom towers
  fit into the general framework of \cite[Sections~4 and 5]{Mur20}. The top
  tower also fits within this framework since the required structure is
  preserved by taking homotopy fibres. This implies that the (truncated)
  spectral sequence and obstructions are also defined for the top tower, that is
  statements \eqref{truncated_ss}--\eqref{obstructions_fibre-wise} hold. Moreover, the maps of
  towers induce maps between the corresponding (truncated) spectral sequences,
  and these maps are compatible with the obstructions. Also, the maps on $\BK[1]$ form a short exact sequence.

	The computations in \eqref{first_pag_fibre-wise} and \eqref{first_differential_fibre-wise} follow from the fact that the map on $\BK[1]$ induced by the bottom map of towers is the second map in the short exact sequence of complexes \eqref{operad_multiplication_exact_sequence}, and \eqref{second_page_fibre-wise} follows from \eqref{first_differential_fibre-wise}.

	Items \eqref{lower-diagonal_second_page_fibre-wise},\eqref{second_differential_fibre-wise} and \eqref{primary_obstructions_fibre-wise} follow from the corresponding items in \Cref{operad_obstruction_theory}, since the computations of representatives of cohomology classes here should be done in the same way as in \Cref{operad_obstruction_theory}. The quadratic term $\Hclass{x}^2$ in \Cref{operad_obstruction_theory}\eqref{second_differential} does not appear in \eqref{second_differential_fibre-wise} since $\shift{\IC{\os\I}[\os\O]}[-1]\subset\OC{\os\O}$ is a square-zero associative ideal. There is no analogue of \Cref{operad_obstruction_theory}\eqref{first_obstruction} here since there is no universal Massey product in $\IH{\os\I}[\os\O]$. This is why we have omitted that item number in the statement of this theorem. By contrast, \eqref{first_obstruction_fibre-wise} is new, so we need to prove it. According to \eqref{primary_obstructions_fibre-wise}, the obstruction is represented by 
	\[m^f_2\{m^f_3\}+m^f_3\{m^f_2\}=[m^f_2,m^f_3].\]
	Since $q(m^f_3)=m^h_3$ is a cocycle representing the universal Massey product
  of $h\colon\A\to\O/\I$, then $[m^f_2,m^f_3]=d(m^f_3)\in\IC[3][-1]{\os\I}[\os\O]$
  represents $\delta(\Hclass{m^h_3})\in\IH[3][-1]{\os\I}[\os\O]$. This finishes the proof.
\end{proof}

\subsubsection{Graded bimodules}\label{bimodule_obstruction_theory}

We now explain how to interpret \Cref{fibre-wise_obstruction_theory} when $\O$ is
the linear endomorphism operad of a graded algebra and a compatible graded
bimodule and $\I\subset\O$ is its canonical (associative) operadic ideal.

Let $A$ be a graded algebra and  $M$ a graded $A$-bimodule. In what follows we denote by $\BZ{A^e}{M}\subset\BC{A^e}{M}$ the bigraded vector space of cocycles in the bimodule complex. Given a cocycle $x\in\BZ{A^e}{M}$, we denote its cohomology class by $\Hclass{x}\in\Ext{A^e}{M}{M}$.

We can apply \Cref{fibre-wise_obstruction_theory} to the graded operad
$\O=\E{A,M}$ and its associative operadic ideal $\I=\E*{A,M}$. In this case, $h$
is a minimal $\A$-algebra structure $(A,m^A)$ on the underlying
graded vector space of $A$. For $n\geq 3$, points in $\Str{\A[n]}{h}{\I}$, that
is maps $f\colon\A[n]\to\E{A,M}$ such that $qf=h_{|_{\A[n]}}$, are minimal
$\A[n]$-bimodule structures $(M,m^M_2,\dots,m^M_{n})$ over
$(A,m^A_{2},\dots,m^A_{n})$ on the underlying graded vector space of $M$. The
operadic ideal complex is the bimodule complex, etc.,
see~\Cref{table:EAM-fibrewise}.

\begin{table}
\noindent\begin{tblr}{colspec={X[2,r]|X[3.5,l]XX[3,r]|X[4,l]}}
	\cline[1.5pt]{-}
		$\O$ & $\E{A,M}$&&
		$h\colon\A\to\O/\I$ & $(A,m^A_2,m^A_2,\dots)$\\
		$\I$ & $\E*{A,M}$&&$m^h_i$ & $m^A_i$\\
		$\Str{\A[n]}{h}{\I}$ & $\Str{\A[n]}{(A,m^A_2,\dots,m^A_{n})}{M}$&&$f\colon\A[k+2]\to\O$& $(M,m^M_2,\dots,m^M_{k+2})$\\
		$\IC{\os\I}[\os\O]$ & $\BC{A^e}{M}$&&$m^f_i$ & $m^A_i+m^M_i$\\
		$\IZ{\os\I}[\os\O]$ & $\BZ{A^e}{M}$&&$g\colon\A[k+3]\to\O$ & $(M,\bar{m}^M_2,\dots,\bar{m}^M_{k+3})$\\
		$\IH{\os\I}[\os\O]$ & $\Ext{A^e}{M}{M}$&&$m^g_i$ & $m^A_i+\bar{m}^M_i$\\
		\cline[1.5pt]{-}
\end{tblr}\vspace{1em}
\caption{Interpretation of \Cref{fibre-wise_obstruction_theory} for $\O=\E{A,M}$
  and $\I=\E*{A,M}$, that is for minimal $\A[k]$-bimodule structures on $M$ over the truncation of
  a fixed minimal $A_\infty$-algebra structure on $A$.}
\label{table:EAM-fibrewise}
\end{table}

The representative of the obstruction in \Cref{fibre-wise_obstruction_theory}\eqref{primary_obstructions_fibre-wise} is the following cocycle in $\BZ[k+2][-k]{A^e}{M}$:
\begin{align*}
	\sum_{p+q=k+4}(m^A_p+m^M_p)&\{m^A_q+m^M_q\}\\
	&=\sum_{p+q=k+4}m^A_p\{m^A_q\}
	+\sum_{p+q=k+4}m^M_p\{m^A_q+m^M_q\}\\
	&=m^M_2\{m^A_{k+2}+m^M_{k+2}\}+m^M_{k+2}\{m^A_2+m^M_2\}\\
	&\phantom{=}+\sum_{\substack{p+q=k+4\\p,q>2}}m^M_p\{m^A_q+m^M_q\}\\
	&=\id*[M]\cdot m^A_{k+2}-m^A_{k+2}\cdot\id*[M]+[m^A_2+m^M_2,m^M_{k+2}]\\
	&\phantom{=}+\sum_{\substack{p+q=k+4\\p,q>2}}\left([m^M_p,m^A_q]+m^M_p\circ m^M_q\right).\\
\end{align*}
Here we use that $(A,m^A_3,m^A_4,\dots)$ is an $\A$-algebra and the computations
in the proof of \Cref{connecting_morphism}. Since
$[m^A_2+m^M_2,m^M_{k+2}]=d(m^M_{k+2})$, another representative of the obstruction is
\[\id*[M]\cdot m^A_{k+2}-m^A_{k+2}\cdot\id*[M]+\sum_{\substack{p+q=k+4\\p,q>2}}\left([m^M_p,m^A_q]+m^M_p\circ m^M_q\right).\]
By \Cref{connecting_morphism}, the obstruction in
\Cref{fibre-wise_obstruction_theory}\eqref{first_obstruction_fibre-wise} is the
cohomology class
\[\delta(\Hclass{m^A_3})=\id*[M]\cdot\Hclass{m^A_3}-\Hclass{m^A_3}\cdot\id*[M]\in\Ext[3][-1]{A^e}{M}{M}.\]

%%%%%

\section{Intrinsic formality and almost formality theorems}
\label{sec:Kadeishvili-type_theorems}

In this section we use the obstruction theories developed in
\Cref{sec:obstructions} to prove the various intrinsic formality and almost
formality theorems for DG algebras and DG bimodules stated in the introduction.

\subsection{Intrinsic formality results}

The following result is a vast generalisation of Kadeishvili's Intrinsic
Formality Theorem (\Cref{Kadeishvili_algebras}).

\begin{theorem}\label{Kadeishvili_operads_A-infinity}
	Let $\O$ be a graded operad with a multiplication $m_2$ in $\os\O$. Suppose that the cohomology of the operad complex vanishes in the following range:
	\[\OH[n+2][-n]{\os\O}=0,\qquad n\geq1.\]
	Then, every map $g\colon \A\to\O$ with underlying multiplication $m_2^g=m_2$
  is homotopic to the map $f\colon \A\to\O$ defined by the multiplication $m_2^f=m_2$
  and $m_n^f=0$ for $n\geq 3$, that is by the special point
  $f\colon\A\to\Ass\stackrel{m_2}{\to}\O$ of the mapping space $\Map{\A}{\O}$.
\end{theorem}

\begin{proof}
  Recall that the mapping space $\Map{\A}{\O}$ is the homotopy limit of the
  canonical tower of Kan fibrations
  \[
    \cdots\twoheadrightarrow\Map{\A[s+2]}{\O}\twoheadrightarrow\cdots\twoheadrightarrow\Map{\A[3]}{\O}\twoheadrightarrow\Map{\A[2]}{\O},\qquad s\geq0,
  \]
  which we regard as being based at the restrictions of the map $f\colon\A\to\O$.
  According to \cite[Theorem~IX.3.1]{BK72}, there is a Milnor short exact
  sequence of pointed sets
  \[
    *\rightarrow {\lim_s}^1\; \pi_1\Map{\A[s+2]}{\O}\rightarrow\pi_0\Map{\A}{\O}\rightarrow\lim_s\pi_0\Map{\A[s+2]}{\O}\rightarrow *
  \]
  
  We prove first that the map $g\colon\A\to\O$ lies in the pointed kernel of the
  rightmost map, which is to say that the restrictions of $g$ and $f$ to
  $\A[s+2]$ are homotopic for all $s\geq 0$. We prove this claim by induction on
  $s$, noticing that the case $s=0$ is obvious since the assumption
  $m_2^g=m_2=m_2^f$ means that the restrictions of $f$ and $g$ to $\A[2]$ are
  equal. Let $s>0$ and suppose that the restrictions of $f$ and $g$ to
  $\A[s+1]$ are homotopic:
  \[
    \begin{tikzcd}
      \A[s+1]\rar[hook]\dar[hook]\ar[phantom]{dr}[description]{\Rightarrow}&\A[s+2]\rar[hook]\dar&\A[2(s+1)]\dar[hook]\dlar\\
      \A\rar[swap]{f}&\O&\A\lar{g}
    \end{tikzcd}
  \]
  Consider the spectral sequence in
  \Cref{operad_obstruction_theory} associated to the map $f\colon\A\to\O$. By
  assumption, using \Cref{operad_obstruction_theory}\eqref{second_page},
  \[
    \BK[2][s][s]=\OH[s+2][-s]{\os\O}=0,\qquad s>0,
  \]
  and hence also
  \[
    \BK[r][s][s]=0,\qquad r\geq2,\quad s>0.
  \]
  Since in particular $\BK[s+1][s][s]=0$, we conclude from \Cref{operad_obstruction_theory}\eqref{lower-diagonal} that
  the restrictions of $f$ and $g$ to $\A[s+2]$ are indeed homotopic.

  To conclude the proof of the theorem it is enough to show that the
  $\lim^1$ term in the above Milnor short exact sequence collapses to a
  point. This follows immediately from the Complete Convergence Lemma in
  \cite[Section IX.5.4]{BK72} with $i=1$: Since $\BK[r][s][s]=0$ for $r\geq2$
  and $s>0$, for each $s\geq 0$ the term $\BK[r][s][s+1]$ stabilises for
  $r>\max\{2,s\}$, and consequently $\lim^1_r\BK[r][s][s+1]=0$ which is what we
  need to show for applying the aforementioned lemma.
\end{proof}

\subsubsection{Graded algebras}

The following result is a consequence of
\Cref{Kadeishvili_operads_A-infinity}.

\begin{theorem}\label{Kadeishvili_algebras_A-infinity}
	Let $A$ be a graded algebra. Suppose that its Hochschild cohomology vanishes in the following range:
	\[\HH[n+2][-n]{A}=0,\qquad n\geq1.\]
	Then, any minimal $\A$-algebra $(A,m^A)$ is gauge $\A$-isomorphic to the
  trivial minimal $A_\infty$-structure $(A,0)$.
\end{theorem}
\begin{proof}
  Take $\O=\E{A}$ in \Cref{Kadeishvili_operads_A-infinity}, see also \Cref{algebra_obstruction_theory}.
\end{proof}

\begin{remark}
  Kadeishvili's Intrinsic Formality Theorem (\Cref{Kadeishvili_algebras})
  follows immediately from \Cref{Kadeishvili_algebras_A-infinity},
  \Cref{minimal_models_quasi_isomorphic_algebras} for $\O=\Ass$ (or~\cite[Corollaire~1.3.1.3]{Lef03}) and \Cref{quasi-iso_unit_algebras}. See
also \cite[Theorems~10.3.10 and 11.4.9]{LV12} and notice that the characteristic $0$ hypothesis
in \cite{LV12} is not needed here since we are dealing with non-symmetric
operads.
\end{remark}

\subsubsection{Graded algebras-bimodule pairs}

\Cref{Kadeishvili_operads_A-infinity} also yields the following
\emph{simultaneous} intrinsic formality theorems for graded algebra-bimodule
pairs.

\begin{theorem}\label{Kadeishvili_simultaneous_A-infinity}
	Let $A$ be a graded algebra and $M$ an $A$-bimodule. Suppose that the bimodule Hochschild cohomology vanishes in the following range:
	\[\HHE[n+2][-n]{A}{M}=0,\qquad n\geq1.\]
	Then, any pair given by a minimal $\A$-algebra $(A,m^A)$ and a minimal
  $\A$-bimodule $(M,m^M)$ over it is gauge $\A$-isomorphic to the formal pair
  consisting of $(A,0)$ and $(M,0)$.
\end{theorem}
\begin{proof}
  Take $\O=\E{A,M}$ in \Cref{Kadeishvili_operads_A-infinity}, see also \Cref{simultaneous_obstruction_theory}.
\end{proof}

The following result corresponds to \Cref{thm:Kadeishvili_algebras_bimodules} in
the introduction.

\begin{theorem}\label{Kadeishvili_simultaneous}
  Let $A$ be a graded algebra, $M$ a graded $A$-bimodule, and suppose that the
  bimodule Hochschild cohomology of the pair $(A,M)$ vanishes in the following
  range:
  \[
    \RelBimHH[n+2]<-n>{A}{M}=0,\qquad n\geq 1.
  \]
  Then, every pair $(B,N)$ consisting of a DG algebra $B$ such that $\H{B}\cong
  A$ as graded algebras and a DG $B$-bimodule $N$ such that $\H{N}\cong M$ as
  graded $A$-bimodules is quasi-isomorphic to the pair $(A,M)$, that is the
  pairs $(B,N)$ and $(A,M)$ are weakly equivalent in the model category
  $\Bimod*{\Ass}$.
\end{theorem}

\begin{proof}
  Immediate from \Cref{Kadeishvili_simultaneous_A-infinity}, \Cref{minimal_models_quasi_isomorphic_algebras-bimodules} for $\O=\Ass$ and \Cref{quasi-iso_unit_algebras_bimodules}.
\end{proof}

\subsubsection{Fibre-wise formality}

The following result is a fibre-wise version of \Cref{Kadeishvili_operads_A-infinity}.

\begin{theorem}\label{Kadeishvili_fibre-wise_A-infinity}
	Let $\O$ be a graded operad with a multiplication $m_2$ in $\os\O$ and $\I\subset\O$ an associative operadic ideal. Denote the canonical projection by $q\colon\O\twoheadrightarrow\O/\I$. Suppose that the cohomology of the operadic ideal complex vanishes in the following range:
	\[\IH[n+1][-n]{\os\I}[\os\O]=0,\qquad n\geq1.\]
	Then, every map $g\colon \A\to\O$ with underlying multiplication $m_2^g=m_2$ such that $q(m_n^g)=0$ for $n\geq 3$ 
  is fibre-wise homotopic to the map $f\colon \A\to\O$ defined by the multiplication $m_2^f=m_2$
  and $m_n^f=0$ for $n\geq 3$, i.e.~they are homotopic through a homotopy $\A\to\O$ which composes to the trivial homotopy $\A\to\O\twoheadrightarrow\O/\I$, see \Cref{fibre-wise_homotopy}.
\end{theorem}

\begin{proof}
  The proof of \Cref{Kadeishvili_operads_A-infinity} carries over, replacing the
  spectral sequence in \Cref{operad_obstruction_theory} by that in
  \Cref{fibre-wise_obstruction_theory}.
\end{proof}

\subsubsection{Graded bimodules}

We now establish the following theorem for $\A$-bimodules over a \emph{fixed}
minimal $\A$-algebra. Notice that this is not an intrinsic formality result; it
rather states that, under certain assumptions, a  minimal $\A$-bimodules is
determined by their underlying graded bimodule, but the latter need not be an
$\A$-bimodule in general.

\begin{theorem}\label{Kadeishvili_bimodules_A-infinity}
	Let $(A,m^A)$ be a minimal $\A$-algebra and $(M,m^M)$ a
  minimal $\A$-bimodule over it. Suppose that the vector spaces of
  self-extensions of $M$ vanish in the following range:
	\[\Ext[n+1][-n]{A^e}{M}{M}=0,\qquad n\geq1.\]
	Then, any other minimal $\A$-bimodule $(M,\bar{m}^M)$ over $(A,m^A)$ with the same underlying
  graded $A$-bimodule as the above, $\bar{m}_2^M=m_2^M$, is gauge $\A$-isomorphic to $(M,m^M)$.
\end{theorem}

\begin{proof}
  Take $\O=\E{A,M}$ and $\I=\E*{A,M}$ in \Cref{Kadeishvili_fibre-wise_A-infinity}, see also \Cref{bimodule_obstruction_theory}.
\end{proof}

The following intrinsic formality result for graded bimodules over a \emph{fixed}
graded algebra corresponds to \Cref{Kadeishvili_bimodules} in the
introduction.

\begin{theorem}\label{Kadeishvili_bimodules-main-text}
	Let $A$ be a graded algebra and $M$ a graded $A$-bimodule. Suppose that the
  vector spaces of self-extensions of $M$ vanish in the following
  range:
  \[
    \Ext[n+1][-n]{A^e}{M}{M}=0,\qquad n\geq1.
  \]
	Then, every DG $A$-bimodule $N$ such that $\dgH{N}\cong M$ as graded
  $A$-bimodules is formal, that is $N$ quasi-isomorphic to $M$ where the latter
  is viewed as a DG $A$-bimodule with vanishing differential.
\end{theorem}

\begin{proof}
  Immediate from \Cref{Kadeishvili_bimodules_A-infinity}, \Cref{minimal_models_quasi_isomorphic_bimodules} for $\O=\Ass$ and \Cref{quasi-iso_unit_bimodules}.
\end{proof}

\subsection{Almost-formality results}

We now establish a generalisation of \Cref{Kadeishvili_operads_A-infinity}.

\begin{definition}\label{def:Massey_operad}
  We make the following definitions.
  \begin{enumerate}
	\item A \emph{Massey (graded) operad} is a graded operad $\O$ equipped with a
    multiplication $m_2\in\O<2>^1$ (\Cref{operad_multiplication}) and an operad cohomology class
    \[m_3\in\OH[3][-1]{\os\O},\]
    called \emph{universal Massey product}, satisfying $\Sq[m_3]=0$.
  \item The \emph{Massey operad complex} of a Massey operad is the bigraded (cochain) complex
    \begin{align*}
      \OC[s]{\os\O,m_3}&=\OH[s]{\os\O}&s&\geq 2\\
      \OC[s]{\os\O,m_3}&=0&s&<2,
    \end{align*}
    with bidegree $(2,-1)$ differential given by
    \[d\colon \OH[s][t]{\os\O}\longrightarrow\OH[s+2][t-1]{\os\O},\qquad d(x)=[m_3,x],\]
    except if $\cchar{\kk}=2$ and $(s,t)=(2,-1)$, in which case the differential is given by
    \[d\colon \OH[2][-1]{\os\O}\longrightarrow\OH[4][-2]{\os\O},\qquad
      d(x)=x^2+[m_3,x],\]
    compare with \Cref{operad_obstruction_theory}\eqref{second_differential}.
    The bigraded cohomology of this complex,
    \[\OH{\os\O,m_3}\coloneqq\H[\bullet,*]{\OC{\os\O,m_3}},\]
    is the \emph{Massey operad cohomology} of the Massey operad $\O$.
  \end{enumerate}
\end{definition}

\begin{remark}\label{rem:Massey_operad}
	The Massey operad complex is indeed a complex since, using the equations in \Cref{def:Gerstenhaber_square},
	\[d^2(x)=[m_3,[m_3,x]]=[\Sq[m_3],x]=[0,x]=0.\]
	We must also consider the special casen when $\cchar{\kk}=2$ and $x$ has bidegree $(s,t)=(2,-1)$. Then, using also the Gerstenhaber relation and graded commutativity,
	\[d^2(x)=[m_3,x^2+[m_3,x]]=[m_3,x^2]=2x[m_3,x]=0.\]
\end{remark}

\begin{remark}
	If $\cchar{\kk}\neq 2$, the Massey operad complex is a complex of bigraded vector spaces, and so are the Massey operad cohomology and the cocycles. In $\cchar{\kk}=2$, everything consists of vector spaces except possibly the cohomology in bidegrees $(2,-1)$, and $(4,-2)$ and the cocycles in bidegree $(2,-1)$.
\end{remark}

\begin{remark}\label{rmk:A-O_canonical_Massey_operad}
Let $f\colon \A[k+2]\to\O$ be a map of DG operads with $k\geq 3$. It follows
from \Cref{operad_obstruction_theory}\eqref{first_obstruction} that $\O$ is a
Massey operad with multiplication
\[
  m_2^f\in\O<2>^0=\os\O<2>^1
\]
and universal Massey product
\[
  \Hclass{m_3^f}\in\OH[3][-1]{\os\O};
\]
the third page of the (truncated) spectral sequence in \Cref{operad_obstruction_theory}\eqref{truncated_ss} satisfies 
	\[\BK[3][s][t]=\OH[s+2][-1]{\os\O,\Hclass{m_3^f}},\qquad s\geq 1,\qquad t\in\ZZ,\]
	where we also use \Cref{operad_obstruction_theory}\eqref{second_differential}.
\end{remark}

The following result is a vast generalisation of \cite[Theorem~B]{JKM22}.

\begin{theorem}\label{secondary_Kadeishvili_operads_A-infinity_uniqueness}
	Let $\O$ be a graded operad and $f\colon\A\to\O$ a morphism of DG operads, so
  that $\O$ is equipped with the multiplication $m_2^f$ and the universal Massey
  product $\Hclass{m_3^f}$ (\Cref{rmk:A-O_canonical_Massey_operad}). Suppose that the corresponding Massey operad
  cohomology vanishes in the following range:
	\[\OH[n+2][-n]{\os\O,\Hclass{m_3}}=0,\qquad n>1.\]
	Then, every map $g\colon \A\to\O$ inducing the same multiplication
  $m_2^g=m_2^f$ on $\os\O$ and the same universal Massey product,
  \[
    \Hclass{m_3^g}=\Hclass{m_3^f}\in\OH[3][-1]{\os\O},
  \]
  is homotopic to $f$.
\end{theorem}

\begin{proof}
  The argument is entirely analogous to the one used to prove
  \Cref{Kadeishvili_operads_A-infinity}, using also the spectral sequence
  \Cref{operad_obstruction_theory} associated to the map $f\colon\A\to\O$. It
  suffices to observe that, by assumption and using \Cref{operad_obstruction_theory}\eqref{second_differential},
  \[
    \BK[3][s][s]=\OH[s+2][-s]{\os\O,\Hclass{m_3^f}}=0,\qquad s>1,
  \]
  and hence also
  \[
    \BK[r][s][s]=0,\qquad r\geq3,\quad s>1.
  \]
  We prove by induction on $s\geq 0$ that the restrictions of $f$ and $g$ to
  $\A[s+2]$ are homotopic. For $s=0$ this corresponds to the fact that the
  restrictions of $f$ and $g$ to $\A[2]$ are equal since $m_2^f=m_2^g$ by
  assumption; the case $s=1$ follows from
  \Cref{operad_obstruction_theory}\eqref{lower-diagonal_second_page}. Assuming
  the result for $s-1$, the case $s$ follows from
  \Cref{operad_obstruction_theory}\eqref{lower-diagonal} since
  $\BK[s+1][s][s]=0$. The rest is exactly as in the proof of
  \Cref{Kadeishvili_operads_A-infinity}.
\end{proof}

\begin{remark}
  \label{rmk:Massey_implies_Kadeishvili}
  Notice that \Cref{Kadeishvili_operads_A-infinity} is a direct consequence of
  \Cref{secondary_Kadeishvili_operads_A-infinity_uniqueness}: Under the
  assumptions of \Cref{Kadeishvili_operads_A-infinity}, we have
  \[
    \Hclass{m_3^f}\in\OH[3][-1]{\os\O}=0
  \]
  and therefore
  \[
    \OH[n+2][-n]{\os\O,\Hclass{m_3^f}}=\OH[n+2][-n]{\os\O,0}=\OH[n+2][-n]{\os\O}=0,\qquad n>1.
  \]
  The conclusion in \Cref{Kadeishvili_operads_A-infinity} therefore follows
  by applying \Cref{secondary_Kadeishvili_operads_A-infinity_uniqueness}
  (compare with the proof of \cite[Corollary~5.2.8]{JKM22}).
\end{remark}

The following existence result is of independent interest (compare with~\cite[Theorem~5.1.2]{JKM22}).

\begin{theorem}\label{secondary_Kadeishvili_operads_A-infinity_existence}
	Let $\O$ be a graded operad and $f\colon\A[5]\to\O$ a morphism of DG operads
  with multiplication $m_2^f$ and universal Massey product $\Hclass{m_3^f}$ (\Cref{rmk:A-O_canonical_Massey_operad}).
  Suppose that the corresponding Massey operad cohomology vanishes in the
  following range:
	\[\OH[n+3][-n]{\os\O,\Hclass{m_3}}=0,\qquad n>1.\]
	Then, there exists a map $g\colon \A\to\O$ inducing the same multiplication
  $m_2^g=m_2^f$ on $\os\O$ and the same universal Massey product as $f$,
  \[
    \Hclass{m_3^g}=\Hclass{m_3^f}\in\OH[3][-1]{\os\O}.
  \]
\end{theorem}

\begin{proof}
	We show by induction the existence of maps \[h_s\colon\A[s+3]\to\O,\qquad s\geq 0,\] such that:
	\begin{itemize}
    \item $h_0$ and $h_1$ are the restrictions of $f$ to $\A[3]$ and $\A[4]$, respectively.
		\item Each $h_s$ extends to $\A[s+4]$.
		\item The restrictions of $h_s$ and $h_{s-1}$ to $\A[s+1]$ coincide:
      \[
        \begin{tikzcd}
          \A[s+1]\rar[hook]\dar[hook]\ar[phantom]{dr}[description]{=}&\A[s+3]\dar{h_s}\\
          \A[s+2]\rar[swap]{h_{s-1}}&\O
        \end{tikzcd}
      \]
	\end{itemize}
	There is nothing to prove for $s=0,1$. Let $s\geq 2$ and suppose that we have
  proved the claim up to $s-1$. Let $\tilde{h}_s\colon\A[s+3]\to\O$ be an
  extension of $h_{s-1}\colon\A[s+2]\to\O$, whose existence is part of the
  claim for $s-1$. We consider the truncated spectral sequence in
  \Cref{operad_obstruction_theory} associated to $\tilde{h}_s$. By
  \Cref{operad_obstruction_theory}\eqref{second_differential} and
  \Cref{rmk:A-O_canonical_Massey_operad}, the vanishing hypothesis yields
  \[
    \BK[3][s+1][s]=\OH[s+3][-s]{\os\O,\Hclass{m_3^f}}=0,\qquad s>1.
  \]
  By \Cref{operad_obstruction_theory}\eqref{obstructions} for $k=s+1$ and
  $l=s-1$, since $s\geq 2$, there exists a map
  $\tilde{h}_{s+1}\colon\A[s+4]\to\O$ whose restriction to $\A[s+1]$ coincides
  with the restriction of $\tilde{h}_s$, that is with $h_{s-1}$. The claim for
  $s$ follows by defining $h_s\colon\A[s+3]\to\O$ as the restriction of $\tilde{h}_{s+1}$.
	
	Consider the sequence of maps $g_s\colon\A[s+2]\to\O$, $s\geq 0$, defined as
  the restrictions of $h_s\colon\A[s+3]\to\O$ to $\A[s+2]$. In this case, $g_1$
  is the restriction of $f$ to $\A[3]$ and, for $s\geq1$, the restriction of
  $g_s$ to $\A[s+1]$ is $g_{s-1}$. Hence, the above sequence yields a point $g\colon \A\to\O$ of the homotopy limit
  \[
    \Map{\A}{\O}\simeq\operatorname{holim}_{s\geq0}\Map{\A[s+2]}{\O},
  \]
  whose restriction to $\A[s+2]$ is $g_s$
  for all $s\geq 0$ (compare with~\cite[Section~IX.3.1]{BK72}).
	By \Cref{operad_obstruction_theory}\eqref{lower-diagonal_second_page}, the map
  $g$ induces the same multiplication, $m_2^g=m_2^f$, and the same universal
  Massey product as $f$,
  \[
    \Hclass{m_3^g}=\Hclass{m_3^f}\in\OH[3][-1]{\os\O}.
  \]
  This finishes the proof.
\end{proof}

\begin{remark}
  The properties of the map $g$ in \Cref{secondary_Kadeishvili_operads_A-infinity_existence} are equivalent to saying that 
  the restrictions of the maps $f$ and $g$ to $\A[3]$ are homotopic,
  see \Cref{operad_obstruction_theory}\eqref{lower-diagonal_second_page}.
\end{remark}

\subsubsection{Graded algebras}

\begin{definition}\label{def:Massey_algebra}
  We make the following definitions.
  \begin{enumerate}
  \item
    A \emph{Massey (graded) algebra} is a pair $(A,m_3)$ consisting of a graded algebra $A$ and a Hochschild cohomology class
    \[m_3\in\HH[3][-1]{A},\]
    called \emph{universal Massey product}, satisfying $\Sq[m_3]=0$.
	\item The \emph{Hochschild-Massey complex} of a Massey algebra is the bigraded
    (cochain) complex
    \begin{align*}
      \HMC[s]{A}[m_3]&=\HH[s]{A}&s&\geq 2,\\
      \HMC[s]{A}[m_3]&=0&s&<2,
    \end{align*}
    with the bidegree $(2,-1)$ differential given by
    \[d\colon \HH[s][t]{A}\longrightarrow\HH[s+2][t-1]{A},\qquad d(x)=[m_3,x],\]
    except if $\cchar{\kk}=2$ and $(s,t)=(2,-1)$, in which case the differential is given by
    \[d\colon \HH[2][-1]{A}\longrightarrow\HH[4][-2]{A},\qquad d(x)=x^2+[m_3,x].\]
    The bigraded cohomology of this complex,
    \[\HMH{A}[m_3]\coloneqq\H[\bullet,*]{\HMC{A}[m_3]},\]
    is called \emph{Hochschild--Massey cohomology} of the Massey algebra $(A,m_3)$.
  \end{enumerate}
\end{definition}

A Massey algebra structure on a graded vector space $A$ is the same as a Massey
operad structure on its endomorphism operad $\E{A}$. Hence,
\Cref{secondary_Kadeishvili_operads_A-infinity_uniqueness,secondary_Kadeishvili_operads_A-infinity_existence}
have the following immediate consequences.

\begin{theorem}\label{secondary_Kadeishvili_algebras_A-infinity_uniqueness}
	Let $(A,m^A)$ be a minimal $\A$-algebra. Suppose that the Hochschild--Massey cohomology of
  the pair $(A,\Hclass{m_3^A})$ vanishes in the following range:
	\[\HMH[n+2]<-n>{A}[\Hclass{m^A_3}]=0,\qquad n>1.\]
	Then, every minimal $\A$-algebra $(A,\bar{m}^A)$ with the same underlying
  graded algebra, $\bar{m}_2^A=m_2^A$, and the same universal Massey product,
  \[
    \Hclass{\bar{m}^A_3}=\Hclass{m^A_3},
  \]
  is gauge $\A$-isomorphic to $(A,m^A)$.
\end{theorem}
\begin{proof}
  Take $\O=\E{A}$ in \Cref{secondary_Kadeishvili_operads_A-infinity_uniqueness}.
\end{proof}

The following result corresponds to \Cref{thm:B} in the introduction. Below, the
by the universal Massey product of a DG algebra we mean the universal Massey
product of any of its minimal models (that this is well defined follows
from \Cref{rmk:UMP-gauge_invariant} and
\Cref{minimal_models_quasi_isomorphic_algebras}).

\begin{theorem}
  \label{thm:B-main_text}
  Let $A$ be a DG algebra. Choose a minimal model $(\H{A},\Astr)$ of $A$ and
  suppose that the Hochschild--Massey cohomology of the Massey algebra
  $(\H{A},\Hclass{\Astr<3>[A]})$ vanishes in the following
  range:\footnote{Notice the strict inequality.}
  \[
    \HMH[n+2]<-n>{\H{A}}[\Hclass{\Astr<3>}],\qquad n>1.
  \]
  Then, every DG algebra $B$ such that $\H{B}\cong\H{A}$ as graded algebras and
  whose universal Massey product satisfies
  \[
    \Hclass{\Astr<3>[B]}=\Hclass{\Astr<3>[A]}\in\HH[3][-1]{\H{A}}
  \]
  is quasi-isomorphic to $A$.
\end{theorem}
\begin{proof}
  Immediate from \Cref{secondary_Kadeishvili_algebras_A-infinity_uniqueness}
  with $\O=\E{A}$, \Cref{minimal_models_quasi_isomorphic_algebras} for $\O=\Ass$ and \Cref{quasi-iso_unit_algebras}.
\end{proof}

The following existence result is of independent interest.

\begin{theorem}\label{secondary_Kadeishvili_algebras_A-infinity_existence}
	Let $(A,m_3)$ be a Massey algebra. Suppose that the Hochschild--Massey
  cohomology vanishes in the following range:
	\[\HMH[n+3]<-n>{A}[m_3]=0,\qquad n>1.\]
	Then, there exists a minimal $\A$-algebra $(A,m^A)$ with $\Hclass{m^A_3}=m_3$.
\end{theorem}
\begin{proof}
  Take $\O=\E{A}$ in \Cref{secondary_Kadeishvili_operads_A-infinity_existence}.
\end{proof}

\subsubsection{Graded algebra-bimodule pairs}

\begin{definition}\label{def:Massey_bimodule}
  Let $A$ be a graded algebra. We make the following definitions.
  \begin{enumerate}
  \item A \emph{Massey (graded) $A$-bimodule} is a pair $(M,m_3)$ consisting of
    an $A$-bimodule $M$ and a bimodule Hochschild cohomology class
    \[m_3\in\HHE[3][-1]{A}{M},\]
    called \emph{bimodule universal Massey product}, satisfying $\Sq[m_3]=0$.
	\item The \emph{bimodule Hochschild complex} of a Massey $A$-bimodule is the
    bigraded (cochain) complex
	\[\AlgBimHMC*!A![s]{M}[m_3]=\HHE[s]{A}{M},\quad s\geq 2;\qquad \AlgBimHMC*!A![s]{M}[m_3]=0,\quad s<2;\]
	with the bidegree $(2,-1)$ differential given by
	\[d\colon \HHE[s][t]{A}{M}\longrightarrow\HHE[s+2][t-1]{A}{M},\qquad d(x)=[m^{M},x],\]
	except if $\cchar{\kk}=2$ and $(s,t)=(2,-1)$, in which case the differential is given by
	\[d\colon \HHE[2][-1]{A}{M}\longrightarrow\HHE[4][-2]{A}{M},\qquad d(x)=x^2+[m^{M},x].\]
	The bigraded cohomology of this complex,
  \[\AlgBimHMH*!A!{M}[m_3]\coloneqq\H[\bullet,*]{\AlgBimHMC*!A!{M}[m_3]},\]
  is called \emph{bimodule Hochschild--Massey cohomology} of the triple
  $(A,M,m_3)$.
  \end{enumerate}
\end{definition}

\begin{remark}
Let $A$ be a graded algebra and $(M,m_3)$ a Massey $A$-bimodule. Recall the map
$p_*\colon\HHE{A}{M}\to\HH{A}$ in \eqref{long_exact_sequence_Hochschild}. The pair $(A,p_*(m_3))$ is a Massey algebra.
\end{remark}

Given graded vector spaces $A,M$, a Massey operad structure on $\E{A,M}$ is the
same as a graded algebra structure in $A$ and a Massey $A$-bimodule structure on $M$. Therefore, \Cref{secondary_Kadeishvili_operads_A-infinity_uniqueness,secondary_Kadeishvili_operads_A-infinity_existence} have the following immediate consequences.

\begin{theorem}\label{secondary_Kadeishvili_simultaneous_A-infinity_uniqueness}
	Let $(A,m^A)$ be a minimal $\A$-algebra and $(M,m^M)$ a minimal
  $\A$-bimodule over it. Suppose that the bimodule Hochschild--Massey cohomology
  vanishes in the following range:
	\[\AlgBimHMH*!A![n+2]<-n>{M}[\Hclass{m_3^A+m_3^M}]=0,\qquad n>1.\]
	Then, every other pair given by a minimal $\A$-algebra $(A,\bar{m}^A)$  with the same
  underlying graded algebra, $\bar{m}_2^A=m_2^A$, and a
  minimal $\A$-bimodule $(M,\bar{m}^M)$ over it with the same
  underlying graded $A$-bimodule, $\bar{m}_2^M=m_2^M$, and the same bimodule
  universal Massey product,
  \[
    \Hclass{\bar{m}^A_3+\bar{m}^M_3}=\Hclass{m^A_3+m^M_3}\in\HHE[3][-1]{A}{M},
  \]
  is gauge $\A$-isomorphic to the pair formed by $(A,m^A)$ and $(M,m^M)$.
\end{theorem}
\begin{proof}
  Take $\O=\E{A,M}$ in \Cref{secondary_Kadeishvili_operads_A-infinity_uniqueness}.
\end{proof}

The following result corresponds to
\Cref{secondary_Kadeishvili_simultaneous_DG_uniqueness} in the introduction. Below, the
by the bimodule universal Massey product of a DG bimodule we mean the bimodule universal Massey
product of any of its minimal models (that this is well defined follows
from \Cref{rmk:UMP-gauge_invariant} and
\Cref{minimal_models_quasi_isomorphic_algebras-bimodules}).

\begin{theorem}\label{secondary_Kadeishvili_simultaneous_DG_uniqueness-main_text}
  Let $A$ be a DG algebra and $M$ a DG $A$-bimodule. Choose a minimal model
  $(\H{A},\Astr[A])$ of $A$ and a compatible minimal model $(\H{M},\Astr[M])$ of
  $M$. Let
  \[
    \Astr<3>[A\ltimes M]\coloneqq\Astr<3>[A]+\Astr<3>[M]\in\RelBimHH[3]<-1>{\H{A}}{\H{M}}
  \]
  and suppose that the bimodule
  Hochschild--Massey cohomology of the Massey $\H{A}$-bimodule
  \[
    (\H{M},\Hclass{\Astr<3>[A\ltimes M]})
  \]
  vanishes in the following range:
  \[
    \AlgBimHMH!\H{A}![n+2]<-n>{\H{M}}[3][A\ltimes M]=0,\qquad n>1.
  \]
  Then, every pair $(B,N)$ consisting of
  \begin{itemize}
  \item a DG algebra $B$ such that $\H{B}\cong\H{A}$ as graded algebras and
  \item a DG $B$-bimodule $N$ such that $\H{N}\cong\H{M}$ as graded
    $\H{A}$-bimodules and
  \item whose bimodule universal
    Massey product satisfies
    \[
      \Hclass{\Astr<3>[B\ltimes N]}=\Hclass{\Astr<3>[A\ltimes
        M]}\in\RelBimHH[3]<-1>{\H{A}}{\H{M}}
    \]
  \end{itemize}
  is quasi-isomorphic to the pair $(A,M)$.
\end{theorem}
\begin{proof}
  Immediate from \Cref{secondary_Kadeishvili_simultaneous_A-infinity_uniqueness}, \Cref{minimal_models_quasi_isomorphic_algebras-bimodules} for $\O=\Ass$ and \Cref{quasi-iso_unit_algebras_bimodules}.
\end{proof}

The following existence result is of independent interest.

\begin{theorem}\label{secondary_Kadeishvili_simultaneous_A-infinity_existence}
	Let $A$ be a graded algebra and $(M,m^M)$ be a Massey $A$-bimodule. Suppose
  that the bimodule Hochschild--Massey bimodule cohomology vanishes in the
  following range:
	\[\AlgBimHMH*!A![n+3]<-n>{M}[\Hclass{m_3^A+m_3^M}]=0,\qquad n>1.\]
	Then, there exists a minimal $\A$-algebra $(A,m^A)$ and a minimal
  $\A$-bimodule $(M,m^M)$ over it with
  \[
    \Hclass{m^A_3+m^M_3}=m^M\in\HHE[3][-1]{A}{M}.
  \]
\end{theorem}
\begin{proof}
  Take $\O=\E{A,M}$ in \Cref{secondary_Kadeishvili_operads_A-infinity_existence}.
\end{proof}

\subsubsection{Fibre-wise almost-formality}

\begin{definition}
  Let $\O$ be a Massey graded operad with universal Massey product
  $m_3\in\OH[3][-1]{\os\O}$ and let $\I\subset\O$ be an associative operadic ideal. The \emph{Massey operadic ideal complex} is the bigraded (cochain) complex
  \begin{align*}
    \IC[s]{\os\I,m_3}[\os\O]&=\IH[s]{\os\I}[\os\O]&s&\geq 2,\\
    \IC[s]{\os\I,m_3}[\os\O]&
    =0&s&<2,
  \end{align*}
  with the bidegree $(2,-1)$ differential given by
  \[d\colon \IH[s][t]{\os\I}[\os\O]\longrightarrow\IH[s+2][t-1]{\os\I}[\os\O],\qquad d(x)=[m_3,x],\]
  compare with \Cref{fibre-wise_obstruction_theory}\eqref{second_differential_fibre-wise}.
  The bigraded cohomology of this complex,
  \[\IH{\os\I,m_3}[\os\O]\coloneqq\H[\bullet,*]{\IC{\os\I,m_3}[\os\O]},\]
  is called \emph{Massey operadic ideal cohomology}.
\end{definition}

\begin{theorem}\label{secondary_Kadeishvili_ideals_A-infinity_uniqueness}
	Let $\O$ be a graded operad and $f\colon\A\to\O$ a morphism of DG operads, so
  that $\O$ is equipped with the multiplication $m_2^f$ and the universal Massey
  product $\Hclass{m_3^f}$ (\Cref{rmk:A-O_canonical_Massey_operad}). 
  Let $\I\subset\O$ be an associative operadic ideal and $q\colon\O\twoheadrightarrow\O/\I$ the canonical projection.
  Suppose that the corresponding Massey operadic ideal
  cohomology vanishes in the following range:
	\[\IH[n+1][-n]{\os\I,\Hclass{m_3^f}}[\os\O]=0,\qquad n>1.\]
	Then, every map $g\colon \A\to\O$ inducing the same multiplication
  $m_2^g=m_2^f$ on $\os\O$ and such that 
  \[
    \Hclass{m_3^g-m_3^f}\in\IH[2][-1]{\os\I}[\os\O],
  \]
  and $qf=qg$ 
  is fibre-wise homotopic to $f$, i.e.~$f$ and $g$ are homotopic through a homotopy $\A\to\O$ which composes to the trivial homotopy
  $\A\to\O\twoheadrightarrow\O/\I$.
\end{theorem}

\begin{proof}
  The proof is identical to that of
  \Cref{secondary_Kadeishvili_operads_A-infinity_uniqueness}, replacing the references to \Cref{operad_obstruction_theory} with \Cref{fibre-wise_obstruction_theory}.
\end{proof}

\subsubsection{Graded bimodules}

\begin{definition}
	Let $A$ be a graded algebra and $(M,m_3)$ a Massey $A$-bimodule. The
  \emph{Massey bimodule complex} is the bigraded (cochain) complex
	\begin{align*}
    \BimHMC*!A![s]{M}[m_3]&=\Ext[s]{A^e}{M}{M}&s&\geq 1,\\
    \BimHMC*!A![s]{M}[m_3]&=0&s&<1,
  \end{align*}
	with the bidegree $(2,-1)$ differential given by
	\[d\colon \Ext[s][t]{A^e}{M}{M}\longrightarrow\Ext[s+2][t-1]{A^e}{M}{M},\qquad d(x)=[m_3,x].\]
	The bigraded cohomology of this complex,
  \[
    \BimHMH*!A!{M}[m_3][]\coloneqq\H[\bullet,*]{\BimHMC*!A!{M}[m_3]},
  \]
  is the \emph{Massey bimodule cohomology} of the pair $(M,m_3)$.
\end{definition}

\begin{theorem}\label{secondary_Kadeishvili_bimodules_A-infinity}
	Let $(A,m^A)$ be a minimal $\A$-algebra and $(M,m^M)$ a minimal
  $\A$-bimodule over it. Suppose that Massey bimodule cohomology vanishes in the
  following range:
	\[\BimHMH*!A![n+1]<-n>{M}[\Hclass{m_3^A+m_3^M}]=0,\qquad n>1.\]
	Then, every minimal $\A$-bimodule $(M,\bar{m}^M)$ over $(A,m^A)$ with the
  same underlying graded $A$-bimodule, $\bar{m}_2^M=m_2^M$, and such that
  \[
    \Hclass{\bar{m}^M_3-m^M_3}=0\in\Ext[2][-1]{A^e}{M}{M}
  \]
	is gauge $\A$-isomorphic to $(M,m^M)$.
\end{theorem}
\begin{proof}
  Take $\O=\E{A,M}$ and $\I=\E*{A,M}$ in \Cref{secondary_Kadeishvili_ideals_A-infinity_uniqueness}.
\end{proof}

\begin{remark}
	It follows from the existence of the long exact sequence
  \eqref{long_exact_sequence_Hochschild} and \Cref{connecting_morphism} that the condition
  \[
    \Hclass{\bar{m}^M_3-m^M_3}=0\in\Ext[2][-1]{A^e}{M}{M}
  \]
  in \Cref{secondary_Kadeishvili_bimodules_A-infinity} is equivalent to the
  condition
  \[
    \Hclass{m^A_3+\bar{m}^M_3}=\Hclass{m^A_3+m^M_3}\in\HHE[3][-1]{A}{M}
  \]
  whenever $\id*[M]\cdot x=x\cdot\id*[M]$ for all $x\in \HH[2][-1]{A}$. This
  happens, for instance, if $\Ext{A^e}{M}{M}$ is a symmetric $\HH{A}$-bimodule,
  see \Cref{cor:symmetric_bimodule}.
\end{remark}

The following result corresponds to \Cref{thm:B_bimodules} in the introduction.

\begin{theorem}\label{thm:B_bimodules-main_text}
  Let $A$ be a DG algebra and $M$ a DG $A$-bimodule. Choose a minimal model
  $(\H{A},\Astr)$ of $A$, a compatible minimal model $(\H{M},\Astr[M])$
  of $M$ over $(\H{A},\Astr)$, and suppose that the Massey bimodule
  cohomology of the pair $(\H{M},\Hclass{\Astr<3>[A\ltimes M]})$ vanishes in the
  following range:
  \[
    \BimHMH!\H{A}![n+1]<-n>{\H{M}}[3][A\ltimes M]=0,\qquad n>1.
  \]
  Then, every DG $A$-bimodule $N$ such that $\H{N}\cong\H{M}$ as graded
  $\H{A}$-bimodules and such that
  \[
    \Hclass{\Astr<3>[M]-\Astr<3>[N]}=0\in
    \BimHMH!\H{A}![2]<-1>{\H{M}}[3][A\ltimes M]
  \]
  is quasi-isomorphic to $M$.
\end{theorem}
\begin{proof}
  Immediate from \Cref{secondary_Kadeishvili_bimodules_A-infinity}, \Cref{minimal_models_quasi_isomorphic_bimodules} for $\O=\Ass$ and \Cref{quasi-iso_unit_bimodules}.
\end{proof}

%%%%%

\section{The sparse case}
\label{sec:sparse}

In this section we establish $d$-sparse analogues of
\Cref{secondary_Kadeishvili_operads_A-infinity_uniqueness,secondary_Kadeishvili_operads_A-infinity_existence,secondary_Kadeishvili_bimodules_A-infinity}.
Our motivation for considering these results stems from the applications
discussed in \cite{JKM22,JKM24,JM25}.

\subsection{Obstruction theories in the $d$-sparse case}

Fix an integer $d\geq1$.

\begin{definition}
  \label{def:sparse}
  We make the following definitions:
  \begin{enumerate}
  \item A graded vector space $V$ is \emph{$d$-sparse} if $V^i=0$
    whenever $i\not\in d\ZZ$. Similarly, a DG vector space $V$ is
    \emph{cohomologically $d$-sparse} if its cohomology $\H{V}$ is $d$-sparse.
  \item A graded operad $\O$ is \emph{$d$-sparse} if its graded vector spaces of
    operations $\O<n>$, $n\geq0$, are $d$-sparse.
  \end{enumerate}
  Obviously, if $d=1$ then the above notions impose no conditions.
\end{definition}

\begin{example}
  \label{ex:EV-sparse}
  Let $V$ be a $d$-sparse graded vector space. Then, the endomorphism operad
  $\E{V}$ is $d$-sparse. Indeed, for degree reasons, a homogeneous morphism
  \[
    V^{\otimes n}\longrightarrow V,\qquad n\geq0,
  \]
  of degree $i\not\in d\ZZ$ must be identically zero.
\end{example}

\begin{example}
  Let $V$ and $W$ be $d$-sparse graded vector spaces. Then, the linear
  endomorphism operad $\E{V,W}$ is $d$-sparse. This is a consequence of
  \Cref{ex:EV-sparse} and the fact that $\E{V,W}$ is a suboperad of the
  endomorphism operad $\E{V\oplus W}$ of the $d$-sparse graded vector space
  $V\oplus W$.
\end{example}

\begin{remark}
  \label{rmk:sparse_maps}
  Let $\O$ be a $d$-sparse graded operad for some $d\geq 2$ and $m^A\colon\A[k]\to\O$ a morphism of DG operads
  for some $\infty\geq k\geq 2$. We make the following observations:
  \begin{enumerate}
  \item For degree reasons, the operations
    \[m_n\in\OC[n][2-n]{\os\O}=\O<n>^{2-n},\qquad 2\leq n\leq k,\]
    vanish whenever $n-2\not\in d\ZZ$. Moreover, the equations satisfied by these operations
    \[\sum_{p+q=n+2}m_p\{m_q\}=0,\qquad 2\leq n\leq k-1,\]
    are tautological except when $n-2\in d\ZZ$. Hence, maps $\A[k]\to\O$ are only
    meaningful for $k=di+2,di+3$, $i\geq 0$ (in other cases such a map is equivalent
    to the closest meaningful one below).
  \item  For $d\geq 2$, both a map
    $\A[di+2]\to\O$ and a map $\A[di+3]\to\O$ consist of operations
    \[
      m_{dj+2}\in\O<dj+2>^{-dj},\qquad 0\leq j\leq i;
    \]
    however, these operations must satisfy one more equation in the case of a map
    $\A[di+3]\to\O$.
  \item Every map $\A[di+3]\to\O$ extends to
    $\A[d(i+1)+2]$ (for example by taking $m_{d(i+1)+2}=0$ or, indeed, any other choice).
  \item The operad complex $\OC{\os\O}$ is vertically $d$-sparse (that is,
    $d$-sparse with respect to the vertical degree), and so is its cohomology
    $\OH{\os\O}$. The (truncated) spectral sequence $\BK$ in
    \Cref{operad_obstruction_theory} is also vertically $d$-sparse. Therefore, the
    only possibly non-trivial spectral sequence differentials are $\ssd{di+1}$,
    $i\geq 0$, and the different pages are those of the form $\BK[1]$ and
    $\BK[di+2]$, $i\geq 0$.
  \end{enumerate}
\end{remark}

\begin{definition}
  Let $\O$ be a $d$-sparse graded operad and $m^A\colon\A[k]\to\O$ a morphism of
  DG operads. If $\infty\geq k\geq d+3$, then the operation $m_{d+2}\in\OC[d+2][-d]{\os\O}$
  is a cocycle of the operad complex. Its cohomology class,
  \[\Hclass{m_{d+2}}\in\OH[d+2][-d]{\os\O},\]
  is the \emph{universal Massey product of length $d+2$}. Similar to the case
  $d=1$, this cohomology class is a homotopy invariant.
\end{definition}

The following result is a `$d$-sparse variant' of \Cref{operad_obstruction_theory}.

\begin{theorem}
	\label{sparse_spectral_sequence_differential_operad}
	Let $d\geq 2$. Let $\O$ be a $d$-sparse graded operad equipped with a morphism
  of DG operads ${f\colon \A[k+2]\to\O}$ with $\infty\geq k\geq 2$. The
  following statements concerning the (truncated) spectral sequence in
  \Cref{operad_obstruction_theory} hold (they replace the corresponding
  statements therein):
	\begin{enumerate}\setcounter{enumi}{11}
	\item\label{lower-diagonal_second_page_sparse} Let $k\geq d+1$, $r=d+1$ and $s=d$,
    and let
    \[
      \Hclass{m_{d+2}^f}\in\OH[d+2][-d]{\os\O}
    \]
    be the universal Massey product of $f\colon\A[k]\to\O$.
    \Cref{lower-diagonal,second_page} define a pointed bijection between the
    pointed set of homotopy classes of maps $g\colon\A[d+2]\to\O$ which extend
    to $\A[2d+2]$ whose restriction to $\A[d+1]$ is homotopic to the restriction
    of $f$,
    \[
      \begin{tikzcd}
        \A[d+1]\rar[hook]\dar[hook]\ar[phantom]{dr}[description]{\Rightarrow}&\A[d+2]\dar{g}\rar[hook]&\A[2d+2]\dlar[dotted]\\
        \A[k+2]\rar[swap]{f}&\O,
      \end{tikzcd}
    \]
    and the pointed set $\BK[d+1][d][d]=\OH[d+2][-d]{\os\O}$. The base point in
    the source is the restriction of $f$ to $\A[d+2]$ and the base point in the
    target is $0$. The bijection maps the homotopy class of $g\colon\A[d+2]\to\O$ to $\Hclass{m_{d+2}^g}-\Hclass{m_{d+2}^f}$.
    \item\label{second_differential_operad_sparse} If $k\geq 2d+1$, the
      differential
      \[
        \ssd{d+1}\colon\BK[d+1][s][t]\longrightarrow\BK[d+1][s+d+1][t+d]
      \]
      of the (truncated) spectral sequence $\BK$ in
      \Cref{operad_obstruction_theory} is given by the Gerstenhaber Lie bracket
      with the universal Massey product of length $d+2$,
      \begin{align*}
        \ssd{d+1}\colon\OH[s+2][-t]{\os\O}&\longrightarrow\OH[s+d+3][-t-d]{\os\O}&s\geq 1,\\
        x&\longmapsto[\Hclass{m_{d+2}^f},x],
      \end{align*}
      possibly composed with the projection of cocycles onto cohomology,
      \begin{align*}
        \ssd{d+1}\colon\OZ[2][-t]{\os\O}&\longrightarrow\OH[d+3][-t-d]{\os\O}&t>d\\
        x&\longmapsto[\Hclass{m_{d+2}^f},\Hclass{x}],
      \end{align*}
      and, finally, by
      \begin{align*}
        \ssd{d+1}\colon\OH[d+1][-d]{\os\O}&\longrightarrow\OH[2d+2][-2d]{\os\O}\\
        x&\longmapsto x^2+[\Hclass{m_3},x].
      \end{align*}
    \item\label{first_obstruction_operad_sparse} For $k= 2d$ and $\ell=d$,
      the obstruction in \Cref{obstructions} to the existence of a map
      $g\colon\A[2d+3]\to\O$ such that the restrictions of $f$ and $g$ to
      $\A[d+2]$ coincide,
      \[
        \begin{tikzcd}
          \A[d+2]\dar[hook]\rar[hook]\ar[phantom]{dr}[description]{=}&\A[2d+3]\dar[dotted]{g}\\
          \A[2d+2]\rar[swap]{f}&\O
        \end{tikzcd}
      \]
      is the Gerstenhaber square of the universal Massey product
      \[
        \Sq[\Hclass{m_{d+2}^f}]\in\OH[2d+3][-2d]{\os\O}=\BK[d+1][2d+1][2d].
      \]
    \item\label{primary_obstruction_operad_sparse} More generally, for $i\geq 1$
      and $k= di$, the obstruction in \Cref{obstructions} to the existence
      of a map $g\colon\A[di+3]\to\O$ such that the restrictions of $f$ and $g$ to
      $\A[d(i-1)+2]$ coincide,
      \[
        \begin{tikzcd}
          \A[d(i-1)+2]\dar[hook]\rar[hook]\ar[phantom]{dr}[description]{=}&\A[di+3]\dar[dotted]{g}\\
          \A[di+2]\rar[swap]{f}&\O,
        \end{tikzcd}
      \]
      is the cohomology class in
      \[
        \OH[di+3][-di]{\O}=\BK[d+1][di+1][di]
      \]
      represented by the operad
      cocycle \[\sum_{\substack{p+q=di+4\\p,q>d+1}}m_p^f\{m_q^f\}\in\OC[di+3][-di]{\O}.\]
	\end{enumerate}
\end{theorem}

\begin{proof}
	\Cref{second_differential_operad_sparse} is a generalisation of
  \cite[Proposition~5.2.2]{JKM22} from endomorphism operads of $d$-sparse graded
  vector spaces to arbitrary $d$-sparse graded operads $\O$.
	\Cref{first_obstruction_operad_sparse} is a consequence of
  \cite[Proposition~3.4]{Mur20}, as noticed in \cite[Section~5]{JKM22} in case
  $\O=\E{V}$ with $V$ a $d$-sparse graded vector space.
	\Cref{primary_obstruction_operad_sparse} also follows from \cite[Proposition
  3.4]{Mur20}, and it admits an elementary proof similar to that of the
  corresponding statement in \Cref{operad_obstruction_theory}.
\end{proof}

\begin{remark}
Keeping \Cref{rmk:sparse_maps} in mind, \Cref{table:EA,table:EAM} give the
substitutions that are needed to apply
\Cref{sparse_spectral_sequence_differential_operad} when the target graded
operad $\O$ is the endomorphism operad of a $d$-sparse graded vector space or the linear
endomorphism operad of a pair of $d$-sparse graded vector spaces, respectively.
\end{remark}

The following result is a `$d$-sparse variant' of \Cref{fibre-wise_obstruction_theory}.

\begin{theorem}
	\label{sparse_spectral_sequence_differential_bimodule}
	Let $d\geq 2$. Let $\O$ be a $d$-sparse graded operad equipped with a
  ($d$-sparse) operadic ideal $\I\subset\O$ and a
  map $h\colon\A\to\O/\I$. Suppose that we have a map $f\colon\A[k+2]\to\O$ for some
  $\infty\geq k\geq d$ such that $qf\colon\A[k+2]\to\O/\I$ is the restriction of
  $h\colon\A\to\O/\I$ to $\A[k+2]$, that is $q(m^f_i)=m^h_i$, $2\leq i\leq k+2$:
  \[
  \begin{tikzcd}
    \A[k+2]\dar[hook]\rar{f}\ar[phantom]{dr}[description]{=}&\O\dar[two heads]{q}\\
    \A\rar[swap]{h}&\O/\I.
  \end{tikzcd}
\]
The following statements concerning the (truncated) spectral sequence in
  \Cref{fibre-wise_obstruction_theory} hold (they replace the corresponding
  statements therein):
	\begin{enumerate}\setcounter{enumi}{11}
  \item\label{lower-diagonal_second_page_bimodule_sparse} If $k\geq d+1$, then
    \Cref{lower-diagonal_fibre-wise,second_page_fibre-wise} define a pointed
    bijection between the set of homotopy classes of maps $g\colon\A[d+2]\to\O$
    that extend to $\A[d+3]$, such that the restrictions of $f$ and $g$ to
    $\A[d+1]$ are fibre-wise homotopic, and such that $qg\colon\A[d+2]\to\O/\I$ equals the
    restriction of $h\colon\A\to\O/\I$ to $\A[d+2]$,
    \[
        \begin{tikzcd}
          \A[d+1]\dar[hook]\rar[hook]\ar[phantom]{dr}[description]{\Rightarrow}&\A[d+2]\dar{g}\rar[equals]\ar[phantom]{ddr}[description]{=}&\A[d+3]\ar[hook]{dd}\\
          \A[k+2]\rar[swap]{f}\dar[hook]\ar[phantom]{dr}[description]{=}&\O\dar[two heads]{q}\\
          \A\rar[swap]{h}&\O/\I&\A\lar{h}
        \end{tikzcd}\qquad
        \begin{tikzcd}
          \A[d+2]\rar[hook]\dar[swap]{g}&\A[d+3]\ar[dotted]{dl}\\
          \O
        \end{tikzcd}
      \]
    % the intersection of
    %   \[\ker\left[\pi_0\Str{\A[d+3]}{h}{\I}\to\pi_0\Str{\A[d+2]}{h}{\I}\right]\]
    %   with
    %   \[\im\left[\pi_0\Str{\A[d+3]}{h}{\I}\to\pi_0\Str{\A[d+3]}{h}{\I}\right],\]
    and the pointed set $\IH[d+1][-d]{\os\I}$. The base
    point in the source is the restriction of $f$ to $\A[d+2]$ and the base
    point in the target is $0$. The bijection maps the homotopy class of
    $g\colon\A[d+2]\to\O$ to
    $\Hclass{m_{d+2}^g-m_{d+2}^f}$.
  \item\label{second_differential_bimodule_sparse}
    If $k\geq 2d+1$, 
		the
    differential \[\ssd{d+1}\colon\BK[d+1][s][t]\longrightarrow\BK[d+1][s+d+1][t+d]\]
    of the $(d+1)$-st page is given by the Lie bracket with the
    universal Massey product of length $d+2$,
    \begin{align*}
      \ssd{d+1}\colon\IH[s+1][-t]{\os\I}[\os\O]&\longrightarrow\IH[s+d+2][-t-d]{\os\I}[\os\O]& s\geq 1,\\
      x&\longmapsto[\Hclass{m^f_{d+2}},x],
    \end{align*}
    possibly composed with the projection of the cocycles onto the cohomology,
    \begin{align*}
      \ssd{d+1}\colon\IZ[1][-t]{\os\I}[\os\O]&\longrightarrow\IH[d+2][-t-d]{\os\I}[\os\O]&t>d,\\
      x&\longmapsto[\Hclass{m^f_{d+2}},x].
    \end{align*}
    \setcounter{enumi}{14}
  \item\label{primary_obstructions_fibre-wise_sparse}
    Let $i\geq1$ and $k= di$. The
    obstruction in the term
    \[
      \BK[d+1][di+1][di]=\IH[di+2][-di]{\os\I}[\os\O]
    \]
    to the existence of a
    map $g\colon\A[di+3]\to\O$ such that $qg\colon\A[di+3]\to\O/\I$ is the
    restriction of ${h\colon\A\to\O/\I}$ to $\A[di+3]$ and the restrictions of
    $f\colon\A[di+2]\to\O$ and $g\colon\A[di+3]\to\O$ to $\A[di+1]$ coincide,
    \[
      \begin{tikzcd}
        \A[di+1]\rar[hook]\dar[hook]\ar[phantom]{dr}[description]{=}&\A[di+3]\dar[dotted]{g}\rar[equals]\ar[phantom]{ddr}[description]{=}&\A[di+3]\ar[hook]{dd}\\
        \A[di+2]\rar{f}\dar[hook]\ar[phantom]{dr}[description]{=}&\O\dar[two heads]{q}\\
        \A\rar[swap]{h}&\O/\I&\A\lar{h}
      \end{tikzcd}
    \]
    is represented by the
    cocycle \[\sum_{p+q=di+4}m^f_p\{m^f_q\}\in\IC[di+2][-di]{\os\I}.\]
  \item\label{first_obstruction_bimodule_sparse} For $k= d$ and $\ell=0$, the obstruction in the term
    \[
      \BK[d+1][d+1][d]=\IH[d+2][-d]{\os\I}
    \]
    to the existence of a map
    ${g\colon\A[d+3]\to\O}$ such that $qg\colon\A[d+3]\to\O/\I$ is the
    restriction of $h\colon\A\to\O/\I$ to $\A[d+3]$ and the restrictions
    of $f$ and $g$ to $\A[d+1]$ coincide,
    \[
      \begin{tikzcd}
        \A[d+1]\rar[hook]\dar[hook]\ar[phantom]{dr}[description]{=}&\A[d+3]\dar[dotted]{g}\rar[equals]\ar[phantom]{ddr}[description]{=}&\A[d+3]\ar[hook]{dd}\\
        \A[d+2]\rar{f}\dar[hook]\ar[phantom]{dr}[description]{=}&\O\dar[two heads]{q}\\
        \A\rar[swap]{h}&\O/\I&\A\lar{h}
      \end{tikzcd}
    \]
    is the cohomology class
    \[\delta(\Hclass{m_{d+2}^f})\in\IH[d+2][-d]{\os\I}.\]
    Here, $\delta$ is the connecting morphism in the long exact sequence
    \eqref{operad_multiplication_long_exact_sequence}.
  \end{enumerate}
\end{theorem}
\begin{proof}
The same arguments as in the proof of the last three items of
\Cref{fibre-wise_obstruction_theory} work for
\Cref{sparse_spectral_sequence_differential_bimodule}. The obstruction in \Cref{sparse_spectral_sequence_differential_bimodule}\eqref{first_obstruction_bimodule_sparse} fits with the long exact sequence \eqref{long_exact_sequence_Hochschild}.
\end{proof}

\begin{remark}
  Let $A$ be a $d$-sparse graded algebra and $M$ a $d$-sparse graded
  $A$-bimodule. When $\O=\E{A,M}$ and $\I=\E*{A,M}$, the obstruction in
  \Cref{primary_obstructions_fibre-wise_sparse} is represented by the bimodule cocycle
            \begin{align*}
            m^A_{di+2}\cdot\id*[M]&-\id*[M]\cdot m^A_{di+2}\\&+\sum_{\substack{p+q=di+4\\p,q>d+1}}[m^M_p,m^A_q]+m^M_p\circ m^M_q\in\BC[di+2][-di]{A^e}{M},
          \end{align*}
compare \Cref{bimodule_obstruction_theory}.
\end{remark}

\subsection{Almost formality theorems in the sparse case}

We conclude this article by stating analogues of
\Cref{secondary_Kadeishvili_operads_A-infinity_uniqueness} and
\Cref{secondary_Kadeishvili_operads_A-infinity_existence} in the $d$-sparse
case. We omit the proofs since they are almost identical to those of the
aforementioned results, only noticing that one needs to use
\Cref{sparse_spectral_sequence_differential_operad} to extend them to the general
case $d\geq1$.

\begin{definition}\label{def:Massey_operad-sparse}
  Let $d\geq1$. We make the following definitions.
  \begin{enumerate}
	\item A \emph{$d$-sparse Massey (graded) operad} is a $d$-sparse graded operad
    $\O$ equipped with a multiplication $m_2\in\O<2>^1$
    (\Cref{operad_multiplication}) and an operad cohomology class
    \[m_{d+2}\in\OH[d+2][-d]{\os\O},\]
    called \emph{universal Massey product of length $d+2$}, satisfying $\Sq[m_{d+2}]=0$.
  \item The \emph{Massey operad complex} of a $d$-sparse Massey operad is the bigraded
    (cochain) complex\footnote{\Cref{rem:Massey_operad} applies \emph{mutatis
        mutandis} to show that the Massey operad complex is indeed a complex.}
    \begin{align*}
      \OC[s]{\os\O,m_{d+2}}&=\OH[s]{\os\O}&s&\geq 2,\\
      \OC[s]{\os\O,m_{d+2}}&=0&s&<2,
    \end{align*}
    with bidegree $(d+1,-d)$ differential given by
    \[d\colon \OH[s][t]{\os\O}\longrightarrow\OH[s+d+1][t-d]{\os\O},\qquad d(x)=[m_{d+2},x],\]
    except if $\cchar{\kk}=2$ and $(s,t)=(d+1,-d)$, in which case the differential is given by
    \[d\colon \OH[d+1][-d]{\os\O}\longrightarrow\OH[2d+2][-2d]{\os\O},\qquad
      d(x)=x^2+[m_{d+2},x],\]
    compare with \Cref{sparse_spectral_sequence_differential_operad}\eqref{second_differential_operad_sparse}.
    The bigraded cohomology of this complex,
    \[\OH{\os\O,m_{d+2}}\coloneqq\H[\bullet,*]{\OC[s]{\os\O,m_{d+2}}},\]
    is the \emph{Massey operad cohomology} of the $d$-sparse Massey operad $\O$.
  \end{enumerate}
\end{definition}

\begin{remark}\label{rmk:A-O_canonical_Massey_operad-sparse}
Let $d\geq 1$ and $\O$ a $d$-sparse graded operad equipped with a map $f\colon \A[k+2]\to\O$ with $k\geq 2d+1$. It follows
from \Cref{sparse_spectral_sequence_differential_operad}\eqref{first_obstruction_operad_sparse} that $\O$ is a
$d$-sparse Massey operad with multiplication
\[
  m_2^f\in\O<2>^0=\os\O<2>^1
\]
and universal Massey product of length $d+2$
\[
  \Hclass{m_{d+2}^f}\in\OH[d+2][-d]{\os\O}.
\]
\end{remark}

\begin{theorem}\label{secondary_Kadeishvili_operads_A-infinity_uniqueness-sparse}
	Let $d\geq1$. Let $\O$ be a $d$-sparse graded operad and $f\colon\A\to\O$ a
  morphism of DG operads, so that $\O$ is equipped with the multiplication
  $m_2^f$ and the universal Massey product of length $d+2$ (\Cref{rmk:A-O_canonical_Massey_operad-sparse})
  \[
    \Hclass{m_{d+2}^f}\in\OH[d+2][-d]{\os\O}.
  \]
  Suppose that the corresponding Massey operad cohomology vanishes in the
  following range:
	\[\OH[n+2][-n]{\os\O,\Hclass{m_{d+2}^f}}=0,\qquad n>d.\]
	Then, every map $g\colon \A\to\O$ inducing the same multiplication
  $m_2^g=m_2^f$ on $\os\O$ and the same universal Massey product of length $d+2$,
  \[
    \Hclass{m_{d+2}^g}=\Hclass{m_{d+2}^f}\in\OH[d+2][-d]{\os\O},
  \]
  is homotopic to $f$.
\end{theorem}

\begin{theorem}\label{secondary_Kadeishvili_operads_A-infinity_existence-sparse}
	Let $d\geq1$. Let $\O$ be a $d$-sparse graded operad and $f\colon\A[2d+3]\to\O$ a morphism of DG operads
  with multiplication $m_2^f$ and universal Massey product of length $d+2$  (\Cref{rmk:A-O_canonical_Massey_operad-sparse})
  \[
    \Hclass{m_{d+2}^f}\in\OH[d+2][-d]{\os\O}.
  \]
  Suppose that the corresponding Massey operad cohomology vanishes in the
  following range:
	\[\OH[n+3][-n]{\os\O,\Hclass{m_{d+2}^f}}=0,\qquad n>d.\]
	Then, there exists a map $g\colon \A\to\O$ inducing the same multiplication
  $m_2^g=m_2^f$ on $\os\O$ and the same universal Massey product as $f$,
  \[
    \Hclass{m_{d+2}^g}=\Hclass{m_{d+2}^f}\in\OH[d+2][-d]{\os\O}.
  \]
\end{theorem}

\begin{definition}\label{def:Massey_algebra_sparse}
  Ler $d\geq1$. 
  We make the following definitions.
  \begin{enumerate}
  \item
    A \emph{$d$-sparse Massey (graded) algebra} is a pair $(A,m_{d+2})$ consisting of a $d$-sparse graded algebra $A$ and a Hochschild cohomology class
    \[m_{d+2}\in\HH[d+2][-d]{A},\]
    called \emph{universal Massey product of length $d+2$}, satisfying $\Sq[m_{d+2}]=0$.
	\item The \emph{Hochschild-Massey complex} of a $d$-sparse Massey algebra is the bigraded
    (cochain) complex
    \begin{align*}
      \HMC[s]{A}[m_{d+2}]&=\HH[s]{A}&s&\geq 2,\\
      \HMC[s]{A}[m_{d+2}]&=0&s&<2,
    \end{align*}
    with the bidegree $(d+1,-d)$ differential given by
    \[d\colon \HH[s][t]{A}\longrightarrow\HH[s+d+1][t-d]{A},\qquad d(x)=[m_{d+2},x],\]
    except if $\cchar{\kk}=2$ and $(s,t)=(d+1,-d)$, in which case the differential is given by
    \[d\colon \HH[d+1][-d]{A}\longrightarrow\HH[2d+2][-2d]{A},\qquad d(x)=x^2+[m_{d+2},x].\]
    The bigraded cohomology of this complex,
    \[\HMH{A}[m_{d+2}]\coloneqq\H[\bullet,*]{\HMC{A}[m_{d+2}]},\]
    is called \emph{Hochschild--Massey cohomology} of the $d$-sparse Massey algebra $(A,m_{d+2})$.
  \end{enumerate}
\end{definition}

\Cref{secondary_Kadeishvili_operads_A-infinity_uniqueness-sparse}
has the following immediate consequence, which generalises \Cref{thm:B-main_text}.

\begin{theorem}
  \label{thm:B-main_text_sparse}
  Let $A$ be a cohomologically $d$-sparse DG algebra for some $d\geq1$. Choose a minimal model $(\H{A},\Astr)$ of $A$ and
  suppose that the Hochschild--Massey cohomology of the $d$-sparse Massey algebra
  $(\H{A},\Hclass{\Astr<d+2>[A]})$ vanishes in the following
  range:\footnote{Notice the strict inequality.}
  \[
    \HMH[n+2]<-n>{\H{A}}[\Hclass{\Astr<d+2>[A]}],\qquad n>d.
  \]
  Then, every DG algebra $B$ such that $\H{B}\cong\H{A}$ as graded algebras and
  whose universal Massey product of length $d+2$ satisfies
  \[
    \Hclass{\Astr<d+2>[B]}=\Hclass{\Astr<d+2>[A]}\in\HH[d+2][-d]{\H{A}}
  \]
  is quasi-isomorphic to $A$.
\end{theorem}
\begin{proof}
  In view of \Cref{minimal_models_quasi_isomorphic_algebras} for $\O=\Ass$ and \Cref{quasi-iso_unit_algebras} we can
  prove the theorem by passing to minimal models, in which case the claim
  follows from \Cref{secondary_Kadeishvili_operads_A-infinity_uniqueness-sparse}
  taking $\O=\E{\H{A}}$.
\end{proof}

\begin{definition}\label{def:Massey_operad_sparse}
  Let $A$ be a $d$-sparse graded algebra. We make the following definitions.
  \begin{enumerate}
  \item A \emph{$d$-sparse Massey (graded) $A$-bimodule} is a pair $(M,m_{d+2})$ consisting of
    an $A$-bimodule $M$ and a bimodule Hochschild cohomology class
    \[m_{d+2}\in\HHE[d+2][-d]{A}{M},\]
    called \emph{bimodule universal Massey product of length $d+2$}, satisfying $\Sq[m_{d+2}]=0$.
	\item The \emph{bimodule Hochschild complex} of a $d$-sparse Massey $A$-bimodule is the
    bigraded (cochain) complex
  \begin{align*}
    \AlgBimHMC*!A![s]{M}[m_{d+2}] &= \HHE[s]{A}{M}, && s \geq 2, \\
    \AlgBimHMC*!A![s]{M}[m_{d+2}] &= 0, && s < 2.
  \end{align*}
	with the bidegree $(d+1,-d)$ differential given by
	\begin{align*}d\colon \HHE[s][t]{A}{M}&\longrightarrow\HHE[s+2][t-1]{A}{M}\\ x&\longmapsto[m_{d+2},x],\end{align*}
	except if $\cchar{\kk}=2$ and $(s,t)=(d+1,-d)$, in which case the differential is given by
	\begin{align*}d\colon \HHE[d+1][-d]{A}{M}&\longrightarrow\HHE[2d+2][-2d]{A}{M}\\x&\longmapsto x^2+[m_{d+2},x].\end{align*}
	The bigraded cohomology of this complex,
  \[\AlgBimHMH*!A!{M}[m_{d+2}]\coloneqq\H[\bullet,*]{\AlgBimHMC*!A!{M}[m_{d+2}]},\]
  is called \emph{bimodule Hochschild--Massey cohomology} of the $d$-sparse Massey $A$-bimodule.
  \end{enumerate}
\end{definition}

The following result us another straightforward consequence of \Cref{secondary_Kadeishvili_operads_A-infinity_uniqueness-sparse}. It extends \Cref{secondary_Kadeishvili_simultaneous_DG_uniqueness-main_text}

\begin{theorem}\label{secondary_Kadeishvili_simultaneous_DG_uniqueness-sparse}
  Let $d\geq1$. Let $A$ be a cohomologically $d$-sparse DG algebra and $M$ a cohomologically $d$-sparse DG $A$-bimodule. Choose a minimal model
  $(\H{A},\Astr[A])$ of $A$ and a compatible minimal model $(\H{M},\Astr[M])$ of
  $M$. Suppose that the bimodule Hochschild--Massey cohomology of the $d$-sparse Massey $\H{A}$-bimodule $(\H{M},\Hclass{\Astr<d+2>[A\ltimes M]})$, where $\Astr<d+2>[A\ltimes M]\coloneqq\Astr<d+2>[A]+\Astr<d+2>[M]$, vanishes in the following range:
  \[
    \AlgBimHMH!\H{A}![n+2]<-n>{\H{M}}[d+2][A\ltimes M]=0,\qquad n>d.
  \]
  Then, every pair $(B,N)$ consisting of
  \begin{itemize}
  \item a DG algebra $B$ such that $\H{B}\cong\H{A}$ as graded algebras and
  \item a DG $B$-bimodule $N$ such that $\H{N}\cong\H{M}$ as graded
    $\H{A}$-bimodules and
  \item whose bimodule universal
    Massey product of length $d+2$ satisfies
    \[
      \Hclass{\Astr<d+2>[B\ltimes N]}=\Hclass{\Astr<d+2>[A\ltimes
        M]}\in\RelBimHH[d+2]<-d>{\H{A}}{\H{M}}
    \]
  \end{itemize}
  is quasi-isomorphic to the pair $(A,M)$.
\end{theorem}
\begin{proof}
  In view of \Cref{minimal_models_quasi_isomorphic_algebras-bimodules} for $\O=\Ass$ and \Cref{quasi-iso_unit_algebras_bimodules} we can
  prove the theorem by passing to minimal models, in which case the claim
  follows from \Cref{secondary_Kadeishvili_operads_A-infinity_uniqueness-sparse}
  taking $\O=\E{\H{A},\H{M}}$.
\end{proof}

We also state an analogue of
\Cref{secondary_Kadeishvili_ideals_A-infinity_uniqueness} in the $d$-sparse
case. The proof, also omitted in this case, is almost identical to that of the
aforementioned result, using
\Cref{sparse_spectral_sequence_differential_bimodule} to extend it to the general
case $d\geq1$.

\begin{definition}
  Let $\O$ be a $d$-sparse Massey graded operad with universal Massey product
  $m_{d+2}\in\OH[d+2][-d]{\os\O}$ and let $\I\subset\O$ be an associative operadic ideal. The \emph{$d$-sparse Massey operadic ideal complex} is the bigraded (cochain) complex
  \begin{align*}
    \IC[s]{\os\I,m_{d+2}}[\os\O]&=\IH[s]{\os\I}[\os\O]&s&\geq 2,\\
    \IC[s]{\os\I,m_{d+2}}[\os\O]&
    =0&s&<2,
  \end{align*}
  with the bidegree $(d+1,-d)$ differential given by
  \[d\colon \IH[s][t]{\os\I}[\os\O]\longrightarrow\IH[s+d+1][t-d]{\os\I}[\os\O],\qquad d(x)=[m_{d+2},x],\]
  compare with \Cref{sparse_spectral_sequence_differential_bimodule}\eqref{second_differential_bimodule_sparse}.
  The bigraded cohomology of this complex,
  \[\IH{\os\I,m_{d+2}}[\os\O]\coloneqq\H[\bullet,*]{\IC{\os\I,m_{d+2}}[\os\O]},\]
  is called \emph{Massey operadic ideal cohomology} of the associative operadic ideal $\I$ of the $d$-sparse Massey graded operad $\O$.
\end{definition}

\begin{theorem}\label{secondary_Kadeishvili_ideals_A-infinity_uniqueness_sparse}
	Let $\O$ be a $d$-sparse graded operad and $f\colon\A\to\O$ a morphism of DG operads, so
  that $\O$ is equipped with the multiplication $m_2^f$ and the universal Massey
  product of length $d+2$, $\Hclass{m_{d+2}^f}$ (\Cref{rmk:A-O_canonical_Massey_operad-sparse}). 
  Let $\I\subset\O$ be an associative operadic ideal and $q\colon\O\twoheadrightarrow\O/\I$ the canonical projection.
  Suppose that the corresponding <Massey operadic ideal
  cohomology vanishes in the following range:
	\[\IH[n+1][-n]{\os\I,\Hclass{m_{d+2}^f}}[\os\O]=0,\qquad n>d.\]
	Then, every map $g\colon \A\to\O$ inducing the same multiplication
  $m_2^g=m_2^f$ on $\os\O$ and such that 
  \[
    \Hclass{m_{d+2}^g-m_{d+2}^f}\in\IH[2][-1]{\os\I}[\os\O],
  \]
  and $qf=qg$ 
  is fibre-wise homotopic to $f$, i.e.~$f$ and $g$ are homotopic through a homotopy $\A\to\O$ which composes to the trivial homotopy
  $\A\to\O\twoheadrightarrow\O/\I$.
\end{theorem}

\begin{proof}
  The proof is identical to that of
  \Cref{secondary_Kadeishvili_operads_A-infinity_uniqueness}, replacing the references to \Cref{operad_obstruction_theory} with \Cref{fibre-wise_obstruction_theory}.
\end{proof}

\begin{definition}
  Let $d\geq1$. Let $A$ be a $d$-sparse graded algebra and $(M,m_{d+2})$ a $d$-sparse
  Massey $A$-bimodule. The
  \emph{Massey bimodule complex} is the bigraded (cochain) complex
	\begin{align*}
    \BimHMC*!A![s]{M}[m_{d+2}]&=\Ext[s]{A^e}{M}{M}&s&\geq 1,\\
    \BimHMC*!A![s]{M}[m_{d+2}]&=0&s&<1,
  \end{align*}
	with the bidegree $(d+1,-d)$ differential given by
	\[d\colon \Ext[s][t]{A^e}{M}{M}\longrightarrow\Ext[s+d+1][t-d]{A^e}{M}{M},\qquad d(x)=[m_{d+2},x].\]
	The bigraded cohomology of this complex,
  \[
    \BimHMH*!A!{M}[m_{d+2}][]\coloneqq\H[\bullet,*]{\BimHMC*!A!{M}[m_{d+2}]},
  \]
  is the \emph{Massey bimodule cohomology} of the $d$-sparse
  Massey $A$-bimodule $(M,m_{d+2})$.
\end{definition}

We also record the following variant of \Cref{thm:B_bimodules-main_text}.

\begin{theorem}\label{thm:B_bimodules-sparse}
  Let $d\geq1$. Let $A$ be a cohomologically $d$-sparse DG algebra and $M$ a
  cohomologically $d$-sparse DG $A$-bimodule. Choose a minimal model
  $(\H{A},\Astr)$ of $A$, a compatible minimal model $(\H{M},\Astr[M])$
  of $M$, and suppose that the Massey bimodule
  cohomology of the $d$-sparse Massey $\H{A}$-bimodule $(\H{M},\Hclass{\Astr<d+2>[A\ltimes M]})$ vanishes in the
  following range:
  \[
    \BimHMH!\H{A}![d+1]<-n>{\H{M}}[d+2][A\ltimes M]=0,\qquad n>d.
  \]
  Then, every DG $A$-bimodule $N$ such that $\H{N}\cong\H{M}$ as graded
  $\H{A}$-bimodules and such that
  \[
    \Hclass{\Astr<d+2>[M]-\Astr<d+2>[N]}=0\in\BimHH!\H{A}^e![d+1]<-d>{\H{M}}
  \]
  is quasi-isomorphic to $M$.
\end{theorem}
\begin{proof}
  In view of \Cref{minimal_models_quasi_isomorphic_bimodules} for $\O=\Ass$ and \Cref{quasi-iso_unit_bimodules} we can
  prove the theorem by passing to minimal models, in which case the claim
  follows from \Cref{secondary_Kadeishvili_ideals_A-infinity_uniqueness_sparse}
  taking $\O=\E{\H{A},\H{M}}$ and $\I=\E*{\H{A},\H{M}}$.
\end{proof}

%%% Local Variables:
%%% mode: LaTeX
%%% TeX-master: "main"
%%% End: